\documentclass[11pt]{asens}

\usepackage{amssymb,latexsym}
\usepackage{upgreek}
\usepackage{amsthm}
\usepackage[english,francais]{babel}
\usepackage[T1]{fontenc}

\mathchardef\varGamma="0100 
\mathchardef\varTheta="0102 
\mathchardef\varLambda="0103 
\mathchardef\varXi="0104 
\mathchardef\varPi="0105 
\mathchardef\varSigma="0106 
\mathchardef\varPhi="0108 
\mathchardef\varPsi="0109 
\mathchardef\varOmega="010A 

\font\medit=ptmri at 11pt
\font\smallmedit=ptmri at 9pt
 3
 1
\def\sk{\hskip-2pt\mathop{\phantom{\vrule width1pt height2.7pt depth0pt}}\limits^*\hskip-2.7pt}

\def\w{^{\phantom i}}

\def\q{p}
\def\done{}

\def\bl{B}
\def\ee{E}
\def\mf{M}
\def\nr{N}

\def\tyb{T\hskip-2.3pt_y\w\hskip-.7pt\sd}
\def\ez{\mathcal{E}}
\def\pl{\mathcal{L}}
\def\on{\mathcal{N}}

\def\oy{\mathcal{Y}}
\def\tw{\mathcal{T}}
\def\zy{\mathcal{Z}}
\def\lf{\varLambda}
\def\mg{\ve}
\def\vx{H}
\def\pv{P}

\def\da{\Delta}

\def\vps{\varPsi}
\def\cro{\overline{\hskip-2pt\partial}}

\def\zx{c}
\def\du{A}
\def\ic{\varGamma}

\def\hg{\hat{g\hskip2pt}\hskip-1.3pt}
\def\dv{S}
\def\bs{\varSigma}

\def\vt{{\tau\hskip-4.55pt\iota\hskip.6pt}} 
\def\tp{\vt\hskip-1.8pt_+\w}
\def\tm{\vt\hskip-1.7pt_-\w}
\def\tpm{\vt\hskip-1.5pt_\pm\w}
\def\tmp{\vt\hskip-1pt_\mp\w}
\def\htp{\hat\vt\hskip-1.8pt_+\w}
\def\htm{\hat\vt\hskip-1.7pt_-\w}
\def\htpm{\hat\vt\hskip-1.6pt_\pm\w}

\def\gp{g}
\def\hr{\rho}

\def\hp{h\nnh^+}
\def\hm{h\nnh^-}
\def\hpm{h\nh^\pm}

\def\navp{\nabla\hskip-.5pt\vt}
\def\bbR{\mathrm{I\!R}}
\def\rto{\bbR^2}

\def\bbC{{\mathchoice {\setbox0=\hbox{$\displaystyle\mathrm{C}$}\hbox{\hbox 
to0pt{\kern0.4\wd0\vrule height0.9\ht0\hss}\box0}} 
{\setbox0=\hbox{$\textstyle\mathrm{C}$}\hbox{\hbox 
to0pt{\kern0.4\wd0\vrule height0.9\ht0\hss}\box0}} 
{\setbox0=\hbox{$\scriptstyle\mathrm{C}$}\hbox{\hbox 
to0pt{\kern0.4\wd0\vrule height0.9\ht0\hss}\box0}} 
{\setbox0=\hbox{$\scriptscriptstyle\mathrm{C}$}\hbox{\hbox 
to0pt{\kern0.4\wd0\vrule height0.9\ht0\hss}\box0}}}} 
\def\bbCP{\bbC\mathrm{P}}
\def\bbZ{\hbox{$\mathsf{Z}\hskip-4.5pt\mathsf{Z}$}}

\def\dimc{\dim_{\hskip.4pt\bbC}\w\nh}
\def\lie{\pounds}
\def\hs{\hskip.7pt}
\def\hh{\hskip.4pt}
\def\hn{\hskip-.4pt}
\def\nh{\hskip-.7pt}
\def\nnh{\hskip-1.5pt}
\def\hrz{^{\hskip.5pt\text{\rm hrz}}}

\def\la{\langle}
\def\ra{\rangle}
\def\lr{\langle\,,\rangle}
\def\sd{\varSigma}
\def\sb{B}

\def\ks{\mathsf{K}}
\def\ls{\mathsf{L}}

\def\vs{\mathsf{V}}
\def\ws{\mathsf{W}}

\def\ve{\varepsilon}

\def\hyp{\hskip.5pt\vbox
{\hbox{\vrule width2.5ptheight0.5ptdepth0pt}\vskip2pt}\hskip.5pt}

\begin{document}

\renewcommand{\theequation}{\arabic{section}.\arabic{equation}}
\title[K\"ah\-ler manifolds with geodesic hol\-o\-mor\-phic 
gradients]{K\"ah\-ler manifolds with geodesic\\ hol\-o\-mor\-phic gradients}

\author{Andrzej DERDZINSKI}

\address{Department of Mathematics, The Ohio State University, Columbus, OH 
43210, USA}

\email{andrzej@math.ohio-state.edu}


\author{Paolo PICCIONE}

\address{Instituto de Matem\'atica e Estat\'\i stica, Universidade de S\~ao 
Paulo, Rua do Mat\~ao 1010, CEP 05508-900, S\~ao Paulo, SP, Brazil}

\email{piccione@ime.usp.br}

\theoremstyle{plain} 
\newtheorem{theorem}{Theorem}[section]
\newtheorem{lemma}[theorem]{Lemma}
\newtheorem{corollary}[theorem]{Corollary}
\newtheorem{proposition}[theorem]{Proposition}
\newtheorem{claim}[theorem]{Claim}
\theoremstyle{definition} 
\newtheorem{definition}[theorem]{Definition}
\newtheorem{remark}[theorem]{Remark}
\newtheorem{example}[theorem]{Example}
\newtheorem{notation}[theorem]{Notation}

\selectlanguage{english}
\begin{abstract}
A vector field on a Riemannian manifold is called geodesic if its integral 
curves are reparametrized geodesics. We classify compact K\"ah\-ler manifolds 
admitting nontrivial real-holomorphic geodesic gradient vector fields 
that satisfy an additional in\-te\-gra\-bil\-i\-ty condition.
They 
are all biholomorphic to bundles of complex projective spaces.
\vskip 0.5\baselineskip

\selectlanguage{francais}
\noindent\hbox{\sc R\'esum\'e.\hskip2.5pt--\hskip2pt}{\bf Vari\'et\'es 
k\"ahl\'eriennes admettant des gradients g\'eod\'e\-siques holomorphes.} \ 
Un champ de vecteurs sur une vari\'et\'e rieman\-nienne est dit g\'eod\'esique 
si ses courbes int\'egrales sont g\'eod\'esiques non param\'etr\'es. On 
classifie des vari\'et\'es k\"ahl\'eriennes compactes qui admettent des 
gradients g\'eod\'esiques r\'eels \hskip-.64ptholomorphes \hskip-.64ptnon \hskip-.64pttriviaux \hskip-.64pt
satisfaisant \hskip-.64pt\`a \hskip-.64ptune \hskip-.64ptcondition \hskip-.64ptadditionnelle \hskip-.64ptd'int\'egra\-bilit\'e.         
Elles sont toutes 
biholomorphes \`a fibr\'es en espaces projectifs complexes.
\selectlanguage{english}
\end{abstract}




\maketitle 

\section*{Introduction\done}
\setcounter{equation}{0}
We say that a vector field on a Riemannian manifold is {\medit geodesic\/} if 
its integral curves are re\-pa\-ram\-e\-trized geodesics. The present paper 
discusses
\begin{equation}\label{trp}
\begin{array}{l}
\mathrm{triples\ }\hs(\mf\nh,g,\vt)\hs\mathrm{\ consisting\ of\ a\ compact\ 
complex\ manifold\ }\hs\mf\nh\mathrm{,\ a}\\
\mathrm{K}\ddot{\mathrm{a}}\mathrm{hler\ metric\ }\,g\,\mathrm{\ on\ 
}\,\mf\nnh\mathrm{,\ and\ a\ nonconstant\ function\ 
}\,\vt:\mf\nh\to\bbR\mathrm{,}\\
\mathrm{the\hs\hs\ }\,g\hyp\mathrm{gra\-di\-ent\hs\hs\ of\hs\hs\ which\hs\hs\ 
is\hs\hs\ both\hs\hs\ geodesic\hs\hs\ and\hs\hs\ 
real}\hyp\mathrm{hol\-o\-mor\-phic.}
\end{array}
\end{equation}
We observe (Remark~\ref{dppdm}) that for $\,m=\dimc\nh\mf\hs$ and 
$\,d_\pm\w\nh=\dimc\nh\sd^\pm\nh$, where $\,\sd^+$ and $\,\sd^-$ are the 
maximum and minimum level sets of $\,\vt$, one then has
\begin{equation}\label{ddm}
d_+\w+\hs\,d_-\w\hs\,\ge\,\,m\,-\,1\,\,\ge\hs\,\,d_\pm\w\hs\,\ge\,\,0\hh,
\end{equation}
and every $\,(d_+\w,d_-\w,m)\in\bbZ^3$ satisfying (\ref{ddm}) is realized by 
some $\,(\mf\nh,g,\vt)\,$ with (\ref{trp}).

One of our three main results, Theorem~\ref{cpbdl}, classifies the triples 
(\ref{trp}) such that
\begin{equation}\label{spn}
\mathrm{Ker}\hskip2.3ptd\pi^+\mathrm{\ and\ 
}\,\hs\mathrm{Ker}\hskip2.3ptd\pi^-\mathrm{\ span\ an\ in\-te\-gra\-ble\ 
distribution\ on\ }\,\mf'\nh.
\end{equation}
Here $\,\mf'\nh=\mf\smallsetminus(\sd^+\nnh\cup\sd^-)$, while 
$\,\pi^\pm\nnh:\mf\nnh\smallsetminus\nh\sd^\mp\nh\to\sd^\pm$ sends each 
$\,x\in\mf\nnh\smallsetminus\nh\sd^\mp$ to the unique point nearest $\,x\,$ in 
$\,\sd^\pm\nh$. (In case (\ref{trp}) $\,\pi^\pm$ always are disk-bundle 
projections, and their vertical distributions 
$\,\mathrm{Ker}\hskip2.3ptd\pi^\pm$ span a vector sub\-bun\-dle of 
$\,T\nnh\mf'\nh$, cf.\ Section~\ref{cc}; however, (\ref{trp}) does not imply 
(\ref{spn}) -- see Remark~\ref{nintg}.)

As a consequence of Theorem~\ref{cpbdl}, in every triple with (\ref{trp}) and  
(\ref{spn}),
\begin{equation}\label{bdl}
\begin{array}{l}
\mf\,\mathrm{\nh\ is\nh\ bi\-hol\-o\-mor\-phic\nh\ to\nh\ a\nh\ bundle\nh\ 
of\nh\ (pos\-i\-tive\hs}\hyp\hs\mathrm{di\-men\-sion\-al)\nh\ complex}\\
\mathrm{projective\hs\ spaces\hs\ over\hs\ some\hs\ base\hs\ manifold\ 
}\,\,\,\sb\,\,\mathrm{\ with\ }\,\,\dimc\nnh\sb\ge0.
\end{array}
\end{equation}
The remaining two main results of the paper, Theorems~\ref{tgimm} 
and~\ref{dicho}, deal with the general case of (\ref{trp}), that is, do not 
assume (\ref{spn}).

According to Theorem~\ref{tgimm}, whenever $\,\varPi^\pm$ is a leaf of either 
(obviously in\-te\-gra\-ble) vertical distribution 
$\,\mathrm{Ker}\hskip2.3ptd\pi^{\hs\pm}\nh$, the {\medit other\/} projection 
$\,\pi^\mp$ maps $\,\varPi^\pm\nh\cap\mf'$ onto the image $\,F(\bbCP^k)\,$ of 
some totally geodesic hol\-o\-mor\-phic immersion $\,F:\bbCP^k\nh\to\sd^\mp$ 
inducing on $\,\bbCP^k$ a multiple of the Fu\-bi\-ni-Stu\-dy metric, with 
$\,k=k_\pm\w\nh\ge0\,$ given by $\,k_\pm\w\nh=m-1-\hs d_\pm\w$. 
Both $\,\sd^\pm$ are themselves (connected) totally geodesic 
compact complex sub\-man\-i\-folds of $\,\mf\nh$, cf.\ Remark~\ref{ascdt}(iii).

The third main result reveals a dichotomy involving the assignment
\begin{equation}\label{asg}
\mf'\nh\ni x\,
\mapsto\,d\pi\nh_x^{\hs\pm}\nh(\mathrm{Ker}\hskip2.3ptd\pi\nh_x^\mp\hn)
\in\hs\mathrm{Gr}\nh_k\w\nh(\tyb^\pm\nh)\hskip7pt\mathrm{for}\hskip7pt
y=\pi^\pm\nh(x)\hh,
\end{equation}
$\mathrm{Gr}\nh_k\w\nh(\tyb^\pm\nh)\,$ being the complex Grass\-mann\-i\-an, 
with $\,k=k_\pm\w$ defined as before.

Specifically, Theorem~\ref{dicho} states that one of the following two cases 
has to occur. First, (\ref{asg}) may be {\medit constant\/} on every leaf of 
$\,\mathrm{Ker}\hskip2.3ptd\pi\nh_x^{\hs\pm}$ in $\,\mf'\nh$, that is, on 
every fibre $\,\varPi^\pm$ of the projection 
$\,\pi^\pm\nnh:\mf\nnh\smallsetminus\nh\sd^\mp\nh\to\sd^\pm$ restricted to 
$\,\mf'\nh$, with either sign $\,\pm\hs$. Otherwise, $\,l=k_\mp\w$ and 
$\,k=k_\pm\w$ are positive for both signs $\,\pm\hs$, while (\ref{asg})
restricted to any such leaf $\,\varPi^\pm$ must be a composite mapping 
$\,\varPi^\pm\nnh\to\hh\bbCP\hh^l\nnh
\to\hh\mathrm{Gr}\nh_k\w\nh(\tyb^\pm\nh)\,$ 
formed by a hol\-o\-mor\-phic bundle projection 
$\,\varPi^\pm\nnh\to\hh\bbCP\hh^l\nh$, having the fibre 
$\,\bbC\smallsetminus\{0\}$, and a {\medit nonconstant hol\-o\-mor\-phic 
embedding\/} $\,\bbCP\hh^l\nnh\to\hh\mathrm{Gr}\nh_k\w\nh(\tyb^\pm\nh)$.

The first case of Theorem~\ref{dicho} is equivalent to condition (\ref{spn}), 
and the immersions $\,\bbCP^k\nh\to\sd^\pm\nh$, mentioned in the above summary 
of Theorem~\ref{tgimm}, are then embeddings, for both signs $\,\pm\hs$, while 
their images constitute foliations of $\,\sd^\pm\nh$, both with the same leaf 
space $\,\sb\,$ appearing in (\ref{bdl}). See Remark~\ref{infty}.

In the second case (cf.\ Remark~\ref{nncst}) the images of these immersions, 
rather than being pairwise disjoint, are totally geodesic, 
hol\-o\-mor\-phic\-al\-ly immersed complex projective spaces, an uncountable 
family of which passes through each point of $\,\sd^\pm\nh$.

Three special classes of the objects (\ref{trp}) have been studied before. One 
is provided by the gradient K\"ah\-\hbox{ler\hs-}\hskip0ptRic\-ci sol\-i\-tons 
discovered by Koiso \cite{koiso} and, independently, Cao \cite{cao}, where 
$\,\vt\,$ is the sol\-i\-ton function; two more -- by special 
K\"ah\-\hbox{ler\hs-}\hskip0ptRic\-ci potentials $\,\vt\,$ on compact 
K\"ah\-ler manifolds \cite{derdzinski-maschler-06}, and by triples with 
(\ref{trp}) such that $\,\mf\,$ is a (compact) complex surface 
\cite{derdzinski-kp}. Each of these three classes satisfies (\ref{spn}).

The papers \cite{derdzinski-maschler-06,derdzinski-kp} provide complete 
explicit descriptions of the classes discussed in them. Our 
Theorem~\ref{cpbdl} generalizes their classification results, namely, 
\cite[Theorem 16.3]{derdzinski-maschler-06} and 
\cite[Theorem 6.1]{derdzinski-kp}.

For more details on the preceding two paragraphs, see Remark~\ref{clone}.

Functions with geodesic gradients on arbitrary Riemannian manifolds, usually 
called {\it trans\-nor\-mal}, have been studied extensively as well 
\cite{wang,miyaoka,bolton}.

Both authors' research was supported in part by a FAPESP\nh--\hs OSU 2015 
Regular Research Award (FAPESP grant: 2015/50265-6). The authors wish to thank 
Fangyang Zheng for helpful comments.

\section{Preliminaries\done}\label{pr}
\setcounter{equation}{0}
Manifolds, mappings and tensor fields, including Riemannian metrics and 
functions, are by definition of class $\,C^\infty\nnh$. A (sub)manifold 
is always assumed connected.

Our sign convention about the curvature tensor $\,R=\nnh R\hn^\nabla$ of a 
connection $\,\nabla\,$ in a vector bundle $\,\ee\,$ over a manifold $\,\mf\,$ 
is
\begin{equation}\label{cur}
R\hh(\hn v,w)\hh\xi\,=\,\nabla\hskip-3pt_w\w\nh\nabla\hskip-3pt_v\w\hs\xi\,
-\,\nabla\hskip-3pt_v\w\nh\nabla\hskip-3pt_w\w\hs\xi\,
+\,\nabla\hskip-3pt_{[v,w]}\w\hh\xi
\end{equation}
for any section $\,\xi\,$ of $\,\ee\,$ and vector fields $\,v,w\,$ tangent to 
$\,\mf\nh$. One may treat $\,R\hh(\hn v,w)$, the covariant derivative 
$\,\nabla\xi$, and any function $\,f\,$ on $\,\mf\,$ as bundle morphisms
\begin{equation}\label{nwt}
R\hh(\hn v,w),\hs f:\ee\to\ee\hh,\hskip22pt
\nabla\xi:T\nnh\mf\to\ee
\end{equation}
sending $\,\xi\,$ or $\,v\,$ as above to $\,R\hh(\hn u,v)\hh\xi$, $\,f\xi\,$ 
or, respectively, $\,\nabla\hskip-3pt_v\w\xi$. Notation of (\ref{nwt}) is used 
in the next three displayed relations. 

In the case of a Riemannian manifold $\,(\mf\nh,g)$, the symbol $\,\nabla\,$ 
will always stand for the Le\-vi-Ci\-vi\-ta connection of $\,g\,$ as well as 
the $\,g$-gra\-di\-ent. Given a function $\,\vt\,$ and vector fields 
$\,w,w\hh'$ on $\,(\mf\nh,g)$, one has the Lie-de\-riv\-a\-tive relation
\begin{equation}\label{lvg}
[\nh\lie\hskip-.5pt_v\w g](\hn w,w\hh'\hh)=2\hh g(\dv\nh w,w\hh'\hh)\hh,\hskip8pt\mathrm{where}
\hskip5ptv=\navp\hskip5pt\mathrm{and}\hskip5pt\dv
=\nabla\nh v:T\nnh\mf\to T\nnh\mf\hh,
\end{equation}
due to the local-coordinate equalities 
$\,[\nh\lie\hskip-.5pt_v\w g]_{jk}\w=v_{j,k}\w+v_{k,j}\w=2v_{j,k}\w$. For 
vector fields $\,v,u\,$ on a manifold $\,\mf\,$ and a bundle morphism 
$\,\bl:T\nnh\mf\to T\nnh\mf\nh$, the Leib\-niz rule gives 
$\,[\nh\lie\hskip-.5pt_v\w\bl]u=[v,\bl u]-\bl[v,u]
=[\nabla\hskip-3pt_u\w\bl]u+\bl\nabla\hskip-3pt_u\w v
-\nabla\hskip-3pt_{\bl u}\w v$, and so
\begin{equation}\label{lie}
\lie\hskip-.5pt_v\w\bl\,=\,\nabla\hskip-3pt_u\w\bl\,+\,[\bl,\nabla\nh v]\hh.
\end{equation}
Next, let $\,u\,$ be a Kil\-ling vector field on a Riemannian manifold 
$\,(\mf\nh,g)$. The Ric\-ci and Bian\-chi identities imply, as in 
\cite[bottom of p.\ 572]{derdzinski-roter-07}, the well-known relation
\begin{equation}\label{scd}
\nabla\hskip-3pt_v\w\du\,=\,R(\hn u,v)\hh,\hskip12pt\mathrm{with\ }\,\,\du
=\nabla\nh u\hh.
\end{equation}
Since the flow of a Kil\-ling field preserves the Le\-vi-Ci\-vi\-ta 
connection, (\ref{scd}) also follows from the classical Lie-de\-riv\-a\-tive 
equality$\,[\nh\lie\hskip-.5pt_u\w\nnh\nabla]\hn _v\w w
=[\nabla\hskip-3pt_v\w\du]\hs w-R\hh(\hn u,v)\hh w$, 
with $\,\du=\nabla\nh u$, cf.\ \cite[formula (1.8) on p. 337]{schouten}, valid 
for any connection $\,\nabla\,$ 
in $\,T\nnh\mf\nh$.

Whenever $\,\vt:\mf\to\bbR\,$ is a function on a Riemannian manifold 
$\,(\mf\nh,g)$, we have
\begin{equation}\label{tnd}
\nabla Q\,=\,2\hh\nabla\hskip-3pt_v\w v\hh,\hskip12pt\mathrm{where}\hskip7ptv
=\navp\hskip7pt\mathrm{and}\hskip7ptQ=g(\hn v,v)\hh,
\end{equation}
as one sees noting that, in local coordinates,
$\,(\vt\nnh_{,\hs k}\w\vt^{\hh,\hh k})\nh_{,\hh j}\w
=2\vt\nnh_{,\hs kj}\w\vt^{\hh,\hh k}$. Also, obviously
\begin{equation}\label{dvt}
d_v\w\vt\,=\,g(\hn v,\navp)\,=\,Q\hskip12pt\mathrm{if}\hskip7ptv=\navp\hskip7pt
\mathrm{and}\hskip7ptQ=g(\hn v,v)\hh.
\end{equation}
\begin{remark}\label{dvwwp}Relation (\ref{lvg}) becomes 
$\,d_v\w[\hs g(\hn w,w\hh'\hh)]=2\hh g(\dv\nh w,w\hh'\hh)\,$ if, in addition, 
$\,v$ commutes with $\,w\,$ and $\,w\hh'\nh$. Namely, 
$\,d_v\w=\nh\lie\hskip-.5pt_v\w$ on functions, so that we may evaluate 
$\,d_v\w[\hs g(\hn w,w\hh'\hh)]\,$ using the Leib\-niz rule for the Lie 
derivative with $\,\lie\hskip-.5pt_v\w w=\lie\hskip-.5pt_v\w w\hh'\nh=0$.
\end{remark}
\begin{remark}\label{gpgrd}Whenever the $\,g$-gra\-di\-ent 
$\,v=\navp\,$ of a function $\,\vt\,$ on a Riemannian manifold 
$\,(\mf\nh,g)\,$ is tangent to a sub\-man\-i\-fold $\,\varPi\,$ with the 
sub\-man\-i\-fold metric $\,g'\nnh$, the restriction of $\,v\,$ to 
$\,\varPi\,$ obviously equals the $\,g'\nnh$-gra\-di\-ent of 
$\,\vt:\varPi\to\bbR$.
\end{remark}
\begin{remark}\label{cifot}Given a manifold $\,\mf\,$ and 
$\,\sigma,\vt:\mf\to\bbR$, we call $\,\sigma\,$ {\it a\/ $\,C^\infty$ function 
of\/} $\,\vt\,$ if $\,\vt\,$ is nonconstant (so that its range $\,\vt(\mf)\,$ 
is an interval) and $\,\sigma=\chi\circ\tau\,$ for some $\,C^\infty$ function 
$\,\chi:\vt(\mf)\to\bbR$. Note that $\,\chi\,$ is then uniquely determined by 
$\,\sigma$ and $\,\vt$. We will denote by $\,\sigma\,$ both the original 
function $\,\mf\to\bbR\,$ and the function $\,\chi:I\to\bbR\,$ of the variable 
$\,\vt\in I\nh$.
\end{remark}
Let $\,(t,s)\mapsto x(t,s)\in\mf\,$ be a fixed {\medit variation of curves\/} 
in a manifold $\,\mf\nh$, that is, a $\,C^\infty$ mapping in which the real 
variables $\,t,s\,$ range independently over intervals. The {\medit partial 
derivative\/} $\,x_t\w$ (or, $\,x_s\w$) then assigns to each 
$\,(t_0\w,s_0\w)\,$ the velocity vector at $\,t_0\w$ (or, $\,s_0\w$) of the 
curve $\,t\mapsto x(t,s_0\w)\,$ or, respectively, $\,s\mapsto x(t_0\w,s)$. 
(Thus, $\,x_t\w$ and $\,x_s\w$ are sections of a specific pull\-back 
bundle.) A connection $\,\nabla\,$ on $\,\mf\,$ allows us to define 
the {\medit mixed sec\-ond-or\-der partial derivatives\/} $\,x_{ts}\w$ and 
$\,x_{st}\w$ of the variation, so that, for instance, the value of 
$\,x_{ts}\w$ at $\,(t_0\w,s_0\w)\,$ is the $\,\nabla\nh$-co\-var\-i\-ant 
derivative, at the parameter $\,s_0\w$, of the vector field 
$\,s\mapsto x_t\w(t_0\w,s)\,$ along the curve $\,s\mapsto x(t_0\w,s)$, and 
analogously for $\,x_{st}\w$. Obviously, $\,x\nh_{st}\w\nh=x_{ts}\w$ when 
$\,\nabla\,$ is tor\-sion-free, cf.\ \cite[p.\ 101]{derdzinski-maschler-06}.
\begin{remark}\label{jcobi}For a tor\-sion-free connection $\,\nabla\,$ on a 
manifold $\,\mf\,$ and a smooth variation 
$\,(t,s)\mapsto x(t,s)=\exp_{y(s)}\w\hs(t-t_0\w)\xi(s)\,$ of 
$\,\nabla\nh$-ge\-o\-des\-ics, with $\,(t,s)\,$ near $\,(t_0\w,s_0\w)$ in 
$\,\rto$ and a vector field 
$\,s\mapsto\xi(s)\in T\hskip-3pt_{y(s)}\w\hskip-.7pt\mf\,$ along a curve 
$\,s\mapsto y(s)\in\mf\nh$, let $\,t\mapsto\hat w(t)\,$ be the Ja\-co\-bi 
vector field along the geodesic $\,t\mapsto x(t)=x(t,s_0\w)\,$ defined by 
$\,\hat w(t)=x_s\w(t,s_0\w)\,$ (notation of the last paragraph). Then 
$\,[\nabla\hskip-3pt_{\dot x}\w\hat w](t_0\w)
=[\nabla\hskip-3pt_{\dot y}\w\hs\xi](s_0\w)$, 
which is nothing else than $\,x\nh_{st}\w\nh=x_{ts}\w$ (see above) at 
$\,(t,s)=(t_0\w,s_0\w)$. Also, clearly, $\,\hat w(t_0\w)=\dot y(s_0\w)$. Note 
that $\,\nabla\hskip-3pt_{\dot y}\w\xi\,$ may be nonzero even if $\,y(s)=y\,$ 
is a constant curve, as it then equals the ordinary derivative of 
$\,s\mapsto\xi(s)\in T\hskip-3pt_y\w\hskip-.7pt\mf\,$ with respect to $\,s$.
\end{remark}
\begin{remark}\label{tlspc}Let $\,\nr\,$ be a vector bundle over a manifold 
$\,\sd$. We use the same symbol $\,\nr\hs$ for its total space, which we 
identify, as a set, with 
$\,\{(y,\xi):y\in\sd\,\mathrm{\ and\ }\,\xi\in\nr\hskip-2.4pt_y\w\}$. Given 
a connection $\,\mathrm{D}\,$ in $\,\nr\hs$ and vector fields $\,v,w\,$ 
tangent to $\,\sd$,
\begin{equation}\label{cmt}
\mathrm{the\ }\hs\mathrm{D}\hh\hyp\hn\mathrm{hor\-i\-zon\-tal\ lifts\ of\ 
}\hs v\hs\mathrm{\ and\ }\hs w\hs\mathrm{\ commute\ if\ so\ do\ 
}\hs v,w\hs\mathrm{\ and\ }\hs R^{\mathrm{D}}\nh(\hn v,w)=0
\end{equation}
(notation of (\ref{nwt})). Namely, at any $\,x=(y,\xi)\in\nr\nnh$, the vertical 
(or, horizontal) component of the Lie bracket of the horizontal lifts of 
$\,v\,$ and $\,w\,$ equals 
$\,R_y^{\mathrm{D}}\nh(\hn v_y\w,w_y\w)\hh\xi$ (an easy exercise) or, 
respectively, the horizontal lift of $\,[v,w]_y\w$, cf.\ 
\cite[p.\ 10]{kobayashi-nomizu}.
\end{remark}
\begin{remark}\label{nexpm}Let $\,\mathrm{D}\,$ be the normal connection in 
the normal bundle $\,\nr\hskip-2.3pt\sd\,$ of a totally geodesic 
sub\-man\-i\-fold $\,\sd\,$ in a Riemannian manifold $\,(\mf\nh,g)$. We denote by 
$\,\mathrm{Exp}^\perp\nnh:U\nh\to\mf\,$ the normal exponential mapping 
of $\,\sd$, the domain of which is an open sub\-man\-i\-fold $\,\,U$ of the 
total space $\,\nr\hskip-2.3pt\sd\,$ such that, for every normal space 
$\,\nr\hskip-2.4pt_y\w\sd$, where $\,y\in\sd^\pm\nh$, the intersection 
$\,\,U\nh\cap\nr\hskip-2.4pt_y\w\sd\,$ is nonempty and star-shap\-ed (in the 
sense of being a union of line segments emanating from $\,0$). 
Remark~\ref{jcobi} leads to the following well-known description of the 
differential $\,d\hs\mathrm{Exp}^\perp_{(y,\,\xi)}$ of $\,\mathrm{Exp}^\perp$ 
at any $\,(y,\xi)\in\,U\nnh$, cf.\ Remark~\ref{tlspc}. Specifically, we may 
assume that $\,\xi\ne0\,$ since, clearly, 
$\,d\hs\mathrm{Exp}^\perp_{(y,\,\xi)}=\hs\mathrm{Id}$ when $\,\xi=0$, under 
the obvious isomorphic identification 
$\,T_{(y,0)}\w[\nr\hskip-2.3pt\sd]
=T\hskip-3pt_y\w\sd\oplus\nr\hskip-2.4pt_y\w\sd
=T\hskip-3pt_y\w\mf\nh$. The point $\,y\in\sd\,$ and the normal vector 
$\,\xi\in\nr\hskip-2.4pt_y\w\sd$ thus have the property that the nontrivial 
geodesic $\,r\mapsto x(\hn r\nh)=\exp_y\w\hn r\xi\,$ is defined for all 
$\,r\in[\hs0,1\hh]$. If $\,r>0$, a vector tangent to 
$\,\nr\hskip-2.3pt\sd\,$ at $\,(y,r\xi)\,$ can be uniquely written as 
$\,r\eta+w_r\hrz\nh$, where 
$\,\eta\in\nr\hskip-2.4pt_y\w\sd=T_{(y,\,r\xi)}\w[\nr\hskip-2.4pt_y\w\sd]\,$ 
is vertical and $\,w_r\hrz$ denotes the $\,\mathrm{D}\hh$-\hn hor\-i\-zon\-tal 
lift of some $\,w\in T\hskip-3pt_y\w\sd$. Then, for the Ja\-cobi field 
$\,r\mapsto\hat w(\hn r\nh)\,$ along our geodesic $\,r\mapsto x(\hn r\nh)\,$ 
such that $\,\hat w(0)=w\,$ and 
$\,[\nabla\hskip-3pt_{\dot x}\w\hat w](0)=\eta$,
\begin{equation}\label{dxp}
d\hs\mathrm{Exp}^\perp_{(y,\,r\xi)}(r\eta+w_r\hrz)\,
=\,\hat w(\hn r\nh)\hskip13pt\mathrm{whenever}\hskip6ptr\in[\hs0,1\hh]\hh.
\end{equation}
In fact, linearity of both sides in $\,(\eta,w)\,$ allows us to consider two 
separate cases, $\,w=0\,$ and $\,\eta=0$. For $\,s\,$ close to $\,0\,$ 
in $\,\bbR\,$ and $\,r\in[\hs0,1\hh]$, let us set 
$\,x(r,s)=\exp_{y(s)}\w\hs r\xi(s)$, where in the former case 
$\,(y(s),\xi(s))=(y,\xi+s\eta)$, and in the latter $\,s\mapsto\xi(s)\,$ is the 
$\,\mathrm{D}\hh$-\nh par\-al\-lel normal vector field with $\,\xi(0)=\xi\,$ 
along a fixed curve $\,s\mapsto y(s)\in\sd\,$ such that $\,y(0)=y\,$ and 
$\,\dot y(0)=w$. Thus, in both cases, the curve $\,s\mapsto (y(s),r\xi(s))\,$ 
in $\,\,U\hs$ has, at $\,s=0$, the velocity $\,r\eta+w_r\hrz$. The 
velocity at $\,s=0$ of its $\,\mathrm{Exp}^\perp\nh$-im\-age curve 
$\,s\mapsto x(r,s)\,$ therefore equals the left-hand side of (\ref{dxp}). At 
the same time this last velocity is $\,\hat w(\hn r\nh)=x_s\w(r,0)\,$ for 
$\,\hat w\,$ defined as in Remark~\ref{jcobi} with the variable $\,t\,$ and 
$\,(t_0\w,s_0\w)\,$ replaced by $\,r\,$ and $\,(0,0)$. Now (\ref{dxp}) follows 
since the two definitions of $\,\hat w\,$ agree: according to 
Remark~\ref{jcobi}, both Ja\-cobi fields denoted by $\,\hat w\,$ satisfy the 
same initial conditions at $\,s=0$.
\end{remark}
\begin{remark}\label{kiljc}Every Kil\-ling vector field $\,u\,$ on a 
Riemannian manifold is a Ja\-co\-bi field along any geodesic 
$\,t\mapsto x(t)$. In fact, the local flow of $\,u$, applied to the geodesic, 
yields a variation of geodesics. (Equivalently, one may note that (\ref{scd}) 
with $\,v=\dot x$, evaluated on $\,\dot x$, is precisely the Ja\-co\-bi 
equation.)
\end{remark}
\begin{remark}\label{ttgim}Let $\,\vps:\varPi\to\mf\,$ be a totally geodesic 
immersion of a manifold $\,\varPi\hs$ in a Riemannian manifold $\,(M,g)$. If 
$\,\vps(\nh\varLambda)\subseteq\sd\,$ and 
$\,\vps(\varPi\nh\smallsetminus\varLambda)\subseteq\mf\smallsetminus\sd\,$ for 
sub\-man\-i\-folds $\,\varLambda\,$ of $\,\varPi\hs$ and $\,\sd\,$ of 
$\,\mf\nh$, such that $\,\sd\,$ is totally geodesic in $\,(M,g)$, then, for 
$\,\sd$ endowed with the sub\-man\-i\-fold metric, 
$\,\vps:\varLambda\to\sd\,$ is a totally geodesic immersion.

In fact, every point of $\,\varLambda\,$ has a neighborhood $\,\,U\,$ in 
$\,\varPi\,$ on which $\,\vps\,$ is an embedding with a totally geodesic image 
$\,\vps(\hn u)$. Our claim now follows since the sub\-man\-i\-fold 
$\,\vps(\varLambda\cap U)\,$ of $\,\sd$, being the 
intersection of the totally geodesic sub\-man\-i\-folds $\,\vps(\hn u)$ and 
$\,\sd$, must itself be totally geodesic.
\end{remark}
\begin{remark}\label{crvcm}Let $\,R,R\hh'$ and $\,\hat R\,$ be the curvature 
tensors of connections $\,\nabla,\nabla'$ in vector bundles $\,\ee,\ee'$ over 
a fixed base manifold and, respectively, of the connection $\,\hat\nabla\,$ 
induced by them in the vector bundle $\,\mathrm{Hom}\hs(\ee,\ee')$. Then 
$\,\hat R\,$ is given by the com\-mu\-ta\-tor-type formula  
$\,\hat R(\hn v,w)\hh\varTheta\hs
=\hs[R\hh'\nh(\hn v,w)]\varTheta\hs-\hs\varTheta[R(\hn v,w)]$, cf.\ 
(\ref{nwt}), for any section $\,\varTheta\,$ of 
$\,\mathrm{Hom}\hs(\ee,\ee')\,$ (that is, any 
vec\-tor-bun\-dle morphism $\,\varTheta:\ee\to\ee'$) and vector fields 
$\,v,w\,$ tangent to the base. This trivially follows from (\ref{cur}) and the 
fact that $\,[\hat\nabla\hskip-3pt_v\w\varTheta]\hh\xi
=\nabla'\hskip-6pt_v\w(\varTheta\hh\xi)-\varTheta\nabla\hskip-3pt_v\w\xi\,$ 
whenever $\,\xi\,$ is a section of $\,\ee$.
\end{remark}

\section{Projectability of distributions\done}\label{pd}
\setcounter{equation}{0}
As usual, whenever $\,\pi:\mf\to\sb\,$ is a mapping between manifolds, we 
say that a vector field $\,w\,$ (or, a distribution $\,\ez$) on 
$\,\mf\,$ is $\,\pi${\medit-pro\-ject\-a\-ble\/} if
\begin{equation}\label{prj}
d\pi\nnh_x\w w_x\w=u_{\pi(x)}\w\hskip12pt\mathrm{or,\ 
respectively,}\hskip8ptd\pi\nnh_x\w(\ez\nnh_x\w)=\mathcal{H}_{\pi(x)}\w
\end{equation}
for some vector field $\,u\,$ (or, some distribution $\,\mathcal{H}$) on 
$\,\sb\,$ and all $\,x\in\mf\nh$.
\begin{remark}\label{liebr}Let $\,\pi:\mf\to\sb\,$ be a bundle projection. A 
vector field $\,w\,$ on $\,\mf\,$ is $\,\pi$-pro\-ject\-a\-ble if and only if, 
for every section $\,v\,$ of the vertical distribution 
$\,\mathcal{V}=\,\mathrm{Ker}\hskip2.3ptd\pi$, the Lie bracket $\,[v,w]\,$ is 
also a section of $\,\mathcal{V}\nh$. This is easily verified in local 
coordinates for $\,\mf\,$ that make $\,\pi\,$ appear as a 
Car\-te\-sian-prod\-uct projection.
\end{remark}
\begin{remark}\label{cmmut}For $\,\pi,\mf,\sb,\mathcal{V}\,$ as in 
Remark~\ref{liebr}, a $\,\pi$-pro\-ject\-a\-ble vector field $\,w\,$ on 
$\,\mf\nh$, and $\,x\in\mf\,$ such that $\,w_x\w\nh\ne0$, every prescribed 
value $\,u_x\w\nh\in\mathcal{V}_{\nh x}\w$ is realized by a local section 
$\,u\,$ of $\,\mathcal{V}\,$ commuting with $\,w$. Namely, we may first 
prescribe such $\,u\,$ along a fixed co\-di\-men\-sion-one sub\-man\-i\-fold 
containing $\,x$, which is transverse to $\,w\,$ at $\,x$, and then use the 
local flow of $\,w\,$ to spread $\,u\,$ over a neghborhood of $\,x$.
\end{remark}
\begin{remark}\label{prjct}Given a vector field $\,v\,$ and a distribution 
$\,\ez\,$ on a manifold, the local flow $\,t\mapsto e^{tv}$ of $\,v\,$ 
preserves $\,\ez\,$ if and only if, whenever $\,w\,$ is a local 
section of $\,\ez\nh$, so is $\,[v,w]$. Namely, 
$\,[v,w]=\lie\hskip-.5pt_v\w w$, while, denoting by 
$\,\varTheta\mapsto (de^{tv})\varTheta\,$ the push-for\-ward action of 
$\,e^{tv}$ on tensor fields $\,\varTheta\,$ of any type, we have
\begin{equation}\label{ddt}
d\hs[(de^{tv})\varTheta]\hh/dt\,
=\,-\hs(de^{tv})\lie\hskip-.5pt_v\w\varTheta\hh.
\end{equation}
In fact, when $\,t=0$, (\ref{ddt}) is just the definition of 
$\,\lie\hskip-.5pt_v\w \varTheta$, while, for arbitrary $\,t,$ it follows 
from the group\hh-homo\-mor\-phic property of $\,t\mapsto e^{tv}\nh$.
\end{remark}
\begin{remark}\label{pralg}We say that a vector field $\,w\,$ (or, a 
distribution $\,\mathcal{H}$) on a manifold $\,\mf$ is {\medit 
pro\-ject\-able along an in\-te\-gra\-ble distribution\/} $\,\mathcal{V}\,$ on 
$\,\mf\nh$, or $\,\mathcal{V}\nh${\medit-pro\-ject\-a\-ble}, if it is 
$\,\pi$-pro\-ject\-a\-ble, as in (\ref{prj}), when restricted to any open 
sub\-man\-i\-fold of $\,\mf\,$ on which $\,\mathcal{V}$ forms the vertical 
distribution $\,\mathrm{Ker}\hskip2.3ptd\pi\,$ of a bundle projection $\,\pi$. 
For $\,w\,$ this amounts to invariance of $\,\mathcal{V}\,$ under the local 
flow of $\,w$, cf.\ Remarks~\ref{liebr} and~\ref{prjct}.
\end{remark}
\begin{remark}\label{tglvs}For an in\-te\-gra\-ble distribution $\,\zy\,$ on a 
Riemannian manifold $\,(\mf\nh,g)$, the following two conditions are 
equivalent.
\begin{enumerate}
  \def\theenumi{{\rm\alph{enumi}}}
\item[{\rm(i)}] $d_w\w[\hs g(\hn v,v)]=0\,$ for all local sections $\,v\,$ of 
$\,\zy\,$ and $\,w\,$ of $\,\zy^\perp$ such that $\,w\,$ is nonzero, 
$\,\zy$-pro\-ject\-a\-ble, and $\,[v,w]=0$.
\item[{\rm(ii)}] Every leaf (maximal integral manifold) of $\,\zy\,$ is 
totally geodesic in $\,(\mf\nh,g)$.
\end{enumerate}
In fact, let $\,b\,$ be the second fundamental form of the leaves of $\,\zy$, 
with $\,b(\hn v,v)\,$ equal to the $\,\zy^\perp$ component of 
$\,\nabla\hskip-3pt_{v}\w v$. If $\,v\,$ and $\,w\,$ commute, 
$\,d_w\w[\hs g(\hn v,v)]=2\hh g(\nabla\hskip-3pt_{w}\w v,v)
=2\hh g(\nabla\hskip-3pt_{v}\w w,v)=-\nnh2\hh g(\nabla\hskip-3pt_{v}\w v,w)
=-\nnh2\hh g(b(\hn v,v),w)$, while $\,b\,$ is symmetric and $\,v,w\,$ as in 
(i) realize, at any $\,x\in\mf\nh$, any given elements of $\,\zy_x\w$ 
and $\,\zy^\perp_x$ (see Remark~\ref{cmmut}).
\end{remark}
\begin{remark}\label{frthr}Clearly, (ii) in Remark~\ref{tglvs} also follows 
when $\,d_w\w[\hs g(\hn v,v)]=0\,$ for all $\,v,w\,$ satisfying specific 
further conditions besides (i), as long as the last line of 
Remark~\ref{tglvs} still applies.
\end{remark}
\begin{lemma}\label{intgr}{\medit 
For two in\-te\-gra\-ble distributions\/ $\,\ez^\pm$ on a manifold\/ 
$\,\mf\,$ such that the span\/ $\,\ez\,$ of\/ 
$\,\ez^+\nnh$ and\/ $\,\ez^-\nh$ has constant dimension, the 
following conditions are all equivalent.
\begin{enumerate}
  \def\theenumi{{\rm\alph{enumi}}}
\item[{\rm(a)}] $\ez\,\,$ is in\-te\-gra\-ble.
\item[{\rm(b)}] $\ez^+$ is pro\-ject\-able along\/ $\,\ez^-\nnh$.
\item[{\rm(c)}] $\ez^-$ is pro\-ject\-able along\/ $\,\ez^+\nnh$.
\end{enumerate}
If\/ {\rm(a)} -- {\rm(c)} hold, the distributions that\/ $\,\ez^\pm$ 
locally project onto are in\-te\-gra\-ble as well.
}
\end{lemma}
\begin{proof}We may assume that $\,\ez^+$ is the vertical distribution 
of a bundle projection $\,\pi:\mf\to\sb\,$ with connected fibres. First, let 
$\,\ez\,\,$ be in\-te\-gra\-ble. Since $\,\ez\,$ contains 
$\,\ez^+\nnh=\,\mathrm{Ker}\hskip2.3ptd\pi$, its leaves are unions of 
fibres and so their $\,\pi$-im\-ages form a foliation of $\,\sb$, tangent to a 
distribution $\,\mathcal{H}\,$ satisfying (\ref{prj}), which proves 
pro\-ject\-abil\-i\-ty of $\,\ez\nh$, and hence of 
$\,\ez^-\nh$, along $\,\ez^+\nh$. In other words, (a) implies 
(c). Conversely, assuming (c), we obtain 
$\,d\pi\nnh_x\w(\ez\nnh_x\w)=d\pi\nnh_x\w(\ez\nnh_x^{\hs-})
=\mathcal{H}_{\pi(x)}\w$ for all $\,x\in\mf\,$ and some distribution 
$\,\mathcal{H}$ on $\,\sb$, which is necessarily in\-te\-gra\-ble: its leaves 
are $\,\pi$-im\-ages of the leaves of $\,\ez^-\nh$. In\-te\-gra\-bil\-i\-ty of 
$\,\ez\,$ now follows, as its leaves are the $\,\pi$-pre\-im\-ages of 
those of $\,\mathcal{H}$.

Finally, as (a) involves $\,\ez^+$ and $\,\ez^-$ symmetrically, it is also 
equivalent to (b).
\end{proof}

\section{K\"ah\-ler manifolds\done}\label{km}
\setcounter{equation}{0}
For K\"ah\-ler manifolds we use symbols such as $\,(\mf\nh,g)$, where $\,\mf\,$ 
stands for the underlying complex manifold. Generally, in complex manifolds,
\begin{equation}\label{jcs}
J\,\mathrm{\ always\ denotes\ the\ com\-plex}\hyp\mathrm{struc\-ture\ tensor.} 
\end{equation}
Let $\,v\,$ be a vector field on a K\"ah\-ler manifold $\,(\mf\nh,g)$. Since 
$\,\nabla\nnh J=0$, one has
\begin{equation}\label{ajs}
\du\,=\,J\dv\,\,\mathrm{\ if\ one\ sets\ }\,\dv=\nabla\nh v\,\mathrm{\ and\ 
}\,\du=\nabla\nh u\hh,\mathrm{\ for\ }\,u=J\nh v\hh,
\end{equation}
$J,S,A\,$ being viewed as bundle morphisms $\,T\nnh\mf\to T\nnh\mf\nh$, cf.\ 
(\ref{nwt}). For the curvature tensor $\,R\,$ of a K\"ah\-ler 
manifold $\,(\mf\nh,g)\,$ and any vector fields $\,u,v\,$ on $\,\mf\nh$,
\begin{equation}\label{rcm}
R(\hn u,v)=R(Ju,J\nh v):T\nnh\mf\to T\nnh\mf\hskip8pt\mathrm{and}\hskip8ptJ:T\nnh\mf
\to T\nnh\mf\hskip8pt\mathrm{commute.}
\end{equation}
In fact, the condition $\,\nabla\nnh J=0\,$ turns $\,\nabla\,$ into a 
connection in $\,T\nnh\mf\,$ treated as a {\medit complex\/} vector bundle, 
$\,g\,$ being the real part of a $\,\nabla\nnh$-par\-al\-lel Her\-mit\-i\-an 
fibre metric, that is,
\begin{equation}\label{gjw}
g(J\nh w,J\nh w\hh'\hh)\,=\,g(\hn w,w\hh'\hh)
\end{equation}
for all vector fields $\,w,w\hh'$ on $\,\mf\nh$, and so the curvature 
operators $\,R(\hn u,v)\,$ are all com\-plex-lin\-e\-ar and 
skew-Her\-mit\-i\-an. The former property now amounts to commutation in 
(\ref{rcm}), the latter to the equality 
$\,g(R(\hn u,v)\hh w,w\hh'\hh)
=g(R(\hn w,w\hh'\hh)\hh u,v)=g(R(\hn w,w\hh'\hh)\hh Ju,J\nh v)
=g(R(Ju,J\nh v)\hh w,w\hh'\hh)$, with any vector fields $\,w,w\hh'\nh$.

Real-hol\-o\-mor\-phic vector fields $\,v\,$ on K\"ah\-ler manifolds will 
always be briefly referred to as {\medit hol\-o\-mor\-phic}. Since they are 
characterized by $\,\lie\hskip-.5pt_v\w J=0$, formula (\ref{lie}) for 
$\,\bl=J\,$ implies that, given a vector field $\,v\,$ on a K\"ah\-ler 
manifold $\,(\mf\nh,g)$,
\begin{equation}\label{hol}
v\,\mathrm{\ is\ holomorphic\ if\ and\ only\ if\ 
}\,\hs\dv=\nabla\nh v\,\mathrm{\ commutes\ with\ }\,J\hh,
\end{equation}
where $\,J,\dv:T\nnh\mf\to T\nnh\mf\,$ as in (\ref{nwt}). For any 
hol\-o\-mor\-phic vector field $\,v$,
\begin{equation}\label{kil}
\begin{array}{l}
J\nh v\hs\mathrm{\ must\ be\ hol\-o\-mor\-phic\ as\ well,\ while\ 
}\hh v\hh\mathrm{\ is\ locally}\\
\mathrm{a\ gradient\ if\ and\ only\ if\ }\,\,u\hs=J\nh v\,\,\mathrm{\ is\ a\ 
Kil\-ling\ field.}
\end{array}
\end{equation}
In fact, for $\,\dv=\nabla\nh v\,$ and $\,\du=\nabla\nh u$, (\ref{ajs}) -- 
(\ref{hol}) give $\,\du=JS=SJ$, and so $\,\du+\du\nh^*\nh=J(S-S^*)$, while the 
lo\-cal-gra\-di\-ent property of $\,v\,$ amounts to $\,S-S^*\nh=0$, and the 
Kil\-ling condition for $\,u\,$ reads $\,\du+\nnh\du\nh^*\nh=0$.
\begin{remark}\label{kilxp}As shown by Kobayashi \cite{kobayashi-fp}, if 
$\,u\,$ is a Kil\-ling vector field on a Riemannian manifold $\,(\mf\nh,g)$, 
the connected components of the zero set of $\,u\,$ are mutually isolated 
totally geodesic sub\-man\-i\-folds of even co\-di\-men\-sions. 
\end{remark}
\begin{lemma}\label{nontr}{\medit 
If a complex manifold\/ $\,\mf\,$ admits a 
K\"ah\-ler metric\/ $\,g$, with the K\"ah\-ler form\/ 
$\,\omega=g(J\,\cdot\,,\,\cdot\,)$, and\/ $\,\mg:\bbCP^k\nh\to\mf\,$ is a 
nonconstant hol\-o\-mor\-phic mapping, then\/ $\,\mg\sk\omega$ represents a 
nonzero de Rham co\-ho\-mol\-o\-gy class in\/ 
$\,H^{2\nh}(\hn\bbCP^k\nnh,\bbR)$.

Whether a hol\-o\-mor\-phic mapping\/ $\,\mg:\bbCP^k\nh\to\mf\,$ is constant, 
or not, the same is the case for all hol\-o\-mor\-phic mappings\/ 
$\,\bbCP^k\nh\to\mf\,$ sufficiently close to\/ $\,\mg\,$ in the\/ 
$\,C^0$ topology.
}
\end{lemma}
\begin{proof}Clearly, $\,\mg\,$ remains nonconstant (and hol\-o\-mor\-phic) 
when restricted to a suitable projective line $\,\bbCP^1\nh\subseteq\bbCP^k$. 
In addition to being positive sem\-i\-def\-i\-nite everywhere, the restriction 
$\,h\,$ of $\,\mg\sk g\,$ to $\,\bbCP^1$ must also be positive definite 
somewhere (or else $\,h$, being Her\-mit\-i\-an, would vanish identically, 
making $\,\mg\,$ constant on $\,\bbCP^1$). The integral of $\,\mg\sk\omega\,$ 
over $\,\bbCP^1$ is thus positive, proving our first claim. The second one 
follows since nearby continuous mappings are, obviously, hom\-o\-top\-ic to 
$\,\mg$. 
\end{proof}
\begin{remark}\label{fibra}We need the following well-known fact, valid 
both in the $\,C^\infty$ and complex (hol\-o\-mor\-phic) categories: any 
in\-te\-gra\-ble distribution with compact simply connected leaves constitutes 
the vertical distribution of a bundle projection.

The required local trivializations are provided by the -- necessarily trivial 
-- holonomy of the underlying foliation; see, for instance, 
\cite[p.\ 71]{camacho-lins-neto}.
\end{remark}
\begin{remark}\label{known}We need two more well-known facts; cf.\ 
\cite[Example 1 of Sect.\ 2.2]{shokurov}.
\begin{enumerate}
  \def\theenumi{{\rm\alph{enumi}}}
\item[{\rm(a)}] A continuous function $\,\,U\to\bbC\,$ on an open set 
$\,\,U\subseteq\bbC$, hol\-o\-mor\-phic on $\,\,U\nh\smallsetminus\varLambda$, 
where $\,\varLambda\subseteq\hs U\,$ is discrete, is necessarily 
hol\-o\-mor\-phic everywhere in $\,\,U\nh$.
\item[{\rm(b)}] The only injective hol\-o\-mor\-phic mappings 
$\,\bbCP^1\nnh\to\bbCP^1$ are bi\-hol\-o\-mor\-phisms.
\end{enumerate}
\end{remark}
\begin{remark}\label{hlext}If $\,\vps:\varPi\to\mf\,$ is a continuous 
mapping between complex manifolds, and a co\-di\-men\-sion-one complex 
sub\-man\-i\-fold $\,\varLambda\,$ of $\,\varPi\nh$, closed as a subset of 
$\,\varPi\nh$, has the property that the restrictions of $\,\vps\,$ to 
$\,\varPi\,$ and to the complement $\,\varPi\nh\smallsetminus\varLambda\,$ are 
both hol\-o\-mor\-phic, then $\,\vps\,$ is hol\-o\-mor\-phic on 
$\,\varPi\nh$.

In fact, let $\,p=\dimc\varPi\nh$. When $\,p=1$, our claim is obvious from 
Remark~\ref{known}(a). Generally, in local hol\-o\-mor\-phic coordinates 
$\,z^1\nh,\dots,z^p$ for $\,\varPi\,$ such that $\,z^2\nh=\ldots=z^p\nh=0\,$ 
on the intersection of $\,\varLambda\,$ with the coordinate domain, the 
complex partial derivatives of the components of $\,\vps\,$ (relative to any 
local hol\-o\-mor\-phic coordinates in $\,\mf$) all clearly exist: for 
$\,\partial/\partial z^1$ this follows from the case $\,p=1$. 
\end{remark}
\begin{remark}\label{posit}As usual, we call a differential $\,2$-form 
$\,\omega\hh$ on a complex manifold {\medit positive\/} if it equals the 
K\"ah\-ler form $\,g(J\,\cdot\,,\,\cdot\,)\,$ of some K\"ah\-ler metric 
$\,g$. This amounts to requiring closedness of $\,\omega\hh$ along with 
symmetry and positive definiteness of the twice-co\-var\-i\-ant tensor field 
$\,-\hs\omega\hh(J\,\cdot\,,\,\cdot\,)$.

In any complex manifold, $\,d\hh\omega=0\,$ and 
$\,\omega\hh(J\,\cdot\,,\,\cdot\,)\,$ symmetric whenever 
$\,\omega=i\hs\partial\hskip1.7pt\cro\hskip-1.5ptf$ or, equivalently, 
$\,2\hs\omega=-\hs d\hskip1pt[\hn J^*\nnh d\hskip-1.2ptf]\,$ for a 
real-val\-ued function $\,f\nh$, with the $\,1$-form 
$\,J^*\nnh d\hskip-1.2ptf\nh$, also denoted by $\,(d\hskip-1.2ptf)\hn J$, 
which sends any tangent vector field $\,v\,$ to $\,d\nnh_{J\nh v}\w f\nnh$. 
Clearly,
\begin{equation}\label{idd}
2\hh i\hs\partial\hskip1.7pt\cro\hskip-1.5ptf\,
=\,\hs2\hh i\nh f'\hn\partial\hskip1.7pt\cro\hskip-1.5pt\chi\,
-\,f''\nh d\chi\nh\wedge\nh J^*\nnh d\chi\hh,\hskip12pt\mathrm{with}
\hskip6ptf'\hs=\,d\hskip-1.2ptf/d\chi\hh,
\end{equation}
if $\,f\,$ is a $\,C^\infty$ function of a function $\,\chi\,$ on the same 
manifold (cf.\ Remark~\ref{cifot}). The ex\-te\-ri\-or-de\-riv\-a\-tive and 
ex\-te\-ri\-or-prod\-uct conventions used here, for any $\,1$-forms 
$\,\iota,\kappa\,$ and vector fields $\,u,v$, are 
$\,(d\hh\kappa)(u,v)=d_u[\kappa(v)]-d_v[\kappa(u)]-\kappa([u,v])\,$ and 
$\,(\iota\wedge\kappa)(u,v)=\iota(u)\kappa(v)-\iota(v)\kappa(u)$. When, in 
addition, $\,v\,$ is real-hol\-o\-mor\-phic, one has
\begin{equation}\label{ojv}
2\hh\omega\hh(J\nh v,\,\cdot\,)\,=\,-\hh d(d_v\w f)\,
-\,J^*\nnh[d(d\nh_{J\hn v}\w f)\nh]\hskip10pt\mathrm{for}\hskip7pt\omega
=i\hs\partial\hskip1.7pt\cro\hskip-1.5ptf.
\end{equation}
See \cite[Lemma 2]{derdzinski-bc}; the K\"ah\-ler metric used in 
\cite{derdzinski-bc} always exists locally.
\end{remark}
\begin{remark}\label{ddnsq}For the real part $\,\lr\,$ of a Her\-mit\-i\-an 
inner product in a fi\-\hbox{nite\hh-}\hskip0ptdi\-men\-sion\-al complex 
vector space $\,\on\nnh$, let $\,\hr:\on\to[\hs0,\infty)\,$ and 
$\,\mathcal{V}\,$ be the {\medit norm function\/} and {\medit complex 
radial distribution\/} on $\,\on\smallsetminus\{0\}$, so that 
$\,\hr(\xi)=\langle \xi,\xi\rangle^{\hs1\nh/2}$ and 
$\,\mathcal{V}\nnh_\xi\w\nh=\mathrm{Span}_\bbC\w(\xi)$.
\begin{enumerate}
  \def\theenumi{{\rm\alph{enumi}}}
\item[{\rm(a)}] $d\nh\hr^2$ is obviously given by 
$\,\xi\mapsto2\langle\xi,\,\cdot\,\rangle$.
\item[{\rm(b)}] $i\hs\partial\hskip1.7pt\cro\hskip-.5pt\hr^2$ coincides with 
twice the K\"ah\-ler form $\,\langle J\,\cdot\,,\,\cdot\,\rangle\,$ of the 
constant metric $\,\lr$.
\item[{\rm(c)}] $d\nh\hr^2\wedge J^*\nnh d\nh\hr^2\nh$, on 
$\,\on\smallsetminus\{0\}$, equals $\,-4\hr^2$ times the restriction of 
$\,\langle J\,\cdot\,,\,\cdot\,\rangle\,$ to $\,\mathcal{V}\nh$.
\end{enumerate}
In fact, (b) -- (c) are immediate from (a) and Remark~\ref{posit}.
\end{remark}
\begin{remark}\label{cptfb}Let $\,\pi:\mf\to\sb\,$ be a surjective submersion 
between manifolds.
\begin{enumerate}
  \def\theenumi{{\rm\alph{enumi}}}
\item[{\rm(a)}] If the pre\-im\-a\-ges $\,\pi^{-\nnh1}(y)$, $\,y\in\sb$, are 
all compact, then $\,\pi\,$ can be factored as $\,\mf\to\varPi\to\sb$, with 
a bundle projection $\,\mf\to\varPi\,$ having compact (connected) fibres, and 
a finite covering projection $\,\varPi\to\sb$.
\item[{\rm(b)}] In the case where $\,\dim\sb=\dim\varPi\,$ and $\,\mf\,$ is 
compact, $\,\pi\,$ must necessarily be a (finite) covering projection.
\end{enumerate}
Namely, (a) is a well-known fact, easily verified using parallel transports 
corresponding to a fixed vector sub\-bun\-dle $\,\mathcal{H}\,$ of 
$\,T\nnh\mf\,$ for which $\,T\nnh\mf=\mathcal{V}\nh\oplus\nnh\mathcal{H}\,$ 
cf.\ \cite[Remark 1.1]{derdzinski-roter-08}, or derived as in 
Remark~\ref{fibra}, since the foliation with the leaves $\,\pi^{-\nnh1}(y)\,$ 
has trivial holonomy. Part (b) -- in which the dimension equality means 
that $\,\pi\,$ is lo\-cal\-ly-dif\-feo\-mor\-phic -- easily follows from (a).
\end{remark}
\begin{remark}\label{ftcvr}For a K\"ah\-ler manifold $\,(\varPi\nh,h)\,$ with 
$\,\dimc\varPi=l$, any hol\-o\-mor\-phic mapping 
$\,F:\bbCP\hh^l\nh\to\varPi\,$ such that $\,F^*\nh h\,$ is a positive constant 
multiple of the Fu\-bi\-ni-Stu\-dy metric on $\,\bbCP\hh^l$ (cf.\ 
Remark~\ref{fbstm}) must be a bi\-hol\-o\-mor\-phism.

In fact, $\,F\,$ is then a covering projection (Remark~\ref{cptfb}(b)) and our 
claim follows since, due to a result of Kobayashi \cite{kobayashi-ck}, 
$\,\varPi\,$ has to be simply connected.
\end{remark}

\section{Ge\-o\-des\-\hbox{ic\hs-}\hskip0ptgra\-di\-ent K\"ah\-ler 
triples\done}\label{gk}
\setcounter{equation}{0}
Given a manifold $\,\mf\,$ endowed with a fixed connection $\,\nabla\nnh$, we 
refer to a vector field $\,v\,$ on $\,\mf\,$ as {\medit geodesic\/} if the 
integral curves of $\,v\,$ are re\-pa\-ram\-e\-trized 
$\,\nabla\nh$-ge\-o\-des\-ics. Equivalently, for some function $\,\psi\,$ on 
the open set $\,\mf'\nh\subseteq\mf\,$ on which $\,v\ne0$,
\begin{equation}\label{nvv}
\nabla\hskip-3pt_{v}\w v\,=\,\psi\hskip.4ptv\qquad\mathrm{everywhere\ 
in}\hskip7pt\mf'.
\end{equation}
A function $\,\vt\,$ on a Riemannian manifold $\,(\mf\nh,g)\,$ is said to 
{\medit have a geodesic gradient\/} if its gradient $\,v\,$ is a geodesic 
vector field relative to the Le\-vi-Ci\-vi\-ta connection $\,\nabla\nnh$.

Functions with geodesic gradients are also called {\it trans\-nor\-mal\/} 
\cite{wang,miyaoka,bolton}.
\begin{lemma}\label{ggqft}{\medit 
For a function\/ $\,\vt\,$ on a Riemannian manifold\/ $\,(\mf\nh,g)$, the 
gradient of\/ $\,\vt\,$ is a geodesic vector field if and only if\/ 
$\,Q=g(\navp,\navp)\,$ is, locally in\/ $\,\mf'\nnh$, a function of\/ $\,\vt$.
}
\end{lemma}
\begin{proof}By (\ref{tnd}), condition (\ref{nvv}) is equivalent to 
$\,dQ\wedge\hs d\vt=0$.
\end{proof}
\begin{definition}\label{ggktr}A {\medit ge\-o\-des\-ic-gra\-di\-ent 
K\"ah\-ler triple\/} $\,(\mf\nh,g,\vt)\,$ consists of any K\"ah\-ler manifold 
$\,(\mf\nh,g)\,$ and a nonconstant real-val\-ued function $\,\vt\,$ on $\,\mf\,$ 
such that the $\,g$-gra\-di\-ent $\,v=\navp\,$ is both geodesic and 
real-hol\-o\-mor\-phic.

Speaking of {\medit compactness\/} of $\,(\mf\nh,g,\vt)$, or its {\medit 
dimension}, we always mean those of the underlying complex manifold 
$\,\mf\nh$, and we call two such triples 
$\,(\mf\nh,g,\vt),\,(\hat\mf,\hg,\hat\vt)$ {\medit iso\-mor\-phic\/} if 
$\,\vt=\hat\vt\circ\varPhi\,$ and $\,g=\varPhi^*\nnh\hg\,$ for some 
bi\-hol\-o\-mor\-phism $\,\varPhi:\mf\hn\to\hat\mf\nh$. 
\end{definition}
For $\,(\mf\nh,g,\vt)\,$ as above, whenever the extrema of $\,\vt\,$ exist, we 
will also consider
\begin{equation}\label{tmm}
\mathrm{the\ }\,\vt\hyp\mathrm{pre\-im\-ages\ }\,\sd^+\nh\mathrm{\ and\ 
}\,\sd^-\nh\mathrm{\ of\ \ }\,\tp\nh=\hs\mathrm{max}\hskip2.7pt\vt\,
\mathrm{\ \ and\ \ }\,\tm\nh=\hs\mathrm{min}\hskip2.7pt\vt\hh.
\end{equation}
\begin{remark}\label{trivl}A ge\-o\-des\-\hbox{ic\hs-}\hskip0ptgra\-di\-ent 
K\"ah\-ler triple $\,(\mf\nh,g,\vt)\,$ can be {\medit trivially modified\/} to 
yield $\,(\mf\nh,g,p\vt+q)$, with any real constants $\,p\ne0\,$ and $\,q$. 
(Clearly, $\,\sd^\pm$ in (\ref{tmm}) then become switched if $\,p<0$.) Any 
such $\,(\mf\nh,g,\vt)\,$ and any complex sub\-man\-i\-fold $\,\varPi$ of 
$\,\mf\nh$, tangent to $\,v=\navp\,$ (that is, forming a union of integral 
curves of $\,v$), and not contained in a single level set of $\,\vt$, 
give rise (cf.\ Remark~\ref{gpgrd}) to the new 
ge\-o\-des\-\hbox{ic\hs-}\hskip0ptgra\-di\-ent K\"ah\-ler triple 
$\,(\varPi,g'\nh,\vt')$, where $\,g'\nh,\vt'$ are the restrictions of $\,g\,$ 
and $\,\vt\,$ to $\,\varPi\nh$.
\end{remark}
As shown next, ge\-o\-des\-\hbox{ic\hs-}\hskip0ptgra\-di\-ent K\"ah\-ler triples 
naturally arise from suitable co\-ho\-mo\-ge\-ne\-i\-ty-one isometry groups. 
\begin{lemma}\label{chone}{\medit 
Let a connected Lie group\/ $\,G\hs$ acting by hol\-o\-mor\-phic isometries on 
a K\"ah\-ler manifold\/ $\,(\mf\nh,g)$, and having some orbits of real 
co\-di\-men\-sion\/ $\,1$, preserve a nontrivial hol\-o\-mor\-phic Kil\-ling 
field\/ $\,u\,$ with zeros. If\/ $\,H^1\nh(\mf,\bbR)=\{0\}$, then\/ 
$\,(\mf\nh,g,\vt)\,$ is a ge\-o\-des\-ic-gra\-di\-ent K\"ah\-ler triple and\/ 
$\,u=J(\navp)\,\,$ for some\/ $\,G$-in\-var\-i\-ant function\/ $\,\vt\,$ on\/ 
$\,\mf\nh$.
}
\end{lemma}
\begin{proof}Since $\,H^1\nh(\mf,\bbR)=\{0\}$, (\ref{kil}) implies both the 
existence of a function $\,\vt$ with $\,u=J(\navp)$, and the fact that its 
gradient $\,v=\navp=-\nh Ju\,$ is hol\-o\-mor\-phic. Thus, elements of $\,G\,$ 
preserve $\,\vt\,$ up to additive constants. Let $\,\bs\,$ now be a fixed 
connected component of the zero set of $\,u$, so that $\,G$, being connected, 
leaves $\,\bs\,$ invariant, while $\,\vt\,$ is constant on $\,\bs\,$ (cf.\ 
Remark~\ref{kilxp}). The additive constants just mentioned are therefore equal 
to $\,0$. Due to their $\,G$-in\-var\-i\-ance, the functions $\,\vt\,$ and 
$\,Q=g(\navp,\navp)$ are constant along co\-di\-men\-sion-one orbits of 
$\,G\,$ and, consequently, functionally dependent. (Note that the union of 
such orbits is dense in $\,\mf\nh$.) Consequently, by Lemma~\ref{ggqft}, the 
gradient $\,v=\navp\,$ is a geodesic vector field.
\end{proof}
\begin{remark}\label{reddt}The assumptions about triviality of 
$\,H^1\nh(\mf,\bbR)\,$ and hol\-o\-mor\-phic\-i\-ty of $\,u\,$ in 
Lemma~\ref{chone} are well-known to be redundant when $\,\mf\,$ is compact 
\cite[p.\ 95, Corollary 4.5]{kobayashi-tg}; see also \cite[formula (A.2c) and 
Theorem A.1]{derdzinski-kp}.
\end{remark}
The following fact will be used in the proof of Theorem~\ref{jacob}.
\begin{lemma}\label{jcbfl}{\medit 
If a vector field\/ $\,w\,$ on a Riemannian manifold\/ $\,(\mf\nh,g)\,$ is 
orthogonal to a geodesic gradient\/ $\,v\,$ and commutes with\/ $\,v$, then\/ 
$\,w\,$ is a Ja\-co\-bi field along every integral curve of\/ $\,v/|v|\,$ in 
the set\/ $\,\mf'$ where\/ $\,v\ne0$.
}
\end{lemma}
\begin{proof}Fix $\,\vt:\mf\to\bbR\,$ with $\,v=\navp$. For 
$\,Q=g(\hn v,v):\mf\to\bbR$, (\ref{tnd}) and (\ref{nvv}) give 
$\,dQ\wedge\hs d\vt=0$, so that $\,|v|\,$ is, locally in $\,\mf'\nnh$, a 
$\,C^\infty$ function of $\,\vt\,$ (cf.\ Remark~\ref{cifot}). As 
$\,\lie_w\w\vt=0\,$ due to the orthogonality assumption, and 
$\,\lie_w\w v=0$, we now have $\,\lie_w\w(\hn v/|v|)=0\,$ on $\,\mf'\nnh$. The 
local flow of $\,w$, applied to any integral curve of $\,v/|v|$, thus yields a 
variation of unit-speed geodesics, and our claim follows.
\end{proof}
\begin{remark}\label{dmone}A compact ge\-o\-des\-\hbox{ic\hs-}\hskip0ptgra\-di\-ent K\"ah\-ler 
triple of complex dimension $\,1\,$ is essentially, up to isomorphisms, 
nothing else than the sphere $\,S^2$ with a rotationally invariant metric. In 
fact, necessity of rotational invariance is due to (\ref{kil}), while its 
sufficiency follows from Lemma~\ref{chone} with $\,G=S^1\nh$, for the sphere 
treated as $\,\bbCP^1$ with a K\"ah\-ler metric. (Once the $\,S^1$ action is 
chosen, the required function $\,\vt\,$ becomes unique up to trivial 
modifications, cf.\ Remark~\ref{trivl}.)
\end{remark}

\section{Examples: Grass\-mann\-i\-an and CP triples\done}\label{eg}
\setcounter{equation}{0}
In this section {\medit vector spaces\/} are complex (except for 
Remark~\ref{dfrtl}) and fi\-\hbox{nite\hh-}\hskip0ptdi\-men\-sion\-al. By 
$\,k${\it-planes\/} in a vector space $\,\vs\,$ we mean 
\hbox{$\,k$-}\hskip0ptdi\-men\-sion\-al vector sub\-spaces of $\,\vs\nh$. 
When $\,k=1$, they will also be called {\medit lines\/} in $\,\vs\nh$.

Given a vector space $\,\vs\,$ and $\,k\in\{0,1,\dots,\dimc\nnh\vs\}$, the 
Grass\-mann\-i\-an $\,\mathrm{Gr}\nh_k\w\nnh\vs\,$ is the set of all 
$\,k$-planes in $\,\vs\nh$. Each $\,\mathrm{Gr}\nh_k\w\nnh\vs\,$ naturally 
forms a compact complex manifold (see Remark~\ref{grass}), and 
$\,\mathrm{P}\vs=\mathrm{G}\hn_1\w\nnh\vs\,$ is the projective space of 
$\,\vs\nh$, provided that $\,\dimc\nnh\vs>0$. We will use the standard 
identification
\begin{equation}\label{inc}
\mathrm{P}(\hn\bbC\times\vs)\hs\,=\,\hs\vs\,\,\cup\,\hs\mathrm{P}\vs\hh,
\end{equation}
of $\,\mathrm{P}(\hn\bbC\times\vs)\,$ with the disjoint union of an open 
subset bi\-hol\-o\-mor\-phic to $\,\vs\,$ and a complex sub\-man\-i\-fold 
bi\-hol\-o\-mor\-phic to $\,\mathrm{P}\vs\,$ via the bi\-hol\-o\-mor\-phism 
sending $\,v\in\vs$, or the line $\,\bbC v\,$ spanned by 
$\,v\in\vs\smallsetminus\{0\}$, to the line $\,\bbC(1,v)\,$ or, respectively, 
$\,\bbC(0,v)$. The {\medit projectivization\/} of a hol\-o\-mor\-phic vector 
bundle $\,\nr\,$ over a complex manifold $\,\sd$ is, as usual, the 
hol\-o\-mor\-phic bundle $\,\mathrm{P}\nnh\nr\,$ of complex projective spaces 
over $\,\sd\,$ with
\begin{equation}\label{prz}
\mathrm{the\ fibres}\hskip7pt
[\mathrm{P}\nnh\nr]_y\w\,=\,\hs\mathrm{P}\vs\hskip7pt\mathrm{for}\hskip7pt\vs\,
=\,\nr\hskip-2.4pt_y\w\hh,\hskip7pt\mathrm{whenever}\hskip7pty\in\sd\hh.
\end{equation}
For a sub\-space $\,\ls\,$ of a vector space $\,\vs\,$ such that 
$\,\dimc\nnh\vs\ge2$, let $\,G\,$ be the group of all 
com\-plex-lin\-e\-ar au\-to\-mor\-phisms of $\,\vs\,$ preserving both 
$\,\ls\,$ and a fixed Her\-mit\-i\-an inner product in $\,\vs\nh$. We now 
define a compact complex manifold $\,\mf\,$ by
\begin{equation}\label{dta}
\begin{array}{rl}
\mathrm{i)}&\mf=\mathrm{Gr}\nh_k\w\nnh\vs\mathrm{,\ where\ 
}\,0<k<\dimc\nnh\vs\,\mathrm{\ and\ }\,\dimc\nh\ls=1\mathrm{,\ or}\\
\mathrm{ii)}&\mf=\mathrm{P}\vs\mathrm{,\ allowing\ 
}\,\dimc\nh\ls\in\{1,\dots,\dimc\nnh\vs-1\}\,\mathrm{\ to\ be\ arbitrary.}
\end{array}
\end{equation}
Then the hypotheses, and hence conclusions, of Lemma~\ref{chone} are 
satisfied by these $\,\mf\nh,G$, any $\,G$-in\-var\-i\-ant 
K\"ah\-ler metric $\,g\,$ on $\,\mf\nh$, and some $\,u$. Specifically, 
$\,u\,$ is a vector field arising from the central circle subgroup $\,S^1$ of 
$\,G\,$ formed by all unimodular elements of $\,G\,$ acting in both 
$\,\ls,\ls\nnh^\perp$ as multiples of $\,\mathrm{Id}$. See the remarks below.

The triples $\,(\mf\nh,g,\vt)\,$ arising via Lemma~\ref{chone} in cases 
(\ref{dta}.i) and (\ref{dta}.ii) will from now on be called {\medit 
Grass\-mann\-i\-an triples\/} and, respectively, {\medit CP triples}.

Since $\,G\,$ as above contains all unit complex multiples of $\,\mathrm{Id}$, 
its action on $\,\mf\,$ is not effective. Lemma~\ref{chone} 
does not require effectiveness of the action.
\begin{remark}\label{grcpt}The co\-ho\-mo\-ge\-ne\-i\-ty-one assumption of 
Lemma~\ref{chone} follows here from the fact that the orbits of $\,G\,$ 
coincide with the levels of a nonconstant real-an\-a\-lyt\-ic function 
$\,f:\mf\to\bbR$. Specifically, in case (\ref{dta}.i), 
$\,f(\ws)=|\hs\mathrm{pr}\hh(X,\ws)|^2\nh$, where
\begin{equation}\label{prx}
\mathrm{pr}\hh(X,\ws)\,\mathrm{\ denotes\ the\ orthogonal\ projection\ of\ 
}\,X\,\mathrm{\ onto\ }\,\ws,
\end{equation}
and $\,X\,$ is some\hs$/$any unit vector spanning $\,\ls$, which yields 
$\,G$-in\-var\-i\-ance of $\,f\nnh$. Conversely, if 
$\,\ws,\tilde\ws\in\mf\,$ and $\,f(\ws)=f(\tilde\ws)$, an element of $\,G\,$ 
sending $\,\ws\,$ to $\,\tilde\ws\,$ is provided by any linear isometry 
mapping the quadruple 
$\,\ws,\ws^\perp\nnh,\mathrm{pr}\hh(X,\ws),\mathrm{pr}\hh(X,\ws^\perp)$ 
onto $\,\tilde\ws,\tilde\ws^\perp\nnh,\mathrm{pr}\hh(X,\tilde\ws),
\mathrm{pr}\hh(X,\tilde\ws^\perp\nnh)$. (Such an isometry will preserve 
$\,X=\mathrm{pr}\hh(X,\ws)+\mathrm{pr}\hh(X,\ws^\perp\nnh)$.) In case 
(\ref{dta}.ii) we may use $\,f\,$ given by 
$\,f(\ws)=|\hs\mathrm{pr}\hh(Y\hskip-3pt_\ws\w,\ls)|^2\nh$, with 
$\,Y\hskip-3pt_\ws\w$ standing for some\hs$/$\hn any unit vector that spans 
the line $\,\ws\nh$.
\end{remark}
\begin{remark}\label{zeros}For a Grass\-mann\-i\-an or CP triple 
$\,(\mf\nh,g,\vt)$, critical points of $\,\vt$, that is, the zeros of $\,u\,$ 
or, equivalently, the fixed points of the central circle subgroup $\,S^1$ 
mentioned above, form the disjoint union of two (connected) compact complex 
sub\-man\-i\-folds, which -- since $\,\vt\,$ is clearly constant on either of 
them -- must be the same as $\,\sd^\pm$ in (\ref{tmm}). With $\,\approx\,$ 
denoting bi\-hol\-o\-mor\-phic equivalence, these $\,\sd^\pm$ are
\begin{equation}\label{spm}
\begin{array}{rlll}
\mathrm{a)}&\{\ws\in\mf:\ls\subseteq\ws\}\nh
\approx\nh\mathrm{Gr}\nh_{k-\nh1}\w[\hh\vs\hskip-2pt/\ls]\hh,\hskip9pt
&\mathrm{b)}&\{\ws\in\mf:\ws\subseteq\ls\hskip-2.5pt^\perp\}\nh
\approx\nh\mathrm{Gr}\nh_k\w\ls\hskip-2.5pt^\perp\nnh,\\
\mathrm{c)}&\{\ws\in\mf:\ws\subseteq\ls\}\nh
\approx\nh\mathrm{P}\ls\hh,
&\mathrm{d)}&\{\ws\in\mf:\ws\subseteq\ls\hskip-2.5pt^\perp\}\nh
\approx\nh\mathrm{P}\ls\hskip-2.5pt^\perp\nnh,
\end{array}
\end{equation}
where (\ref{spm}.a) -- (\ref{spm}.b) correspond to (\ref{dta}.i), and 
(\ref{spm}.c) -- (\ref{spm}.d) to (\ref{dta}.ii). In fact, each of the four 
sets clearly consists of fixed points of $\,S^1\nnh$. Conversely, suppose that 
$\,\ws\in\mf\,$ does not lie in the union of the sets (\ref{spm}.a) -- 
(\ref{spm}.b) (or, (\ref{spm}.c) -- (\ref{spm}.d)), and 
$\,\varXi\in S^1\nh\smallsetminus\{\mathrm{Id},-\nh\mathrm{Id}\}$. Then 
$\,\varXi(\ws)\ne\ws\nh$. Namely, if $\,\varXi\,$ preserved $\,\ws\nh$, the 
equalities $\,\varXi(\ls)=\ls\,$ and 
$\,\varXi(\ls\hskip-2.5pt^\perp\nnh)=\ls\hskip-2.5pt^\perp\nnh$, along with 
$\,\varXi(\ws)=\ws\,$ and $\,\varXi(\ws^\perp\nnh)=\ws^\perp\nnh$, would imply 
analogous equalities for the lines spanned by 
$\,\mathrm{pr}\hh(X,\ws),\,\mathrm{pr}\hh(X,\ws^\perp)\,$ (case 
(\ref{dta}.i)), or by 
$\,\mathrm{pr}\hh(X,\ls),\,\mathrm{pr}\hh(X,\ls\hskip-2.5pt^\perp\nnh)\,$ 
(case (\ref{dta}.ii)), cf.\ (\ref{prx}), $\,X\,$ being any unit vector in the 
line $\,\ls$ (or, respectively, in the line $\,\ws$). In either case, the 
plane $\,\ks\,$ spanned by the two $\,\varXi\nh$-in\-var\-i\-ant lines would 
contain a third such line, in the form of $\,\ls\,$ (or, respectively, 
$\,\ws$). Thus, $\,\varXi\,$ restricted to $\,\ks\,$ would be a multiple of 
$\,\mathrm{Id}$, leading to a contradiction: in both cases, $\,\ks\,$ contains 
nonzero vectors from $\,\ls\,$ and from $\,\ls\hskip-2.5pt^\perp\nnh$, which 
are eigen\-vec\-tors of $\,\varXi\,$ for two distinct eigen\-values.
\end{remark}
\begin{remark}\label{cpone}All compact ge\-o\-des\-\hbox{ic\hs-}\hskip0ptgra\-di\-ent 
K\"ah\-ler triples of complex dimension $\,1$ are obviously isomorphic to CP 
triples constructed from the data (\ref{dta}.ii) with $\,m=2$ and 
$\,\dimc\nh\ls=1$. See Remark~\ref{dmone}.
\end{remark}
\begin{remark}\label{fbstm}Given the real part $\,\lr\,$ of a Her\-mit\-i\-an 
inner product in a vector space $\,\vs\nh$, the {\medit Fu\-bi\-ni-Stu\-dy 
metric\/} on $\,\mathrm{P}\vs\,$ associated with $\,\lr\,$ is, as usual, 
uniquely characterized by requiring that the restriction of the projection 
$\,\xi\mapsto\bbC\xi\,$ to the unit sphere of $\,\lr\,$ be a Riemannian 
submersion. Another such real part $\,\lr\hn'$ yields the same 
Fu\-bi\-ni-Stu\-dy metric as $\,\lr\,$ only if $\,\lr\hn'$ is a constant 
multiple of $\,\lr$. In fact, a $\,\bbC$-lin\-e\-ar auto\-mor\-phism of 
$\,\vs\,$ taking $\,\lr\,$ to $\,\lr\hn'$ descends to an isometry 
$\,\mathrm{P}\vs\to\mathrm{P}\vs$ and hence equals a 
$\,\lr$-u\-ni\-tar\-y auto\-mor\-phism of $\,\vs\,$ followed by 
$\,r\hh e^{i\theta}$ times $\,\mathrm{Id}\,$ for some $\,r,\theta\in\bbR$, 
which gives $\,\lr=r\lr\hn'\nh$.
\end{remark}
\begin{remark}\label{grass}Given a vector space $\,\vs\,$ and 
$\,k\in\{1,\dots,\dimc\nnh\vs\}$, we denote by $\,\mathrm{St}\hn_k\w\nnh\vs$ 
and $\,\pi:\mathrm{St}\hn_k\w\nnh\vs\to\mathrm{Gr}\nh_k\w\nnh\vs\,$ 
the Stie\-fel manifold of all linearly independent ordered $\,k$-tuples of 
vectors in $\,\vs\,$ (forming an open sub\-man\-i\-folds of the 
$\,k\hs$th Cartesian power of $\,\vs$) and, respectively, the projection 
mapping sending each $\,\mathbf{e}\in\mathrm{St}\hn_k\w\nnh\vs\,$ to 
$\,\pi(\mathbf{e})=\mathrm{Span}\hh(\mathbf{e})$. Then 
$\,\mathrm{Gr}\nh_k\w\nnh\vs\,$ has a unique structure of a compact complex 
manifold of complex dimension $\,(n-k)k$, where $\,n=\dimc\nnh\vs\nh$, such 
that $\,\pi\,$ is a hol\-o\-mor\-phic submersion.
\end{remark}
\begin{remark}\label{tgvec}We will use the canonical 
isomorphic identification
\begin{equation}\label{twg}
T\hskip-2.2pt_\ws\w\nh[\hs\mathrm{Gr}\nh_k\w\nnh\vs\hs]\,\,
=\,\,\mathrm{Hom}\hs(\ws,\vs\hskip-1.8pt/\hh\ws)\hskip12pt\mathrm{whenever}
\hskip7ptk\in\{0,1,\dots,\dimc\nnh\vs\}\hh,
\end{equation}
for the tangent space of the complex manifold $\,\mathrm{Gr}\nh_k\w\nnh\vs\,$ 
at any $\,k$-plane $\,\ws\nh$, where $\,\mathrm{Hom}$ means `the space of 
linear operators' and $\,\vs\,$ is any vector space.

Namely, let $\,\mathrm{St}\hn_k\w\nnh\vs\,$ and 
$\,\pi:\mathrm{St}\hn_k\w\nnh\vs\to\mathrm{Gr}\nh_k\w\nnh\vs\,$ be as in 
Remark~\ref{grass}. Under the identification (\ref{twg}), 
$\,\vx\in\mathrm{Hom}\hs(\ws,\vs\hskip-1.8pt/\hh\ws)\,$ corresponds to 
$\,d\pi_{\mathbf{e}}\w(\tilde\vx\nh\mathbf{e})
\in T\hskip-3pt_\ws\w\nh[\hs\mathrm{Gr}\nh_k\w\nnh\vs\hs]\,$ for any linear 
lift $\,\tilde\vx:\ws\to\vs\,$ of $\,\vx\,$ and any basis 
$\,\mathbf{e}\,$ of $\,\ws\,$ which, clearly, does not depend on how such 
$\,\tilde\vx\,$ and $\,\mathbf{e}\,$ were chosen.
\end{remark}
\begin{remark}\label{dfrtl}Given a real (or, complex) manifold $\,\,U\hs$ and 
real (or, complex) vector spaces $\,\tw\nnh,\oy\nnh$, let 
$\,F:U\to\mathrm{Hom}\hs(\tw\nnh,\oy)\,$ be a $\,C^\infty$ (or, 
hol\-o\-mor\-phic) mapping giving rise to a {\medit constant\/} function 
$\,U\ni\xi\mapsto\mathrm{rank}\hskip1.7ptF(\xi)\,$ or, equivalently, 
leading to the same value of $\,k=\dim\hs\mathrm{Ker}\hskip1.4ptF\nh(\xi)\,$ 
for all $\,\xi\in U\nh$. Then the mapping
\begin{equation}\label{ker}
U\ni\xi\,
\mapsto\,\mathrm{Ker}\hskip1.4ptF\nh(\xi)\in\mathrm{Gr}\nh_k\w\nh\tw
\end{equation}
is of class $\,C^\infty$ (or, hol\-o\-mor\-phic) and its differential 
$\,T\hskip-2.3pt_\xi\w U\to\,\mathrm{Hom}\hs(\ws,\tw\hskip-1.8pt/\hh\ws)\,$ 
at any $\,\xi\in U\nh$, with $\,\ws=\mathrm{Ker}\hskip1.4ptF\nh(\xi)$, cf.\ 
(\ref{twg}), sends $\,\eta\in T\hskip-2.3pt_\xi\w U\,$ to the unique 
$\,\vx:\ws\to\tw\hskip-1.8pt/\hh\ws$ having a linear lift 
$\,\tilde\vx:\ws\to\tw\,$ such that
\begin{equation}\label{fxh}
F\nh(\xi)\circ\nh\tilde\vx\hskip7pt\mathrm{equals\ the\ restriction\ 
of}\hskip7pt-\hskip-3ptd\hskip-.8ptF\hskip-3pt_\xi\w\eta\hskip7pt\mathrm{to}
\hskip7pt\ws.
\end{equation}
Here $\,d\hskip-.8ptF\hskip-3pt_\xi\w:T\hskip-2.3pt_\xi\w U
\to\,\mathrm{Hom}\hs(\tw\nnh,\oy)$, so that $\,F\nh(\xi)\,$ and 
$\,d\hskip-.8ptF\hskip-3pt_\xi\w\eta\,$ are linear operators 
$\,\tw\nh\to\oy\nnh$.

In fact, let `regular' mean $\,C^\infty$ or hol\-o\-mor\-phic. Given 
$\,\xi\in U\nh$, we may select a subspace $\,\vs\,$ of $\,\tw\,$ so that 
$\,\tw\nh=\vs\oplus\ws\nh$, where 
$\,\ws\nh=\mathrm{Ker}\hskip1.4ptF\nh(\xi)$. For all $\,\eta\,$ near $\,\xi\,$ 
in $\,U\nh$, the restriction of $\,F(\eta)\,$ to $\,\vs\,$ is clearly an 
isomorphism onto the image of $\,F(\eta)$. Denoting by 
$\,F^{-\nnh1}_{\hskip-2.2pt\eta}$ its inverse isomorphism, we see that 
$\,F^{-\nnh1}_{\hskip-2.2pt\eta}\circ F(\eta)\,$ and $\,\mathrm{pr}\hn_\eta\w
=\mathrm{Id}-F^{-\nnh1}_{\hskip-2.2pt\eta}\circ F(\eta)\,$ coincide with the 
di\-rect-sum projections of 
$\,\tw\nh=\vs\oplus\hs\mathrm{Ker}\hskip1.4ptF\nh(\eta)$ onto $\,\vs\,$ and 
$\,\mathrm{Ker}\hskip1.4ptF\nh(\eta)$. A fixed basis $\,e_1\w,\dots,e_k\w$ of 
$\,\ws\,$ thus gives rise to the basis 
$\,\mathrm{pr}\hn_\eta\w\hh e_1\w,\dots,\mathrm{pr}\hn_\eta\w\hh e_k\w$ of any 
such $\,\mathrm{Ker}\hskip1.4ptF\nh(\eta)$, depending regularly on $\,\eta$, 
which constitutes a regular local lift of (\ref{ker}) valued in the Stie\-fel 
manifold $\,\mathrm{St}\hn_k\w\nnh\tw\,$ (see Remark~\ref{grass}), proving 
regularity of (\ref{ker}).

Replacing our $\,\xi\,$ with $\,\xi(0)\,$ for a curve 
$\,t\mapsto\xi(t)\in U\nh$, and letting $\,e_1\w(0),\dots,e_k\w(0)$ be a fixed 
basis of $\,\ws\nh=\mathrm{Ker}\hskip1.4ptF\nh(\xi(0))$, we set 
$\,e_j\w(t)=\mathrm{pr}\hn_{\xi(t)}\w\hh[\hh e_j\w(0)]\,$ and 
$\,\eta(t)=\dot\xi(t)$. Suppressing from the notation the dependence on $\,t$, 
one thus gets $\,[F(\xi)]\hh e_j\w\nh=0\,$ and, by differentiation, 
$\,[\hs d\hskip-.8ptF\hskip-3pt_\xi\w\eta\hh]\hh e_j\w\nh
+[F(\xi)]\hh\dot e_j\w\nh=0$. The operators 
$\,\tilde\pv=\tilde\pv(t):\mathrm{Ker}\hskip1.4ptF\nh(\xi)\to\vs$ defined by 
$\,\tilde\pv\nh e_j\w\nnh=\dot e_j\w$, $\,j=1,\dots,k$, yield (\ref{fxh}) at 
$\,t=0\,$ for $\,\tilde\pv\,$ instead of $\,\tilde\vx$. On the other hand, 
$\,\mathbf{e}=\mathbf{e}(t)\,$ with 
$\,\mathbf{e}=(e_1\w,\dots,e_k\w)\,$ is a regular lift of 
$\,\mathrm{Ker}\hskip1.4ptF\nh(\xi(t))\,$ to $\,\mathrm{St}\hn_k\w\nnh\tw$, 
showing that $\,d\pi_{\mathbf{e}}\w(\tilde\pv\mathbf{e})
=d\pi_{\mathbf{e}}\w\dot{\mathbf{e}}\,$ is the image of $\,\eta=\dot\xi\,$ 
under the differential of (\ref{ker}) at $\,\xi=\xi(t)$. This equality, at 
$\,t=0$, uniquely determines $\,\pv:\ws\to\tw\hskip-1.8pt/\hh\ws\,$ for which 
our $\,\tilde\pv:\ws\to\tw\,$ is a linear lift realizing 
the image just mentioned as in the final paragraph of Remark~\ref{tgvec}), 
so that $\,\pv\hh=\vx$, and our claim follows.
\end{remark}

\section{Some relevant types of data\done}\label{td}
\setcounter{equation}{0}
We will repeatedly consider quadruples $\,\tm,\tp,a,Q\,$ formed by
\begin{equation}\label{pbd}
\begin{array}{l}
\mathrm{a\ nontrivial\ closed\ interval\ }\,\,[\hh\tm,\tp]\hs\mathrm{,\ a\ 
constant\ }\,\,a\in(0,\infty)\mathrm{,}\\
\mathrm{and\ a\ }\,C^\infty\mathrm{\ function\ }\,Q\,\mathrm{\ of\ the\ 
variable\ }\,\vt\in[\hh\tm,\tp]\mathrm{,\nh\ positive}\\
\mathrm{on\ }\,(\tm,\tp)\mathrm{,\nh\ such\ that\ }\nnh\,Q=0\,\mathrm{\ and\ 
}\hn\,dQ\hh/\nh d\vt=\mp2a\hn\,\mathrm{\ at\ }\hn\,\vt=\tpm\hh,
\end{array}
\end{equation}
$\mp\,$ being the opposite sign of $\,\pm\hh$. As explained below, we may then 
also choose
\begin{equation}\label{sgn}
\begin{array}{l}
\mathrm{a\ sign\ }\hs\pm\,(\mathrm{equal\ to\ }\hs+\hs\mathrm{\ or\ 
}\hs-),\mathrm{\nh\ and\ a\ }C^\infty\hskip-2.5pt\mathrm{\ 
dif\-feo\-mor\-phism}\\
(\tm,\tp)\ni\vt\,\,\mapsto\,\,\hr\in(0,\infty)\,\,\mathrm{\ such\ that\ 
}\,\,d\nh\hr\hs/d\vt\,=\,\mp\hh a\hr/Q\hh,
\end{array}
\end{equation}
a function $\,(0,\infty)\ni\hr\mapsto f\in\bbR$, unique up to an additive 
constant, with
\begin{equation}\label{dfr}
a\hr\hs\,d\hskip-1.2ptf\nnh/\nh d\nh\hr\,=\,2\hs|\hs\vt-\hn\tpm|\hh,
\end{equation}
and the unique increasing dif\-feo\-mor\-phism 
$\,(0,\infty)\ni\hr\mapsto\sigma\in(0,\delta)\,$ such that
\begin{equation}\label{dsr}
a\hr\,d\sigma\nh/\hn d\nh\hr\,=\,Q^{\hs1\nh/2}\hskip10pt\mathrm{and}\hskip7pt
\sigma\nh\to0\hskip6pt\mathrm{as}\hskip6pt\hr\to0\hh,
\end{equation}
the inverse of $\,\vt\mapsto\hr\,$ in (\ref{sgn}) being used to treat 
$\,\vt,Q\,$ as functions of $\,\hr$, for
\begin{equation}\label{int}
\delta\in(0,\infty)\,\mathrm{\ equal\ to\ the\ integral\ of\ 
}\,Q^{-\nnh1\nh/2}\hs d\vt\,\mathrm{\ over\ }\,(\tm,\tp)\hh.
\end{equation}
In fact, one easily verifies that $\,\vt\mapsto\hr\,$ and 
$\,\hr\mapsto\sigma\,$ with the stated properties exist, while $\,\delta\,$ is 
finite. See \cite[Remark 5.1]{derdzinski-maschler-06}, 
\cite[Theorem~10.2(iii)]{derdzinski-kp}, \cite[p.\ 1661]{derdzinski-kp}.
\begin{lemma}\label{xtens}{\medit 
Any\/ $\,f\,$ satisfying\/ {\rm(\ref{dfr})} has a\/ $\,C^\infty$ extension 
to\/ $\,[\hs0,\infty)$, which is also a\/ $\,C^\infty$ function of\/ 
$\,\hr^2\nh\in[\hs0,\infty)\,$ in a sense analogous to 
Remark\/~{\rm\ref{cifot}}. The first and second derivatives of 
$\,f\,$ with respect to\/ $\,\hr^2$ obviously are\/ 
$\,|\hs\vt-\hn\tpm|/(a\hr^2)\,$ and\/ 
$\,(Q-2a\hh|\hs\vt-\hn\tpm|)/(2\hh a^2\hskip-2pt\hr^4)$.
}
\end{lemma}
\begin{proof}Since 
$\,a\hr^2d\hskip-1.2ptf\nnh/\hn d\nh\hr^2\nh=|\hs\vt-\hn\tpm|$, it clearly 
suffices to establish $\,C^\infty\nnh$-ex\-ten\-si\-bil\-i\-ty of 
$\,\hr^2\nh\mapsto(\vt-\hn\tpm)/\hr^2$ to $\,[\hs0,\infty)$. To this end, 
note that, according to 
\cite[top line on p.\ 83 and Remark 4.3(ii)]{derdzinski-maschler-06}, 
$\,(\vt-\hn\tpm)/\hr^2$ (or, $\,Q/\hr^2$) may be extended to a $\,C^\infty$ 
function of $\,\vt\in[\hh\tm,\tp]\smallsetminus\{\tmp\}\,$ (or, of 
$\,\hr^2\nh\in[\hs0,\infty)$) with a nonzero value at $\,\tpm$ (or, 
respectively, at $\,0$). However, by (\ref{sgn}), 
$\,2\hh a\,d\vt/\hn d\nh\hr^2\nh=\mp\hh Q/\hr^2\nh$, so that the variables 
$\,\vt\,$ and $\,\hr^2$ just mentioned depend dif\-feo\-mor\-phic\-al\-ly on 
each other.
\end{proof}
\begin{remark}\label{uniqe}It is the limit condition in (\ref{dsr}) that makes 
$\,\sigma\hs$ unique; by contrast, $\,\hr$ with (\ref{sgn}) is only unique up 
to a positive constant factor.
\end{remark}
\begin{remark}\label{tends}In (\ref{sgn}), the increasing function 
$\,\mp\hr\,$ of the variable $\,\vt\,$ clearly tends to $\,0\,$ as 
$\,\vt\to\tpm$, and  to $\,\infty\,$ as $\,\vt\to\tmp$.
\end{remark}
\begin{remark}\label{compo}The composite 
$\,(0,\delta)\ni\sigma\mapsto\hr\mapsto\vt\in(\tm,\tp)\,$ of the inverses of 
the above two dif\-feo\-mor\-phisms is the unique solution of the autonomous 
equation $\,d\vt/d\sigma=\mp\hh Q^{\hs1\nh/2}\nh$, with the sign $\,\pm\,$ 
fixed as in (\ref{sgn}), such that $\,\vt\to\tpm$ as $\,\sigma\to0$. (We say 
`autonomous' since (\ref{pbd}) makes $\,Q\,$ a function of $\,\vt$.) In fact, 
any two solutions of the equivalent equation 
$\,d\sigma/d\vt=\mp\hh Q^{-\nnh1\nh/2}$ differ by a constant.
\end{remark}
\begin{remark}\label{iddbf}With (\ref{pbd}) -- (\ref{dfr}) fixed as above, 
suppose that $\,\hr\,$ simultaneously denotes a positive function on a complex 
manifold $\,\mf'\nh$, which also turns $\,f\,$ into a function 
$\,\mf'\nh\to\bbR$. Now (\ref{idd}) for $\,\chi=\hr^2$ and the formulae in 
Lemma~\ref{xtens} give
\begin{equation}\label{ddf}
4\hh i\hh a^2\hskip-2pt\hr^4\partial\hskip1.7pt\cro\hskip-1.5ptf\,
=\,\hs4\hh i\hh a\hh\hr^2|\hs\vt-\hn\tpm|\,
\partial\hskip1.7pt\cro\hskip-.5pt\hr^2\hs
+\,(2a\hh|\hs\vt-\hn\tpm|-Q)\,d\nh\hr^2\wedge J^*\nnh d\nh\hr^2.
\end{equation}
\end{remark}
\begin{remark}\label{smdiv}If a $\,C^\infty$ function 
$\,\vt\mapsto\zeta(\vt)\,$ on an interval $\,I\hs$ having an endpoint $\,c$ 
vanishes at $\,c$, then $\,\zeta(\vt)=(\vt-c)\chi(\vt)\,$ for a $\,C^\infty$ 
function $\,\chi\,$ on $\,I\hs$ equal, at $\,c$, to the derivative 
$\,\zeta\hh'$ of $\,\zeta$. This is the Taylor formula, with 
$\,\chi(\vt)=\int_0^1\zeta\hh'(c+s(\vt-c))\,ds$.
\end{remark}
\begin{remark}\label{adtet}Given a $\,C^\infty$ function 
$\,t\mapsto\gamma(t)\,$ on an interval $\,I\hs$ having the endpoint $\,0\,$ 
such that $\,(\gamma(0),\dot{\hskip-1.8pt\gamma\hskip0pt}(0))=(0,1)$, where  
$\,(\hskip2.3pt)\nnh\dot{\phantom o}\nh=\,d/dt$, and $\,\gamma>0\,$ on 
$\,t\in I\nh\smallsetminus\{0\}$, there exists a $\,C^\infty$ 
function $\,\theta:I\nh\to\bbR$, unique up to multiplications by positive 
constants, for which $\,\dot\theta>0\,$ and 
$\,\gamma\hs\dot\theta=\hh\theta\,$ everywhere in $\,I\nh$, while 
$\,\theta(0)=0$. Namely, Remark~\ref{smdiv} implies that some $\,C^\infty$ 
functions $\,\beta\,$ (without zeros) and $\,\alpha\,$ on $\,I\hs$ satisfy the 
conditions $\,\gamma(t)=t/\beta(t)\,$ and $\,\beta(t)=1+t\dot\alpha(t)$. Thus, 
$\,1/\gamma(t)=\dot\alpha(t)+1/t\,$ has the antiderivative 
$\,t\mapsto\alpha(t)+\log|\hh t\hskip-.2pt|\,$ on 
$\,t\in I\nh\smallsetminus\{0\}$, and we may set 
$\,\theta(t)=t\hs e^{\alpha(t)}\nh$.
\end{remark}

\section{The Chern connection\done}\label{ck}
\setcounter{equation}{0}
Let $\,\lr\,$ be the real part of a Her\-mit\-i\-an 
fibre metric in a hol\-o\-mor\-phic complex vector bundle $\,\nr\,$ over a 
complex manifold $\,\sd$. The {\medit Chern connection\/} of $\,\lr$, also 
called its {\medit Her\-mit\-i\-an connection}, is the unique connection 
$\,\mathrm{D}\,$ in $\,\nr\,$ which makes $\,\lr$ parallel and satisfies the 
condition $\,\mathrm{D}^{0,1}\nnh=\,\cro$, meaning that, for any 
section $\,\xi\,$ of $\,\nr\nnh$, the com\-plex-anti\-lin\-e\-ar part of the 
real vec\-tor-bun\-dle morphism $\,\mathrm{D}\hh\xi:T\nnh\sd\to\nr$ equals 
$\,\hs\cro\xi$, the image of $\,\xi\hs$ under the Cau\-chy-Rie\-mann operator.
Cf.\ \cite[Sect.\ 1.4]{kobayashi-dg}.

The following five properties of the Chern connection $\,\mathrm{D}\,$ are 
well known -- (e) is obvious; for (a) -- (d) see \cite[p.\ 32]{ballmann}, 
\cite[Propositions 1.3.5, 1.7.19 and 1.4.18]{kobayashi-dg}.
\begin{enumerate}
  \def\theenumi{{\rm\alph{enumi}}}
\item[{\rm(a)}] $\mathrm{D}\,$ depends on $\,\nr\,$ and $\,\lr\,$ functorially 
with respect to all natural operations, including $\,\mathrm{Hom}$, direct 
sums, and pull\-backs under hol\-o\-mor\-phic mappings.
\item[{\rm(b)}] $R^{\mathrm{D}}\nh(J\nh w,J\nh w\hh'\hh)
=R^{\mathrm{D}}\nh(\hn w,w\hh'\hh)$, with the notation of (\ref{nwt}) and 
(\ref{jcs}), where $\,w,w\hh'\nh,R^{\mathrm{D}}$ are any vector fields on 
$\,\sd\,$ and, respectively, the curvature tensor of $\,\mathrm{D}$.
\item[{\rm(c)}] $\mathrm{D}\,$ is the Le\-vi-Ci\-vi\-ta connection of 
$\,\lr\,$ if $\,\nr=T\nnh\sd\,$ and $\,\lr\,$ is a K\"ah\-ler metric.
\item[{\rm(d)}] $\mathrm{D}\,$ coincides with the normal connection in the 
normal bundle $\,\nr\hskip-2.3pt\sd\,$ for any totally geodesic complex 
sub\-man\-i\-fold $\,\sd\,$ a K\"ah\-ler manifold $\,(\mf\nh,g)\,$ and the 
Riemannian fibre metric $\,\lr\,$ in $\,\nr\,$ induced by $\,g$. (In addition, 
it follows then that $\,\nr\,$ must be a hol\-o\-mor\-phic sub\-bun\-dle of 
$\,T\nnh\mf\nh$.)
\item[{\rm(e)}] $\mathrm{D}\hh$-par\-al\-lel sections of $\,\nr\,$ are 
hol\-o\-mor\-phic.
\end{enumerate}
Any given local hol\-o\-mor\-phic coordinates $\,z^\lambda$ in $\,\sd\,$ and 
local hol\-o\-mor\-phic trivializing sections $\,e_b\w$ for $\,\nr\nnh$, 
on the same domain, associate with the Her\-mit\-i\-an fibre metric 
$\,(\hskip2.3pt,\hskip1.3pt)$ having the real part $\,\lr$, sections 
$\,\xi\,$ of $\,\nr\nnh$, and any connection $\,\mathrm{D}$, their component 
functions $\,\gamma_{b\bar c}\w\nh
=(\nh e_b\w,e\nh_{c\phantom{b}}\w\hskip-3.2pt)\,$ and 
$\,\xi^b\nh,\varOmega_{\hskip-1.2pt\mu b}^{\,c},
\varOmega_{\hskip-1.2pt\bar\mu b}^{\,c}$, the latter characterized by 
$\,\xi=\xi^be_b\w$, as well as $\,\mathrm{D}\nnh_\lambda\w e_b\w
=\varOmega_{\hskip-1.2pt\lambda b}^{\hh c}e\nh_{c\phantom{b}}\w\hskip-2.3pt$ 
and $\,\mathrm{D}\nh_{\bar\mu}\w e_b\w
=\varOmega_{\hskip-1.2pt\bar\mu b}^{\hh c}e\nh_{c\phantom{b}}\w\hskip-2.3pt$. 
Here $\,\mathrm{D}\nnh_\lambda\w$ and $\,\mathrm{D}\nh_{\bar\mu}\w$ denote the 
$\,\mathrm{D}\hh$-co\-var\-i\-ant differentiations in the direction od the 
complexified coordinate vector fields 
$\,\partial_\lambda\w=\partial/\partial z^\lambda$ and 
$\,\partial_{\bar\mu}\w=\partial/\partial\bar z^{\bar\mu}\nh$, repeated 
indices are summed over and, with 
$\,\bar z^{\bar\mu}\nh,\bar\gamma_{\bar c\hh b}\w$ and 
$\,\bar\eta^{\bar c}$ standing for the complex conjugates of 
$\,z^\mu\nh,\gamma_{c\bar b}\w$ and $\,\eta^c\nh$, Her\-mit\-i\-an symmetry of 
$\,(\hskip2.3pt,\hskip1.3pt)$ amounts to 
$\,\gamma_{b\bar c}\w\nh=\gamma_{\bar c\hh b}\w$, while  
$\,(\xi,\eta)=\gamma_{b\bar c}\w\,\xi^b\nh\eta^{\bar c}$ whenever 
$\,\xi,\eta\,$ are sections of $\,\nr\nnh$.

The real coordinate vector fields corresponding to the real coordinates 
$\,\mathrm{Re}\,z^\mu,\,\mathrm{Im}\,z^\mu$ then are 
$\,\partial_\mu\w\nh+\hs\partial_{\bar\mu}^{\phantom{_{j_j}}}\hskip-3pt,
\,i\hs(\partial_\mu\w\nh
-\hs\partial_{\bar\mu}^{\phantom{_{j_j}}}\hskip-3pt)$, and so, given a 
complexified vector field $\,v=v^\mu\partial_\mu\w\nh
+\hh v^{\bar\mu}\partial_{\bar\mu}^{\phantom{_{j_j}}}\hskip-3pt$,
\begin{equation}\label{rvf}
v\,\mathrm{\ is\ real\ if\ and\ only\ if\ each\ } v^{\bar\mu}\mathrm{\ equals\ 
the\ complex\ conjugate\ of\ }\,v^\mu\nh.
\end{equation}
For a connection $\,\mathrm{D}\,$ to make $\,(\hskip2.3pt,\hskip1.3pt)\,$ 
parallel it is clearly necessary and sufficient that 
$\,\partial_\mu\w\gamma_{b\bar c}\w
=\varOmega_{\hskip-1.2pt\mu b}^{\,e}\gamma_{e\bar c}\w
+\varOmega_{\hskip-1.2pt\mu\bar c}^{\,\bar e}\gamma_{b\bar e}\w$ 
and $\,\partial_{\bar\mu}\w\gamma_{b\bar c}\w
=\varOmega_{\hskip-1.2pt\bar\mu b}^{\,e}\gamma_{e\bar c}\w
+\varOmega_{\hskip-1.2pt\bar\mu\bar c}^{\,\bar e}\gamma_{b\bar e}\w$, where 
$\,\varOmega_{\hskip-1.2pt\mu\bar c}^{\,\bar e}$ and 
$\,\varOmega_{\hskip-1.2pt\bar\mu\bar c}^{\,\bar e}$ are defined to be the 
complex conjugates of $\,\varOmega_{\hskip-1.2pt\bar\mu c}^{\,e}$ and 
$\,\varOmega_{\hskip-1.2pt\mu c}^{\,e}$. On the other hand, the 
condition $\,\mathrm{D}^{0,1}\nnh=\,\cro\,$ is obviously equivalent to 
$\,\mathrm{D}\nh_{\bar\mu}\w e_b\w=0$, that is, 
$\,\varOmega_{\hskip-1.2pt\bar\mu b}^{\hh c}=0$. Existence and uniqueness of 
the Chern connection $\,\mathrm{D}$ now follow, its component functions being
\begin{equation}\label{chc}
\varOmega_{\hskip-1.2pt\bar\mu b}^{\hh c}\hs=\,0\hskip10pt
\mathrm{and}\hskip10pt\varOmega_{\hskip-1.2pt\lambda b}^{\,d}\hskip7pt
\mathrm{\ characterized\ 
by}\hskip7pt\varOmega_{\hskip-1.2pt\lambda b}^{\,d}\gamma_{d\bar c}\w\hs
=\,\partial_\lambda\w\gamma_{b\bar c}\w\hh.
\end{equation}
Consequently, the Chern connection $\,\mathrm{D}\,$ has the curvature 
components
\begin{equation}\label{ccd}
R_{\lambda\bar\mu b\bar c}\w=-R_{\bar\mu\lambda b\bar c}\w
=\partial_\lambda\w\partial_{\bar\mu}\w\gamma_{b\bar c}\w
-\varOmega_{\hskip-1.2pt\lambda b}^{\,d}
\partial_{\bar\mu}\w\gamma_{d\bar c}\w\hh,\hskip14pt
R_{\lambda\mu b\bar c}\w=R_{\bar\lambda\bar\mu b\bar c}\w=0
\end{equation}
(which implies (b) above). Here $\,R_{\lambda\bar\mu b\bar c}\w
=(R^{\mathrm{D}}\nh(\partial_\lambda\w,\partial_{\bar\mu}\w)\hh 
e_b\w,e\nh_{c\phantom{b}}\w\hskip-3.2pt)$, and analogously for the other three 
pairs $\,\bar\mu\lambda,\hh\lambda\mu,\hh\bar\lambda\bar\mu\,$ of indices. We 
obtain (\ref{ccd}) from (\ref{cur}) via differentiation by parts, noting that 
$\,[\hh\partial_\lambda\w,\partial_\mu\w]
=[\hh\partial_\lambda\w,\partial_{\bar\mu}\w]=0\,$ while, by (\ref{chc}), 
$\,\mathrm{D}\nh_{\bar\mu}\w e_b\w=0\,$
and 
$\,(\mathrm{D}\nnh_\lambda\w e_b\w,e\nh_{c\phantom{b}}\w\hskip-3.2pt)
=\partial_\lambda\w\gamma_{b\bar c}\w$. For instance, 
$\,R_{\lambda\bar\mu b\bar c}\w
=(\mathrm{D}\nh_{\bar\mu}\w
\mathrm{D}\nnh_\lambda\w e_b\w,e\nh_{c\phantom{b}}\w\hskip-3.2pt)
=\partial_{\bar\mu}\w(
\mathrm{D}\nnh_\lambda\w e_b\w,e\nh_{c\phantom{b}}\w\hskip-3.2pt)
-(\mathrm{D}\nnh_\lambda\w e_b\w,
\mathrm{D}\nnh_\mu\w e\nh_{c\phantom{b}}\w\hskip-3.2pt)
=\partial_{\bar\mu}\w\partial_\lambda\w\gamma_{b\bar c}\w
-\varOmega_{\hskip-1.2pt\lambda b}^{\,d}
\partial_{\bar\mu}\w\gamma_{\bar cd}\w$. Similarly, 
$\,(\mathrm{D}\nnh_\mu\w
\mathrm{D}\nnh_\lambda\w e_b\w,e\nh_{c\phantom{b}}\w\hskip-3.2pt)
=\partial_\mu\w(
\mathrm{D}\nnh_\lambda\w e_b\w,e\nh_{c\phantom{b}}\w\hskip-3.2pt)
-(\mathrm{D}\nnh_\lambda\w e_b\w,
\mathrm{D}\nh_{\bar\mu}\w e\nh_{c\phantom{b}}\w\hskip-3.2pt)\,$ equals 
$\,\partial_\mu\w(
\mathrm{D}\nnh_\lambda\w e_b\w,e\nh_{c\phantom{b}}\w\hskip-3.2pt)
=\partial_\mu\w\partial_\lambda\w\gamma_{b\bar c}\w$, which is symmetric in 
$\,\lambda,\mu\,$ (and so $\,R_{\lambda\mu b\bar c}\w=0$).

Let $\,f:\sd\to\bbR$. With the notational conventions of Remark~\ref{posit},
\begin{equation}\label{ddb}
i\hs\partial\hskip1.7pt\cro\hskip-1.5ptf\,\,
=\,\,[\hh\partial_\lambda\w\partial_{\bar\mu}\w f\hh]\,
dz^\lambda\nh\wedge\hh d\bar z^{\bar\mu}\hh,
\end{equation}
as $\,d\hskip-1.2ptf=[\hh\partial_\lambda\w f\hh]\,dz^\lambda\nh
+[\hh\partial_{\bar\mu}\w f\hh]\,d\bar z^{\bar\mu}$ and, in the case where 
$\,f=z^\lambda$ (or, $\,f=\bar z^{\bar\mu}$), the com\-plex-val\-ued 
$\,1$-form $\,J^*\nnh d\hskip-1.2ptf=(d\hskip-1.2ptf)\hn J\,$ equals 
$\,i\hs dz^\lambda$ or, respectively, $\,-i\hs d\bar z^{\bar\mu}\nh$. 
Equivalently,
\begin{equation}\label{lmz}
(i\hs\partial\hskip1.7pt\cro\hskip-1.5ptf)
(\partial_\lambda\w,\partial_{\bar\mu}\w)\,\,
=\,\,\partial_\lambda\w\partial_{\bar\mu}\w f,\hskip12pt
(i\hs\partial\hskip1.7pt\cro\hskip-1.5ptf)
(\partial_\lambda\w,\partial_\mu\w)\,\,
=\,\,(i\hs\partial\hskip1.7pt\cro\hskip-1.5ptf)
(\partial\hh_{\bar\lambda},\partial_{\bar\mu}\w)\,\,=\,\,0\hh.
\end{equation}
Thus, by (\ref{ddb}) and (\ref{rvf}), for 
$\,\omega=i\hs\partial\hskip1.7pt\cro\hskip-1.5ptf\,$ and any real vector 
fields $\,v,u\,$ on $\,\sd$,
\begin{equation}\label{ovu}
\omega\hh(\hn v,u)\,\,=\,\,2\hs\mathrm{Im}\hskip2pt
(v^\lambda u^{\bar\mu}\partial_\lambda\w\partial_{\bar\mu}\w f)\hh.
\end{equation}
\begin{lemma}\label{ddrsq}{\medit 
With\/ $\,\pi:\nr\nh\to\sd\,$ and\/ $\,\hr:\nr\nh\to[\hs0,\infty)\,$ denoting 
the bundle projection and the norm function of\/ $\,\lr$, \,for\/ 
$\,\nr\nh,\sd,\lr\,$ as above, the Chern connection\/ $\,\mathrm{D}\,$ of\/ 
$\,\lr\,$ and the $\,2$-form\/ 
$\,\omega=i\hs\partial\hskip1.7pt\cro\hskip-.5pt\hr^2$ satisfy the following 
conditions.
\begin{enumerate}
  \def\theenumi{{\rm\roman{enumi}}}
\item[{\rm(i)}] The horizontal distribution of\/ $\,\mathrm{D}\,$ constitutes 
a complex vector sub\-bun\-dle of\/ $\,T\nnh\nr\nnh$, and is\/ 
$\,\omega$-or\-thog\-o\-nal, in an obvious sense, to the vertical 
distribution\/ $\,\mathrm{Ker}\hskip2.3ptd\pi$.
\item[{\rm(ii)}] Part\/ {\rm(b)} of Remark\/~{\rm\ref{ddnsq}} describes\/ 
$\,\omega\hs$ restricted to any fibre\/ $\hs\nr\hskip-2.4pt_y\w$ \,of\/ 
$\hs\nr\nh$, where\/ $\,y\in\sd$.
\item[{\rm(iii)}] Whenever\/ $\,x=(y,\xi)\in\nr\nnh$, cf.\ 
Remark\/~{\rm\ref{tlspc}}, the restriction of\/ $\,\omega\nh_x\w$ to the 
horizontal space of\/ $\,\mathrm{D}\,$ at\/ $\,x\,$ equals the\/ 
$\,d\pi\nnh_x\w$-pull\-back of the\/ $\,2$-form\/ 
$\,\la\nh R_y^{\mathrm{D}}\nh(\,\cdot\,,\hs\cdot\,)\hh\xi,i\hh\xi\ra\,$ at\/ 
$\,y\in\sd$.
\item[{\rm(iv)}] The Chern connection\/ $\,\hat{\mathrm{D}}\,$ of\/ 
$\,e^\theta\nnh\lr$, \,for any function\/ $\,\theta:\sd\to\bbR$, is related 
to\/ $\,\mathrm{D}\,$ by\/ $\,\hat{\mathrm{D}}=\mathrm{D}
+(\partial\theta)\nnh\otimes\nh\mathrm{Id}\hh$, so that\/ 
$\,\hat\varOmega_{\hskip-1.2pt\bar\mu b}^{\hh c}\nh=0\,$ and\/ 
$\,\hat\varOmega_{\hskip-1.2pt\lambda b}^{\,c}\nh
=\varOmega_{\hskip-1.2pt\lambda b}^{\,c}\nh
+\hs\delta_{\nh b}^c\partial\nnh_\lambda\w\nh\theta$. Also, 
$\,\widehat v_x\w\nh\in\mathrm{Span}_\bbC\w(\widetilde v_x\w,\xi)$, 
where\/ $\,\widetilde v_x\w$ \,and\/ $\,\widehat v_x\w$ \,denote the\/ 
$\,\mathrm{D}\hh$-hor\-i\-zon\-tal and\/ 
$\,\hat{\mathrm{D}}\hh$-hor\-i\-zon\-tal lifts of any\/ $\,v\in\tyb\,$ to\/ 
$\,x=(y,\xi)\in\nr\nh$.\phantom{$1^{1^1}$}
\end{enumerate}
}
\end{lemma}
\begin{proof}In terms of complexified coordinate vector fields 
$\,\partial_\lambda\w,\,\partial_{\bar\mu}\w$ in $\,\sd\,$ and their analogs 
$\,\partial_\lambda\w,\,\partial_{\bar\mu}\w,\,\partial_b\w,
\,\partial_{\bar c}\w$ corresponding to the local hol\-o\-mor\-phic 
coordinates $\,z^\lambda\nh,\xi^b$  in $\,\nr\nh$, the 
$\,\mathrm{D}\hh$-hor\-i\-zon\-tal lifts 
$\,\widetilde{\partial}_\lambda\w,\,\widetilde{\partial}_{\bar\mu}\w$ of the 
former are given by
\begin{equation}\label{hlo}
\widetilde{\partial}_\mu\w\hs=\,\partial_\mu\w\hs
-\,\varOmega_{\hskip-1.2pt\mu b}^{\hh e}\xi^b\partial_e\w\hs
-\,\varOmega_{\hskip-1.2pt\bar\mu\bar e}^{\hh\bar c}\xi^{\bar e}
\partial_{\bar c}\w\hh,\hskip16pt\widetilde{\partial}_{\bar\mu}\w\hs
=\,\partial_{\bar\mu}\w\hh.
\end{equation}
Since $\,J\partial_\lambda\w=i\hs\partial_\lambda\w$ and 
$\,J\partial_{\bar\mu}\w=-i\hs\partial_{\bar\mu}\w$, (\ref{hlo}) implies  
com\-plex-lin\-e\-ar\-i\-ty of the $\,\mathrm{D}\hh$-hor\-i\-zon\-tal lift 
operation relative to the com\-plex-stru\-cture tensors $\,J$, proving the 
first part of (i). Assertion (ii) is in turn obvious from naturality of the 
operator $\,i\hs\partial\hskip1.7pt\cro$.

Next, applying (\ref{lmz}) to $\,f=\hr^2$ and the coordinates 
$\,z^\lambda\nh,\xi^b$ rather than just $\,z^\lambda\nh$, we obtain 
$\,\omega\hh(\partial_\lambda\w,\partial_{\bar\mu}\w)
=\xi^b\xi^{\bar c}\partial_\lambda\w\partial_{\bar\mu}\w\gamma_{b\bar c}\w$,  
$\,\omega\hh(\partial_e\w,\partial_{\bar\mu}\w)
=\xi^{\bar c}\partial_{\bar\mu}\w\gamma_{e\bar c}\w$, 
$\,\omega\hh(\partial_{\bar c}\w,\partial_{\bar\mu}\w)
=\omega\hh(\partial_{\bar c}\w,\partial_{\bar e}\w)=0$, 
$\,\omega\hh(\partial_\lambda\w,\partial_{\bar c}\w)
=\xi^b\partial_\lambda\w\gamma_{b\bar c}\w$, and 
$\,\omega\hh(\partial_b\w,\partial_{\bar c}\w)=\gamma_{b\bar c}\w$. Thus, 
$\,\omega\hh(\widetilde{\partial}_\lambda\w,\partial_{\bar c}\w)=0\,$ by 
(\ref{chc}), which amounts to the remaining claim in (i); similarly, 
(\ref{hlo}) and (\ref{ccd}) yield 
$\,\omega\hh(\widetilde{\partial}_\lambda\w,\widetilde{\partial}_{\bar\mu}\w)
=\xi^b\xi^{\bar c}(\partial_\lambda\w\partial_{\bar\mu}\w\gamma_{b\bar c}\w
-\varOmega_{\hskip-1.2pt\lambda b}^{\,e}\partial_{\bar\mu}\w\gamma_{e\bar c}\w)
=\xi^b\xi^{\bar c}R_{\lambda\bar\mu b\bar c}\w$. Now (iii) follows: for the 
$\,\mathrm{D}\hh$-hor\-i\-zon\-tal lifts $\,\widetilde v,\widetilde u\,$ of 
any real vectors $\,v,u\in\tyb$, (\ref{lmz}) -- (\ref{ovu}) and the last 
equality give 
$\,\omega\nh_x\w(\widetilde v,\widetilde u)=2\hs\mathrm{Im}\hskip2pt
(v^\lambda u^{\bar\mu}R_{\lambda\bar\mu b\bar c}\w\xi^b\xi^{\bar c})$, while 
at the same time $\,(R_y^{\mathrm{D}}\nh(\hn v,u)\hh\xi,i\hh\xi)\,$ is 
imaginary, and so 
$\,\la\nh R_y^{\mathrm{D}}\nh(\hn v,u)\hh\xi,i\hh\xi\ra
=\mathrm{Im}\hskip2pt(R_y^{\mathrm{D}}\nh(\hn v,u)\hh\xi,i\hh\xi)
=(R_y^{\mathrm{D}}\nh(\hn v,u)\hh\xi,i\hh\xi)
=\xi^b\xi^{\bar c}
(R_y^{\mathrm{D}}\nh(\hn v,u)\hh e_b\w,e\nh_{c\phantom{b}}\w\hskip-3.2pt)\,$ 
which, analogously, equals $\,2\hs\mathrm{Im}\hskip2pt
(v^\lambda u^{\bar\mu}R_{\lambda\bar\mu b\bar c}\w\xi^b\xi^{\bar c})$. 
Finally, (\ref{chc}) and (\ref{hlo}) imply (iv).
\end{proof}

\section{Examples: Vector bundles\done}\label{ev}
\setcounter{equation}{0}
The ge\-o\-des\-\hbox{ic\hs-}\hskip0ptgra\-di\-ent K\"ah\-ler triples 
constructed in this section are all noncompact. What makes them relevant is 
the fact that some of them serve as universal building blocks for compact 
ge\-o\-des\-\hbox{ic\hs-}\hskip0ptgra\-di\-ent K\"ah\-ler triples. (See 
Theorem~\ref{first}.)

We begin with data $\,\sd,h,\nr\nh,\lr,\tm,\tp,a,Q,\pm\hs,\vt\mapsto\hr\,$ and 
$\,\hr\mapsto f\,$ consisting of
\begin{enumerate}
  \def\theenumi{{\rm\roman{enumi}}}
\item[{\rm(i)}] the real part $\hs\lr\hs$ of a Her\-mit\-i\-an fibre metric 
in a hol\-o\-mor\-phic complex vector bundle $\,\nr\hs$ of positive fibre 
dimension over a K\"ah\-ler manifold $\,(\sd,h)$,
\item[{\rm(ii)}] some objects $\,\tm,\tp,a,Q,\pm\hs,\vt\mapsto\hr\,$ and 
$\,\hr\mapsto f\,$ satisfying (\ref{pbd}) -- (\ref{dfr}).
\end{enumerate}
Letting $\,\pi:\nr\to\sd\,$ stand for the bundle projection, $\,\mathrm{D}\,$ 
for the Chern connection of $\,\lr\,$ (see Section~\ref{ck}), and $\,\hr\,$ 
both for the variable in (ii) and for the norm function 
$\,\nr\to[\hs0,\infty)$, we use the inverse mapping of $\,\vt\mapsto\hr$, 
cf.\ (\ref{sgn}), to
\begin{equation}\label{trt}
\mathrm{treat\ }\,\vt,Q\,\mathrm{\ and\ }\,f\,\mathrm{\ as\ functions\ 
}\,\nr\to\bbR\mathrm{,\ denoted\ here\ by\ }\hat\vt,\hat Q\,\mathrm{\ and\ 
}\,\hat f.
\end{equation}
Denoting by $\,\hat J\,$ (rather than $\,J$) the com\-plex-stru\-cture tensor 
of $\nr\nnh$, we define a K\"ah\-ler metric $\,\,\hg\,$ on $\,\nr\,$ by 
requiring the K\"ah\-ler forms $\,\hat\omega=\hg(\hat J\,\cdot\,,\,\cdot\,)\,$ 
and $\,\omega^h\nh=h(J\,\cdot\,,\,\cdot\,)$ to be related by 
$\,\hat\omega\,=\,\hs\pi\sk\omega^h\hs
+\,i\hs\partial\hskip1.7pt\cro\hskip-1.5pt\hat f\nh$, which amounts to
\begin{equation}\label{hoh}
\hg\,\,=\,\,\pi\sk h\,\,-\,\,(i\hs\partial\hskip1.7pt\cro\hskip-1.5pt\hat f)
(\hat J\,\cdot\,,\,\cdot\,)\hh.
\end{equation}
As $\,\hat\omega\,$ should be positive (Remark~\ref{posit}), it is necessary 
to assume here that
\begin{equation}\label{thd}
\pi\sk h\,\,-\,\,(i\hs\partial\hskip1.7pt\cro\hskip-1.5pt\hat f)
(\hat J\,\cdot\,,\,\cdot\,)\,\,
\mathrm{\ is\ positive}\hyp\mathrm{definite\ at\ every\ point\ of\ }\,\nr.
\end{equation}
The above construction uses the objects (i) -- (ii) with (\ref{thd}), and 
leads to what is shown below (Theorem~\ref{ehggk}) to be a 
ge\-o\-des\-\hbox{ic\hs-}\hskip0ptgra\-di\-ent K\"ah\-ler triple 
$\,(\nr\nh,\hg,\hat\vt)$. 

It is convenient, however, to provide the following equivalent, though less 
concise, description of $\,\hg\,$ and $\,\hat J\,$ restricted to the 
complement $\,\nr'\nh=\nr\nh\smallsetminus\sd$ of the zero section in 
$\,\nr\nnh$. It uses the complex di\-rect-sum decomposition 
\begin{equation}\label{tnp}
T\nnh\nr'\,=\,\hs\hat{\mathcal{V}}\hskip1.2pt\oplus\hs\hat{\mathcal{H}}^\mp\nnh
\oplus\hs\hat{\mathcal{H}}^\bullet\nh, 
\end{equation}
in which $\,\hat{\mathcal{H}}^\bullet$ is the horizontal distribution of 
$\,\mathrm{D}\,$ and 
$\,\hat{\mathcal{V}}\hskip1.2pt\oplus\hs\hat{\mathcal{H}}^\mp\nh
=\hs\mathrm{Ker}\hskip2.3ptd\pi\,$ equals the vertical distribution, with the 
summands $\,\hat{\mathcal{V}}\,$ and $\,\hat{\mathcal{H}}^\mp$ forming, on 
each punctured fibre $\,\nr\hskip-2.4pt_y\w\smallsetminus\{0\}$, the complex 
radial distribution (Remark~\ref{ddnsq}) and, respectively, its 
$\,\lr$-or\-thog\-o\-nal complement in 
$\,\nr\hskip-2.4pt_y\w\smallsetminus\{0\}$. (The word `complex' preceding 
(\ref{tnp}) is justified by Lemma~\ref{ddrsq}(i).) To describe $\,\hg\,$ 
and $\,\hat J$, we declare that the three summands of (\ref{tnp}) are 
$\,\hat J$-in\-var\-i\-ant and mutually $\,\hg$-or\-thog\-o\-nal, that 
$\,\hat J\,$ restricted to $\,\hat{\mathcal{V}}\,$ agrees, along each 
punctured fibre $\,\nr\hskip-2.4pt_y\w\smallsetminus\{0\}$, with its standard 
com\-plex-stru\-cture tensor of the complex vector space 
$\,\nr\hskip-2.4pt_y\w$, that the differential of $\,\pi\,$ at every 
$\,(y,\xi)\in\nr\hskip-2.4pt_y\w\smallsetminus\{0\}$, cf.\ Remark~\ref{tlspc},
maps $\,\mathcal{H}_{(y,\,\xi)}\w$ com\-plex-lin\-e\-ar\-ly onto $\,\tyb\,$ 
and, with the constant $\,a\in(0,\infty)\,$ and function $\,\hat\vt\,$ 
appearing in (i) and (\ref{trt}),
\begin{equation}\label{hgv}
\begin{array}{rl}
\mathrm{a)}&
\hat{\mathcal{V}}\nh,\,\hat{\mathcal{H}}^\mp\mathrm{\ and\ 
}\,\hat{\mathcal{H}}^\bullet\mathrm{\ in\ (\ref{tnp})\ are\ mutually\ 
}\,\hg\hyp\mathrm{or\-thog\-o\-nal,}\phantom{\displaystyle\frac1{2_{1_1}}}\\
\mathrm{b)}&
a^2\nh\hr^2\hg\hs=\,\hat Q\,\lr\,\mathrm{\ on\ }\,\hat{\mathcal{V}}\nh,
\hskip16pta\hh\hr^2\hg\hs=\,2\hs|\hs\hat \vt-\hn\tpm|\,\lr\,\mathrm{\ on\ 
}\,\hat{\mathcal{H}}^\pm\nh,\\
\mathrm{c)}&
\hg\nnh_x\w(\nh w_x\w,w_x\w\hskip-5pt'\hskip1.7pt)=h_y\w(\hn w,w\hh'\hh)-
\displaystyle{\frac{|\hh\hat\vt(x)-\hn\tpm|}{a\hr^2}}\,
\la\nh R_y^{\mathrm{D}}\nh(\hn w,J\hskip-2pt_y\w w\hh'\hh)\hh\xi,i\hh\xi\ra
\hskip8pt\mathrm{with}\hskip5pt\hr=|\hh\xi|\hh,
\end{array}
\end{equation}
at any $\,x=(y,\xi)\in\nr\hskip-2.4pt_y\w\smallsetminus\{0\}$, where 
$\,w,w\hh'$ are any two vectors in $\,\tyb$, and $\,w_x\w,w_x\w\hskip-5pt'$ 
denote their $\,\mathrm{D}\hh$-\hn hor\-i\-zon\-tal lifts to $\,x$. The 
vertical vector fields $\,\hat v,\hat u\,$ with
\begin{equation}\label{lvf}
\hat v_{(y,\,\xi)}\w\,=\,\mp\hh a\hh\xi\hh,\hskip16pt
\hat u_{(y,\,\xi)}\w\,=\,\mp\hh ai\hh\xi\hh,
\end{equation}
allow us to characterize the restrictions of $\,\hg\,$ and $\,\hat J\,$ to 
$\,\hat{\mathcal{V}}=\mathrm{Span}\hs(\hn\hat v,\hat u)\,$ by
\begin{equation}\label{ana}
\hg(\hat v,\hat v)\,=\,\hg(\hat u,\hat u)\,
=\,\hat Q\hh,\hskip12pt\hg(\hat v,\hat u)\,=\,0\hh,\hskip12pt\hat u\,
=\,\hat J\hskip-1.1pt\hat v\hh.
\end{equation}
Note that the symmetry of 
$\,\hg\nnh_x\w(\nh w_x\w,w_x\w\hskip-5pt'\hskip1.7pt)\,$ in 
$\,w_x\w,w_x\w\hskip-5pt'$ reflects (b) in Section~\ref{ck}.

Lemma~\ref{ddrsq} easily implies that the definition (\ref{hgv}) of $\,\hg\,$ 
is actually equivalent to (\ref{hoh}), while condition (\ref{thd}) is nothing 
else than positivity of the right-hand side in (\ref{hgv}.c) whenever 
$\,w=w\hh'\nh\ne0$.
\begin{theorem}\label{ehggk}{\medit 
For any data\/ {\rm(i)} -- {\rm(ii)} with\/ {\rm(\ref{thd})}, let us define\/ 
$\,\hg,\hat\vt\,$ by\/ {\rm(\ref{trt})} -- {\rm(\ref{hoh})}.
\begin{enumerate}
  \def\theenumi{{\rm\alph{enumi}}}
\item[{\rm(a)}] $(\nr\nh,g,\vt)\,$ is a ge\-o\-des\-ic-gra\-di\-ent K\"ah\-ler 
triple.
\item[{\rm(b)}] The fibres $\,\nr\hskip-2.4pt_y\w=\pi^{-\nnh1}(y)$, 
$y\in\sd$, \,are totally geodesic complex \hbox{sub\-man\-i\-folds of 
$\,(\nr\nh,g)$.}
\item[{\rm(c)}] The zero section\/ $\,\sd\subseteq\nr\,$ coincides with\/ 
$\,\sd^\pm\nnh$, \,the\/ $\,\tpm\hs$ level set of\/ $\,\vt$.
\item[{\rm(d)}] The\/ $\,\hg$-gra\-di\-ent\/ 
$\,\hat v=\hat\nabla\hskip-.5pt\hat\vt\,$ and\/ 
$\,\hat{\dv}=\hat\nabla\nh\hat v\,$ satisfy\/ {\rm(\ref{lvf})} -- 
{\rm(\ref{ana})} and the equality
\begin{equation}\label{gsw}
2\hh\hg\nnh_x\w(\hat{\dv}\nh_x\w w_x\w,w_x\w\hskip-5pt'\hskip1.7pt)\,=
\,\pm\,\frac{\hat Q(x)}{a\hr^2}\,
\la\nh R_y^{\mathrm{D}}\nh(\hn w,J\hskip-2pt_y\w w\hh'\hh)\hh\xi,i\hh\xi\ra\hh,
\hskip7pt\mathrm{where}\hskip5pt\hr=|\hh\xi|>0 \hh,
\end{equation}
the assumptions being the same as in\/ {\rm(\ref{hgv}.c)}.
\end{enumerate}
}
\end{theorem}
\begin{proof}The $\,\hat v$-di\-rec\-tion\-al derivatives of the norm squared 
$\,\hr^2$ and of $\,\hr\,$ are, obviously, $\,\mp2\hh a\hr^2$ and 
$\,\mp\hh a\hr$. As $\,d\vt/d\nh\hr=\mp\hh Q/(a\hr)\,$ in (\ref{sgn}), we see  using (\ref{trt}) and (\ref{ana}) that
\begin{enumerate}
  \def\theenumi{{\rm\alph{enumi}}}
\item[{\rm(e)}] $\hat Q=\hg(\hat v,\hat v)\,$ equals the 
$\,\hat v$-di\-rec\-tion\-al derivative 
$\,\hg(\hat v,\hat\nabla\hskip-.5pt\hat\vt)\,$ of $\,\hat\vt$.
\end{enumerate}
Furthermore, $\,\mathrm{D}\hh$-par\-al\-lel transports preserve the real 
fibre metric $\,\lr$. Therefore,
\begin{enumerate}
  \def\theenumi{{\rm\alph{enumi}}}
\item[{\rm(f)}] $\hr,\hat\vt\,$ and $\,\hat Q\,$ are constant along every 
$\,\mathrm{D}\hh$-hor\-i\-zon\-tal curve in $\,\nr\nnh$,
\end{enumerate}
due to (ii) and (\ref{trt}). The equality 
$\,\hat v=\hat\nabla\hskip-.5pt\hat\vt\,$ now follows: 
$\,\hat{\mathcal{V}}=\mathrm{Span}\hs(\hn\hat v,\hat u)$, and $\,\hat\vt\,$ 
is a function of the norm $\,\hr$, so that 
$\,\hat v-\hat\nabla\hskip-.5pt\hat\vt\,$ is $\,\hg$-or\-thog\-o\-nal to 
$\,\hat{\mathcal{H}}^\bullet\nh,\,\hat{\mathcal{H}}^\pm\nh,\hat u\,$ and 
$\,\hat v\,$  by (f), (\ref{hgv}.a), (\ref{ana}), and (e). Also, (\ref{lvf}) 
clearly gives hol\-o\-mor\-phic\-i\-ty of $\,\hat v$, while closedness and 
positivity of the form 
$\,\hg(\hat J\,\cdot\,,\,\cdot\,)=\hs\pi\sk\omega^h\nh
+i\hs\partial\hskip1.7pt\cro\hskip-1.5pt\hat f\,$ imply that $\,\hg\,$ is a 
K\"ah\-ler metric, and $\,\hat\vt\,$ has a geodesic $\,\hg$-gra\-di\-ent 
$\,\hat v$, its $\,\hg$-norm squared $\,\hat Q\,$ being a function of 
$\,\hat\vt$, cf.\ (\ref{trt}) and Lemma~\ref{ggqft}. We have thus proved (a). 
Next, Remark~\ref{tends} yields (c).

The $\,\pi$-pro\-ject\-a\-ble local sections of 
$\,\hat{\mathcal{H}}^\bullet$ are precisely the same as the 
$\,\mathrm{D}\hh$-hor\-i\-zon\-tal lifts of local vector fields tangent to 
$\,\sd$, and their local flows act as $\,\mathrm{D}\hh$-par\-al\-lel 
transports between the fibres. As the the sub\-man\-i\-fold metrics of the 
fibres are defined by (\ref{hgv}.a) -- (\ref{hgv}.b), this last action 
consists -- by (f) -- of isometries which, being linear, also preserve the 
vertical vector field $\,\hat v\,$ with (\ref{lvf}). Hence 
\begin{enumerate}
  \def\theenumi{{\rm\alph{enumi}}}
\item[{\rm(g)}] $\hat v\,$ commutes with all local 
$\,\mathrm{D}\hh$-hor\-i\-zon\-tal lifts $\,w$,
\end{enumerate}
and, at the same time, applying Remark~\ref{tglvs} to any such $\,w\ne0\,$ we 
obtain (b).

Finally, by (g) and Remark~\ref{dvwwp}, the left-hand side of (\ref{gsw}) 
equals the $\,\hat v$-di\-rec\-tion\-al derivative of the right-hand side in 
(\ref{hgv}.c). To evaluate the latter, note that only the factor 
$\,-\hh|\hh\hat\vt(x)-\hn\tpm|=\pm\hh(\hh\hat\vt(x)-\hn\tpm)\,$ in the second term 
needs to be differentiated, as the first term and the remaining factor of the 
second one are constant along $\,\hat v\,$ (due to constancy along 
$\,\hat v\,$ of $\,\xi/\nnh\hr=\xi/\hn|\hh\xi|$, obvious from (\ref{lvf})). 
Now (e) implies (\ref{gsw}), completing the proof.
\end{proof}
A {\medit special K\"ah\-ler-Ric\-ci potential\/} 
\cite{derdzinski-maschler-06} on a K\"ah\-ler manifold $\,(\mf\nh,g)\,$ is any 
nonconstant function $\,\vt:\mf\to\bbR\,$ such that $\,v=\navp\,$ is 
real-hol\-o\-mor\-phic, while, at points where $\,v\ne0$, all nonzero vectors 
orthogonal to $\,v\,$ and $\,J\nh v\,$ are eigen\-vec\-tors of both 
$\,\nabla\nh v$ and the Ric\-ci tensor, with 
$\,\nabla\nh v:T\nnh\mf\to T\nnh\mf\,$ as in (\ref{nwt}). We then call 
$\,(\mf\nh,g,\vt)\,$ an {\medit SKRP triple}. All SKRP triples $\,(\mf\nh,g,\vt)\,$ 
are ge\-o\-des\-\hbox{ic\hs-}\hskip0ptgra\-di\-ent K\"ah\-ler triples, due to 
their eas\-i\-ly-ver\-i\-fied property 
\cite[Remark~7.1]{derdzinski-maschler-03} that $\,v=\navp$, wherever nonzero, 
is an eigenvector of $\,\nabla\nh v$. Cf.\ (\ref{nvv}).

Compact SKRP triples $\,(\mf\nh,g,\vt)\,$ have been classified in 
\cite[Theorem 16.3]{derdzinski-maschler-06}. They are divided into Class 1, in 
which $\,\mf\,$ is the total space of a hol\-o\-mor\-phic $\,\bbCP^1$ bundle, 
and Class 2, with $\,\mf\,$ bi\-hol\-o\-mor\-phic to $\,\bbCP^m$ for 
$\,m=\dimc\nh\mf\nh$.
\begin{lemma}\label{cltwo}{\medit 
Up to isomorphisms, in the sense of Definition\/~{\rm\ref{ggktr}}, compact 
SKRP triples of Class 2 are the same as CP triples constructed using\/ 
{\rm(\ref{dta}.ii)} with\/ $\,\dimc\nh\ls=1\hh$.
}
\end{lemma}
\begin{proof}See \cite[Remark 6.2]{derdzinski-maschler-06}. (Note that the 
case $\,\dimc\nh\ls=m-1\,$ in (\ref{dta}.ii) obviously leads to the same 
isomorphism type.)
\end{proof}
In (i) above, $\,\dimc\nnh\sd\ge0$, which allows the possibility of a 
one\hs-\nh point base manifold $\,\sd=\{y\}$, so that, as a complex manifold, 
$\,\nr\,$ is a complex vector space, namely, the fibre 
$\,\nr\hskip-2.4pt_y\w$. According to 
\cite[pp.\ 85-86]{derdzinski-maschler-06}, under the standard identification 
(\ref{inc}) for $\,\vs=\nr\hskip-2.4pt_y\w$, both $\,\hg\,$ and $\,\hat\vt\,$ 
then can be extended to the projective space 
$\,\mathrm{P}(\hn\bbC\times\nnh\nr\hskip-2.4pt_y\w)$, giving rise to a Class 2 
SKRP triple $\,(\mf,\hg,\hat\vt)$, where 
$\,\mf=\mathrm{P}(\hn\bbC\times\nnh\nr\hskip-2.4pt_y\w)$.
\begin{lemma}\label{xplic}{\medit 
The SKRP triples\/ $\,(\mf,\gp,\vt)\,$ just mentioned, with\/ 
$\,\mf=\mathrm{P}(\hn\bbC\times\nnh\nr\hskip-2.4pt_y\w)$, represent all 
isomorphism types of compact SKRP triples of Class 2. Such types include all 
compact ge\-o\-des\-\hbox{ic\hs-}\hskip0ptgra\-di\-ent K\"ah\-ler triples of 
complex dimension $\,1$.
}
\end{lemma}
\begin{proof}For the first part, see 
\cite[Remark 6.2]{derdzinski-maschler-06}. The final clause is in turn 
immediate from Remark~\ref{cpone} and Lemma~\ref{cltwo}.
\end{proof}
\begin{remark}\label{fbstu}As a consequence of the second part of 
Remark~\ref{trivl}, for $\,(\nr\nh,g,\vt)\,$ as in Theorem~\ref{ehggk}, every 
fibre $\,\nr\hskip-2.4pt_y\w$ is the underlying complex manifold of a 
ge\-o\-des\-ic-gra\-di\-ent K\"ah\-ler triple, realizing a special case of 
Theorem~\ref{ehggk}: that of a one\hs-\nh point base manifold $\,\{y\}$. Its 
projective compactification 
$\,\mathrm{P}(\hn\bbC\times\nnh\nr\hskip-2.4pt_y\w)\,$ constitutes, for 
reasons mentioned above, the underlying complex manifold of an SKRP triple of 
Class 2. The resulting sub\-man\-i\-fold metric on the complement of 
$\,\nr\hskip-2.4pt_y\w$ in 
$\,\mathrm{P}(\hn\bbC\times\nnh\nr\hskip-2.4pt_y\w)\,$ (that is, on the 
projective hyper\-plane at infinity, identified via (\ref{inc}) with 
$\,\mathrm{P}\nnh\nr\hskip-2.4pt_y\w$) equals 
$\,2(\tp\hskip-2.2pt-\nh\tm)/\hn a$ times the Fu\-bi\-ni-Stu\-dy metric 
associated -- as in Remark~\ref{fbstm} -- with $\,\lr$.

Namely, let $\,\xi,\eta\in\nr\hskip-2.4pt_y\w$ have $\,\la\xi,\xi\ra=1\,$ and 
$\,\la\xi,\eta\ra=\la i\xi,\eta\ra=0$. The curve $\,t\mapsto t\eta\,$ of 
vectors $\,t\eta\,$ tangent to $\,\nr\hskip-2.4pt_y\w$ at the points 
$\,t\hh\xi$, satisfies, in view of 
(\ref{hgv}.b) and Remark~\ref{tends}, the limit relation 
$\,\hg\nnh_{(y,\,t\hh\xi)}\w(t\eta,t\eta)
\to2(\tp\hskip-2.2pt-\nh\tm)\la\eta,\eta\ra/\hn a\,$ as $\,t\to\infty$. At the 
same time, $\,t\hh\xi\,$ (or, the tangent vector $\,t\eta$) tends, as 
$\,t\to\infty$, to the point $\,\bbC(0,\xi)\,$ of 
$\,\mathrm{P}(\hn\bbC\times\nnh\nr\hskip-2.4pt_y\w)
\smallsetminus\nr\hskip-2.4pt_y\w$, 
identified with $\,\bbC\xi\in\mathrm{P}\nnh\nr\hskip-2.4pt_y\w$ or, 
respectively, to the vector tangent to $\,\mathrm{P}\nnh\nr\hskip-2.4pt_y\w$ 
at $\,\bbC\xi\,$ which is the image of $\,\eta\,$ under
\begin{enumerate}
  \def\theenumi{{\rm\alph{enumi}}}
\item[{\rm($*$)}] the differential of the projection $\,\xi\mapsto\bbC\xi\,$ 
restricted to the unit sphere of $\,\lr$.
\end{enumerate}
The claim about the tangent vectors, which clearly implies our assertion, can 
be justified as follows. The vector $\,t\eta\,$ equals $\,x_s\w(t,0)\,$ 
(notation preceding Remark~\ref{jcobi}) with 
$\,x(t,s)=t\hh(\xi+s\eta)\in\nr\hskip-2.4pt_y\w$, so that 
$\,|\hs x(t,s)|^2\nh=t^2(1+|s\eta|^2)\,$ and, setting 
$\,\zeta(t,s)
=[1+t^2(1+|s\eta|^2)]^{-\nnh1\nh/2}(1,x(t,s))
\in\bbC\times\nnh\nr\hskip-2.4pt_y\w$, we get $\,|\hh\zeta(t,s)|=1\,$ for the 
di\-rect-sum Euclidean norm. Identifying $\,x(t,s)\,$ with 
$\,\bbC(1,x(t,s))=\bbC\hh\zeta(t,s)
\in\mathrm{P}(\hn\bbC\times\nnh\nr\hskip-2.4pt_y\w)$, we see that $\,t\eta$, 
treated as tangent to 
$\,\mathrm{P}(\hn\bbC\times\nnh\nr\hskip-2.4pt_y\w)\,$ at $\,\bbC(1,t\hh\xi)$, 
is the image, under the analog of ($*$) for 
$\,\bbC\times\nnh\nr\hskip-2.4pt_y\w$, of the vector 
$\,\zeta_s\w(t,0)=(1+t^2)^{-\nnh1\nh/2}(0,x_s\w(t,0))
=(0,t\hh(1+t^2)^{-\nnh1\nh/2}\eta)$ tangent to the unit sphere of 
$\,\bbC\times\nnh\nr\hskip-2.4pt_y\w$ at the point 
$\,\zeta(t,0)=(1+t^2)^{-\nnh1\nh/2}(1,t\hh\xi)\,$ and having the required 
limit $\,(0,\eta)\,$ as $\,t\to\infty$, which we identify with $\,\eta$.
\end{remark}
\begin{remark}\label{unnec}The construction summarized in Theorem~\ref{ehggk} 
has an obvious generalization, arising when, in (\ref{pbd}), 
$\,\vt\mapsto Q\,$ is assumed to be only defined on the half-open interval 
$\,[\hh\tm,\tp]\smallsetminus\{\tmp\}$, and $\,dQ\hh/\nh d\vt=\mp2a\,$ at 
$\,\vt=\tpm$ just for one fixed sign $\,\pm\hs$. Our discussion focuses on a 
narrower case since this is the case relevant to the study of {\medit 
compact\/} ge\-o\-des\-ic-gra\-di\-ent K\"ah\-ler triples.
\end{remark}

\section{Local properties\done}\label{lp}
\setcounter{equation}{0}
Throughout this section $\,(\mf\nh,g,\vt)\,$ is a fixed 
ge\-o\-des\-\hbox{ic\hs-}\hskip0ptgra\-di\-ent K\"ah\-ler triple (Definition~\ref{ggktr}). We 
use the symbols
\begin{equation}\label{sym}
J,\,\,v,\,\,u,\,\,\mf'\nh,\,\,\psi,\,\,Q,\,\,\mathcal{V},\,\,\mathcal{V}^\perp\nh,
\,\,\dv,\,\,\du
\end{equation}
for the com\-plex-struc\-ture tensor $\,J:T\nnh\mf\to T\nnh\mf\,$ of the 
underlying complex manifold $\,\mf\nh$, the gradient $\,v=\navp$, its 
$\,J$-im\-age $\,u=J\nh v$, the open set $\,\mf'$ where $\,v\ne0$, the 
function $\,\psi\,$ on $\,\mf'$ with (\ref{nvv}), the function 
$\,Q=g(\hn v,v)\,$ on $\,\mf\nh$, the distribution 
$\,\mathcal{V}=\mathrm{Span}\hs(\hn v,u)\,$ on $\,\mf'\nnh$, its orthogonal 
complement, as well as the en\-do\-mor\-phisms $\,\dv=\nabla\nh v\,$ and 
$\,\du=\nabla\nh u\,$ of $\,T\nnh\mf\nh$, cf.\ (\ref{nwt}). Under the above 
hypotheses,
\begin{equation}\label{loc}
\begin{array}{rl}
\mathrm{a)}&v,u\,\mathrm{\ are\ both\ hol\-o\-mor\-phic,}\,|v|=|u|
=Q^{\hs1\nh/2}\nh,\mathrm{\ and\ }\,\du\hs=\hs J\dv\hs=\hs\dv\nh J\hh,\\
\mathrm{b)}&u=J\nh v\,\mathrm{\ is\ a\ Kil\-ling\ field\ commuting\ 
with\ }\,v\hh,\mathrm{\ and\ orthogonal\ to\ }\,v\hh,\\
\mathrm{c)}&\nabla\hskip-3pt_{w}\w\du\hs=\hs R(\hn u,w)\,\mathrm{\ and\ 
}\,\nabla\hskip-3pt_{w}\w\dv\hs=\hs-\nh J[R(\hn u,w)]\,\mathrm{\ for\ any\ 
vector\ field\ }\,w\hh,\\
\mathrm{d)}&\dv\,\mathrm{\ is\ self}\hyp\mathrm{ad\-joint\ and\ 
}\,J,\du\,\mathrm{\ are\ skew}\hyp\mathrm{ad\-joint\ at\ every\ point\ of\ 
}\,\mf\hh,\\
\mathrm{e)}&
g([w,w\hh'],u)=-\nh2\hh g(\du w,w\hh'\hh)\,\mathrm{\ for\ any\ local\ 
sections\ }\,w,w\hh'\mathrm{\ of\ }\,\mathcal{V}^\perp\hh,\\
\mathrm{f)}&\nabla\hskip-3pt_{v}\w v=\psi\hskip.4ptv
=-\nabla\hskip-3pt_u\w u\,\mathrm{\ 
and\ }\,\nabla\hskip-3pt_u\w v=\nabla\hskip-3pt_{v}\w u
=\psi\hskip.4ptu\,\mathrm{\ everywhere\ in\ }\,\mf',\\
\mathrm{g)}&Q\,\mathrm{\ is,\ locally\ in\ }\,\mf'\nh,\mathrm{\ 
a\ function\ of\ }\,\vt\hh,\mathrm{\ and\ }\,2\psi=\hs d\hh Q/d\vt\hh,\\
\mathrm{h)}&J,\hs\dv,\hs\du\,\mathrm{\ and\ the\ local\ flows\ of\ 
}\,u\,\mathrm{\ and\ }\,v\,\mathrm{\ leave\ }\,\mathcal{V}\,\mathrm{\ and\ 
}\,\mathcal{V}^\perp\mathrm{\ invariant.} 
\end{array}
\end{equation}
In (\ref{loc}.c), $\,R\,$ denotes the curvature tensor of $\,g$, and the 
notation of (\ref{nwt}) is used.

In fact, hol\-o\-mor\-phic\-i\-ty of $\,v\,$ (cf.\ Definition~\ref{ggktr}) 
combined with (\ref{ajs}) -- (\ref{hol}) gives (\ref{loc}.a), $\,u\,$ being 
hol\-o\-mor\-phic due to (\ref{hol}), as $\,\du\hs=\hs J\dv\hs=\hs\dv\nh J\,$ 
commutes with $\,J$. Next, (\ref{loc}.b) follows from (\ref{kil}) and the 
Lie-brack\-et equality 
$\,[u,v]=\nabla\hskip-3pt_u\w v-\nabla\hskip-3pt_{v}\w u=\dv\nh u-\du v=\dv\nh u-\dv\nh J\nh v=0$, 
obvious in view of (\ref{loc}.a), while (\ref{loc}.c) (or, (\ref{loc}.d)) is a 
direct consequence of (\ref{scd}) and (\ref{loc}.a) or, respectively, of 
of (\ref{loc}.b) combined with the fact that $\,v\,$ is a gradient. We now 
obtain (\ref{loc}.e) from (\ref{loc}.d), noting that 
$\,g(\nabla\hskip-3pt_{w}\w w\hh'\nh,u)
=-\hh g(\hn w\hh'\nh,\nabla\hskip-3pt_{w}\w u)
=-\hh g(\hn w\hh'\nh,\du w)$. On the other hand, (\ref{loc}.b), (\ref{loc}.a) and 
(\ref{nvv}) yield 
$\,\nabla\hskip-3pt_u\w v=\nabla\hskip-3pt_{v}\w u
=\nabla\hskip-3pt_{v}\w(J\nh v)=J\nabla\hskip-3pt_{v}\w v
=\psi\hskip.4ptJ\nh v=\psi\hskip.4ptu\,$ and so 
$\,\nabla\hskip-3pt_u\w u=\nabla\hskip-3pt_u\w(J\nh v)=J\nabla\hskip-3pt_u\w v
=\psi\hskip.4ptJu=-\hh\psi\hskip.4ptv$, establishing (\ref{loc}.f), while 
Lemma~\ref{ggqft}, (\ref{tnd}) and (\ref{loc}.f) imply (\ref{loc}.g). That 
$\,J,\dv,\du\,$ all leave $\,\mathcal{V}=\mathrm{Span}\hs(\hn v,u)\,$ 
invariant is 
clear as $\,J\nh v=u$ and $\,Ju=-\hh v\,$ while, by (\ref{loc}.f), 
$\,\dv\nh v,\dv\nh u,\du v,\du u\,$ are sections of $\,\mathcal{V}\nnh$. The 
same conclusion for $\,\mathcal{V}^\perp$ is now immediate from (\ref{loc}.d). 
By (\ref{loc}.b), the local flows of $\,v\,$ and $\,u\,$ preserve $\,v,u\,$ 
and $\,\mathcal{V}=\mathrm{Span}\hs(\hn v,u)$. The $\,u$-in\-var\-i\-ance of 
$\,\mathcal{V}^\perp$ now follows from (\ref{loc}.b). Finally, let $\,w\,$ be 
a section of $\,\mathcal{V}^\perp\nnh$. Writing $\,\lr\,$ for $\,g$, we get 
$\,\la[v,w],v\ra=\la\nabla\hskip-3pt_{v}\w w-\nabla\hskip-3pt_{w}\w v,v\ra
=-\hh\la w,\nabla\hskip-3pt_{v}\w v\ra-\la\dv\nh w,v\ra=-\la\dv\nh w,v\ra
=-\la w,\dv\nh v\ra=0$, cf.\ (\ref{loc}.d) and (\ref{loc}.f). Similarly, 
$\,\la[v,w],u\ra=\la\nabla\hskip-3pt_{v}\w w-\nabla\hskip-3pt_{w}\w v,u\ra
=-\hh\la w,\nabla\hskip-3pt_{v}\w u\ra-\la\dv\nh w,u\ra=-\la\dv\nh w,u\ra
=-\la w,\dv\nh u\ra=0$. Thus, $\,[v,w]\,$ is a section of 
$\,\mathcal{V}^\perp$ as well. In view of Remark~\ref{prjct}, this completes 
the proof of (\ref{loc}.h). For easy reference, note that, by (\ref{loc}.a) -- 
(\ref{loc}.b),
\begin{equation}\label{gvv}
g(v,v)\,=\,g(u,u)\,
=\,Q\hh,\hskip12ptg(v,u)\,=\,0\hh,\hskip12ptu\,
=\,J\hskip-1.1ptv\hh.
\end{equation}
\begin{lemma}\label{dvgww}{\medit 
Under the assumptions preceding\/ {\rm(\ref{loc})}, on\/ $\,\mf'\nh$,
\begin{enumerate}
  \def\theenumi{{\rm\alph{enumi}}}
\item[{\rm(a)}] the distribution\/ $\,\mathcal{V}=\mathrm{Span}\hs(\hn v,u)\,$ 
is in\-te\-gra\-ble and has totally geodesic leaves,
\item[{\rm(b)}] a local section of\/ $\hs\mathcal{V}^\perp\nnh$ is 
pro\-ject\-able along\/ $\mathcal{V}\hs$ if and only if it commutes with\/ 
$\hs u\hs$ and\/ $\hs v$,
\item[{\rm(c)}] if local sections\/ $\,w\,$ and\/ $\,w\hh'$ of\/ 
$\,\mathcal{V}^\perp\,$ commute with\/ $\,u\,$ and\/ $\,v$, then
\begin{equation}\label{dvg}
\begin{array}{rlrl}
\mathrm{i)}&d_v\w[\hs g(\hn w,w\hh'\hh)]\,=\,2\hh g(\dv\nh w,w\hh'\hh)\hh,
&\mathrm{ii)}&d_v\w[\hs g(\dv\nh w,w\hh'\hh)]\,=\,2\psi\hh g(\dv\nh w,w\hh'\hh)\hh,\\
\mathrm{iii)}&d_v\w[Q^{-\nnh1}g(\dv\nh w,w\hh'\hh)]\,=\,0\hh,&
\end{array}
\end{equation}
\item[{\rm(d)}] $d_v\w Q\nnh=\nh 2\psi Q\hh\,$ and\/ 
$\hs d_u\w[\hs g(\hn w,w\hh'\hh)]\nh=\nh d_u\w[\hs g(\dv\nh w,w\hh'\hh)]\nh
=\nh d_u\w Q\nnh=0\,$ for any\/ $\,w,w\hh'$ as in\/ {\rm(c)},
\item[{\rm(e)}] $[\nabla\hskip-3pt_{v}\w\dv]\hh w\,=\,2(\psi\hs-\hs\dv)\dv\hh w\hs\,$ 
whenever\/ $\,w\,$ is a local section of\/ $\,\mathcal{V}^\perp\nh$.
\end{enumerate}
}
\end{lemma}
\begin{proof}Assertions (a) -- (b) are obvious from (\ref{loc}.b) and, 
respectively, Remark~\ref{liebr} combined with (\ref{loc}.h). Next, let  
$\,\lie\hskip-.5pt_v\w w=\lie\hskip-.5pt_v\w w\hh'\nh=\lie\hskip-.5pt_u\w w
=\lie\hskip-.5pt_v\w w\hh'\nh=0$. Since $\,\lie\hskip-.5pt_v\w$ and 
$\,\lie\hskip-.5pt_u\w$ act on functions as $\,d_v\w$ and $\,d_u\w$, 
(\ref{lvg}) implies (\ref{dvg}.i), and $\,d_u\w[\hs g(\hn w,w\hh'\hh)]=0\,$ as 
$\,\lie\hskip-.5pt_u\w g=0\,$ by (\ref{loc}.b). For similar reasons, 
$\,d_u\w[\hs g(\dv\nh w,w\hh'\hh)]
=\lie\hskip-.5pt_u\w[\hs g(\dv\nh w,w\hh'\hh)]=0$. (Namely, 
(\ref{loc}.c) gives $\,\nabla\hskip-3pt_u\w\dv=0$, so that (\ref{loc}.a) and 
(\ref{lie}), with $\,u,\dv\,$ rather than $\,v,\bl$, yield 
$\,\lie\hskip-.5pt_u\w\dv=0$.) On the other hand, by (\ref{gvv}), 
$\,g(\hn v,v)=Q$. Now (\ref{tnd}), (\ref{loc}.f) and (\ref{loc}.b) imply that 
$\,d_u\w\vt=d_u\w Q=0\,$ and $\,d_v\w Q=2\psi Q$, establishing (d).

Using (\ref{loc}.a) we get $\,g(\dv\nh w,w\hh'\hh)=g(J\dv\nh w,J\nh w\hh'\hh)
=g(\du w,J\nh w\hh'\hh)\,$ which, 
by (\ref{loc}.e), is nothing else than $\,-\hh g([w,J\nh w\hh'],u)/2$. Hence 
$\,2\hs d_v\w[\hs g(\dv\nh w,w\hh'\hh)]=2\lie\hskip-.5pt_v\w[\hs g(\dv\nh w,w\hh'\hh)]
=-\lie\hskip-.5pt_v\w[\hs g(\hn u,[w,J\nh w\hh']))]
=-\hh[\nh\lie\hskip-.5pt_v\w g](\hn u,[w,J\nh w\hh']))$. (Our assumption that 
$\,\lie\hskip-.5pt_v\w w=\lie\hskip-.5pt_v\w w\hh'\nh=0$ gives 
$\,\lie\hskip-.5pt_v\w(J\nh w\hh'\hh)=0$, as $\,v\,$ is hol\-o\-mor\-phic, 
which in turn yields $\,\lie\hskip-.5pt_v\w[w,J\nh w\hh']=0$, while 
$\,\lie\hskip-.5pt_v\w u=0$, cf.\ (\ref{loc}.b).) From (\ref{lvg}), 
(\ref{loc}.f) and (\ref{loc}.a) we now obtain 
$\,d_v\w[\hs g(\dv\nh w,w\hh'\hh)]=-\hh[\nh\lie\hskip-.5pt_v\w g](\hn u,[w,J\nh w\hh']))/2
=-\hh g(\dv\nh u,[w,J\nh w\hh'])=-\nh2g(\psi u,[w,J\nh w\hh'])
=2\psi\hh g(\du w,\nh J\nh w\hh'\hh)\nh=\nh-\nnh2\psi\hh g(J\nnh\du w,w\hh'\hh)\nh
=2\psi\hh g(\dv\nh w,w\hh'\hh)$, that is, (\ref{dvg}.ii), which, since 
$\,d_v\w Q=2\psi Q\,$ by (d), also proves (\ref{dvg}.iii).

Finally, (\ref{loc}.h) and the equality 
$\,\nabla\hskip-3pt_{v}\w\dv\hs=\hs-\nh J[R(\hn u,v)]$, cf.\ (\ref{loc}.c), combined with 
(a), imply that $\,\nabla\hskip-3pt_{v}\w\dv-(2\psi\hs-\hs\dv)\dv\,$ leaves 
$\,\mathcal{V}^\perp$ invariant. To obtain (e), it now suffices to show that 
$\,[\nabla\hskip-3pt_{v}\w\dv]\hh w-(2\psi\hs-\hs\dv)\dv\hh w\,$ is orthogonal to 
$\,w\hh'$ for any local sections $\,w,w\hh'$ of $\,\mathcal{V}^\perp\nh$. We are 
free to assume here that $\,w=w\hh'$ (due to self-ad\-joint\-ness of 
$\,\dv=\nabla\nh v$) and that $\,w\,$ commutes with $\,u\,$ and $\,v\,$ (see 
(b)). Differentiation by parts gives, by (\ref{dvg}.iii) and (\ref{loc}.d), 
$\,g([\nabla\hskip-3pt_{v}\w\dv]\hh w,w)=d_v\w[\hs g(\dv\nh w,w)]
-g(\dv\nabla\hskip-3pt_{v}\w w,w)-g(\dv\nh w,\nabla\hskip-3pt_{v}\w w)
=2\psi\hh g(\dv\nh w,w)-2g(\dv\nh w,\dv\nh w)$, as required, with 
$\,\nabla\hskip-3pt_{v}\w w=\dv\nh w\,$ since $\,[v,w]=0$.
\end{proof}

\section{Horizontal Ja\-co\-bi fields\done}\label{hj}
\setcounter{equation}{0}
In addition to using the assumptions and notations of Section~\ref{lp}, we now 
let $\,\ic$ stand for the underlying 
\hbox{one\hh-}\hskip0ptdi\-men\-sion\-al manifold of a fixed maximal integral 
curve of $\,v\,$ in $\,\mf'\nh$. We restrict the objects in (\ref{sym}) to 
$\,\ic\,$ without changing the notation, and select a unit-speed 
parametrization $\,t\mapsto x(t)\,$ of the geodesic $\,\ic$ such that
\begin{equation}\label{dxe}
\dot x\,=\,v/|v|\,=\,Q^{-\nnh1\nh/2}v\hskip10pt\mathrm{along}\hskip6pt\ic,
\hskip6pt\mathrm{where}\hskip6ptv=\navp\hh.
\end{equation}
As an obvious consequence of (\ref{dxe}), (\ref{dvt}) and Lemma~\ref{dvgww}(d),
\begin{equation}\label{dtq}
\dot\vt\,=\,Q^{\hs1\nh/2},\hskip20pt\dot Q\,=\,2\psi Q^{\hs1\nh/2},
\hskip20pt\mathrm{with}\hskip8pt(\hskip2.3pt)\nnh\dot{\phantom o}\nh=\,d/dt\,
=\,d_{\dot x}\w\hh.
\end{equation}
Any constant $\,\zx\in[\bbR\smallsetminus\vt(\ic)]\cup\{\infty\}$, where 
$\,\vt(\ic)\,$ is the range of $\,\vt\,$ on $\,\ic$, gives rise to the 
function $\,\lambda_\zx\w:\ic\to\bbR\,$ defined by
\begin{equation}\label{lgm}
\lambda_\zx\w\,=\,\hs Q/[2(\vt-\zx)]\hh,
\end{equation}
the convention being that $\,\lambda_\zx\w$ is identically zero when 
$\,\zx=\infty$. We denote by $\,\mathcal{W}\,$ the set of all 
$\,\mathcal{V}^\perp\nh$-val\-ued vector fields 
$\,t\mapsto w(t)\in\mathcal{V}^\perp_{\hskip-1ptx(t)}$ along $\,\ic\,$ 
satisfying the equation
\begin{equation}\label{nxw}
\nabla\hskip-3pt_{\dot x}\w w\,=\,Q^{-\nnh1\nh/2}\dv\nh w\hh.
\end{equation}
Of particular interest to us are $\,\zx\,$ such that
\begin{equation}\label{wze}
\begin{array}{rl}
\mathrm{a)}&
\zx\in[\bbR\smallsetminus\vt(\ic)]\cup\{\infty\}\,\,\mathrm{\ and\ 
}\,\,\mathcal{W}[\zx]\ne\{0\}\hh,\mathrm{\ \ where}\\
\mathrm{b)}&
\mathcal{W}[\zx]\,=\,\{w\in\mathcal{W}:\dv\nh w=\lambda_\zx\w\nh w\}\hh.
\end{array}
\end{equation}
About pro\-ject\-abil\-i\-ty along $\,\mathcal{V}\,$ in (i) below, see
Remark~\ref{pralg} and Lemma~\ref{dvgww}(a).
\begin{theorem}\label{jacob}{\medit 
Under the above hypotheses, the following conclusions hold.
\begin{enumerate}
  \def\theenumi{{\rm\roman{enumi}}}
\item[{\rm(i)}] $\mathcal{V}^\perp\nh$-val\-ued solutions\/ $\,w\,$ to\/ 
{\rm(\ref{nxw})} are nothing else than restrictions to\/ $\,\ic\hs$ of those 
local sections of\/ $\hs\mathcal{V}^\perp\nnh$ with domains containing\/ 
$\,\ic\hs$ which are pro\-ject\-able along\/ $\,\mathcal{V}\nh$.
\item[{\rm(ii)}] All\/ $\,w\,$ as in\/ {\rm(i}), that is, all elements of\/ 
$\,\mathcal{W}\nh$, are Ja\-co\-bi fields along\/ $\,\ic\nh$.
\item[{\rm(iii)}] Every vector in\ $\,\mathcal{V}^\perp_{\hskip-1ptx(t)}$ 
equals\/ $\,w(t)\,$ for some unique\/ $\,w\in\mathcal{W}\nh$.
\item[{\rm(iv)}] $\mathcal{W}\,$ is a complex vector space of complex 
dimension\/ $\,\dimc\nh\mf-1$, and the direct sum of all\/ 
$\,\hs\mathcal{W}[\zx]\,$ for\/ $\hs\zx\,$ in\/ {\rm(\ref{wze}.a)}, with\/ 
$\,w\mapsto J\nh w\,$ serving as the multiplication by\/ $\,i\in\bbC$.
\item[{\rm(v)}] A function\/ $\,t\mapsto\lambda(t)\,$ on the parameter 
interval of\/ $\,t\mapsto x(t)\,$ satisfies the equation\/ 
$\,d\lambda/dt=2(\psi-\lambda)\lambda\hh Q^{-\nnh1\nh/2}$, with\/ 
$\,\psi,Q\,$ evaluated at\/ $\,x(t)$, if and only if\/ 
$\,\lambda(t)=\lambda_\zx\w(x(t))$, cf.\ {\rm(\ref{lgm})}, for some\/ 
$\,\zx\in[\bbR\smallsetminus\vt(\ic)]\cup\{\infty\}\,$ and all\/ $\,t$.
\item[{\rm(vi)}] At any\/ $\,x=x(t)\in\ic\nh$, the eigen\-values of 
$\,\dv\nnh_x\w:\mathcal{V}^\perp_{\hskip-1ptx}\nh
\to\mathcal{V}^\perp_{\hskip-1ptx}$, cf.\ {\rm(\ref{loc}.h)}, are precisely 
the values\/ $\,\lambda_\zx\w(x)\,$ for all\/ 
$\,\zx\,$ in\/ {\rm(\ref{wze}.a)}. The eigen\-space of\/ 
$\,\dv\nnh_x\w:\mathcal{V}^\perp_{\hskip-1ptx}\nh
\to\mathcal{V}^\perp_{\hskip-1ptx}$ corresponding to\/ 
$\,\lambda_\zx\w(x)\,$ is\/ $\,\{w(t):w\in\mathcal{W}[\zx]\}$.
\item[{\rm(vii)}] $R(\hn w,u)\hh u\nnh=\nh\nnh R(\hn w,v)\hh v\nnh
=\nnh(\psi-\dv)\dv\hh w\nnh=\nh\nnh R(\hn v,u)\hh J\nh w\nh/2\,$ on\/ $\hs\mf'\nh$ 
for sections\/ $\hs w\hh$ of\/ $\,\mathcal{V}^\perp\nnh$.
\item[{\rm(viii)}] If\/ $\,\vt(\ic)=(\tm,\tp)\,$ is bounded, then\/ 
$\,Q/(\vt-\hn\tp)\le\,2\dv\hh\le Q/(\vt-\hn\tm)\,$ on\/ 
$\,\mathcal{V}^\perp\nnh$.
\end{enumerate}
}
\end{theorem}
\begin{proof}Any $\,w\,$ as in the second line of (i), restricted to 
$\,\ic\nh$, becomes both a Ja\-co\-bi field (by Lemmas~\ref{jcbfl} 
and~\ref{dvgww}(b)) and a $\,\mathcal{V}^\perp\nh$-val\-ued solution to 
(\ref{nxw}) (since $\,\dv=\nabla\nh v$, so that (\ref{dxe}) and 
Lemma~\ref{dvgww}(b) give 
$\,\nabla\hskip-3pt_{\dot x}\w w=Q^{-\nnh1\nh/2}\nabla\hskip-3pt_{v}\w w
=Q^{-\nnh1\nh/2}\nabla\hskip-3pt_{w}\w v=Q^{-\nnh1\nh/2}\dv\nh w$). With 
$\,\ic\,$ replaced by suitable shorter sub\-geodesics covering all points of 
$\,\ic\nh$, the inclusion just established between the two vector spaces 
appearing in (i) is actually an equality: in either class, the vector field in 
question is uniquely determined by its initial value at any given point 
$\,x\in\ic\nh$. This proves (i) -- (ii) as well as (iii) -- (iv), the latter 
in view of the fact that $\,J\dv=\dv\nh J$, cf.\ (\ref{loc}.a).

For a $\,C^1$ function $\,\lambda\,$ defined on the parameter interval of 
$\,t\mapsto x(t)$, one has
\begin{equation}\label{dlt}
\dot\lambda\,=\,2(\psi-\lambda)\lambda\hh Q^{-\nnh1\nh/2}\hskip12pt
\mathrm{with}\hskip7pt(\hskip2.3pt)\dot{\phantom o}\nh=d/dt
\end{equation}
if and only if either $\,\lambda=0\,$ identically, or $\,\lambda\ne0\,$ 
everywhere and the function $\,\zx$ characterized by 
$\,2\zx=2\vt-Q/\lambda\,$ is constant. (In fact, the ei\-ther-or claim about 
vanishing of $\,\lambda\,$ is due to uniqueness of solutions of 
in\-i\-tial-val\-ue problems, while (\ref{dtq}) yields 
$\,2\dot\zx
=Q\lambda^{-\nnh2}[\dot\lambda-2(\psi-\lambda)\lambda\hh Q^{-\nnh1\nh/2}\hh]$.) 
Now (v) easily follows, all nonzero initial conditions for (\ref{dlt}) at 
fixed $\,t\,$ being realized by suitably chosen constants 
$\,\zx\in\bbR\smallsetminus\vt(\ic)\,$ (and $\,\lambda=0\,$ satisfying (v) 
with $\,\zx=\infty$). 

On the other hand, from (\ref{dxe}) and Lemma~\ref{dvgww}(e), 
\begin{equation}\label{nds}
[\nabla\hskip-3pt_{\dot x}\w\dv]\hh w\,
=\,2Q^{-\nnh1\nh/2}(\psi\hs-\hs\dv)\dv\hh w\hh,
\hskip9pt\mathrm{if\ }\,w\,\mathrm{\ is\ a\ 
}\,\mathcal{V}^\perp\nh\hyp\mathrm{val\-ued\ vector\ field\ along\ }\,\ic.
\end{equation}
Next, we fix $\,x=x(t)\in\ic\,$ and express any prescribed 
ei\-gen\-val\-ue-ei\-gen\-vec\-tor pair for 
$\,\dv\nnh_x\w:\mathcal{V}^\perp_{\hskip-1ptx}\nh
\to\mathcal{V}^\perp_{\hskip-1ptx}$ as $\,\lambda_\zx\w(x)\,$ and $\,w(t)$, 
with some unique $\,\zx\in[\bbR\smallsetminus\vt(\ic)]\cup\{\infty\}$ and 
$\,w\in\mathcal{W}\nh$. By (v), $\,\lambda=\lambda_\zx\w$ satisfies 
(\ref{dlt}), so that, in view of (\ref{nxw}) and (\ref{nds}), the vector field 
$\,\hat w=\dv\nh w-\lambda w\,$ is a solution of the linear homogeneous 
differential equation 
$\,\nabla\hskip-3pt_{\dot x}\w \hat w=Q^{-\nnh1\nh/2}(2\psi-2\lambda-\dv)\hh\hat w$. 
Since $\,\hat w\,$ vanishes at $\,x=x(t)$, it must vanish identically, which 
establishes (vi).

Now let $\,w\in\mathcal{W}\nh$. As $\,\dot Q=2\psi Q^{\hs1\nh/2}$ (see 
the lines following (\ref{dlt})), the Ja\-co\-bi equation and (\ref{nxw}) 
give, by (ii) and (\ref{nds}), 
$\,R(\hn w,\dot x)\hh\dot x=\nabla\hskip-3pt_{\dot x}\w
\nabla\hskip-3pt_{\dot x}\w w
=\nabla\hskip-3pt_{\dot x}\w[Q^{-\nnh1\nh/2}\dv\nh w]
=Q^{-\nnh1}(\psi\hs-\hs\dv)\dv\hh w$, 
that is, $\,R(\hn w,v)\hh v=(\psi\hs-\hs\dv)\dv\hh w$, the second equality in 
(vii). Also, Lemma~\ref{dvgww}(e), (\ref{loc}.c) and (\ref{rcm}) yield 
$\,2(\psi\hs-\hs\dv)\dv\hh w=[\nabla\hskip-3pt_{v}\w\dv]\hh w
=-\nh J[R(\hn u,v)\hh w]=-\hh R(\hn u,v)\hh J\nh w=R(\hn v,u)\hh J\nh w=R(\hn v,J\nh v)\hh J\nh w$, the last 
equality in (vii). Combining the two relations, and repeatedly using 
(\ref{rcm}), we get $\,2R(\hn w,v)\hh v=R(\hn v,J\nh v)\hh J\nh w$, that is, 
$\,R(\hn w,v)\hh v=R(\hn v,w)\hh v+R(\hn v,J\nh v)\hh J\nh w=R(J\nh v,J\nh w)\hh v+R(\hn v,J\nh v)\hh J\nh w$. Thus, 
from the Bian\-chi identity, 
$\,R(\hn w,v)\hh v=R(\hn v,J\nh w)\hh J\nh v=R(J\nh v,JJ\nh w)\hh J\nh v=R(\hn w,u)\hh u$, which proves (vii). 
Finally, (viii) is an easy consequence of (vi) and (\ref{loc}.d).
\end{proof}

\section{Consequences of compactness\done}\label{cc}
\setcounter{equation}{0}
Let $\,(\mf\nh,g,\vt)\,$ be a fixed ge\-o\-des\-\hbox{ic\hs-}\hskip0ptgra\-di\-ent K\"ah\-ler 
triple (Definition~\ref{ggktr}). We use the notation of (\ref{sym}), 
(\ref{tmm}) and -- in (i) below -- the terminology of Remark~\ref{cifot}.
\begin{remark}\label{ascdt}According to 
\cite[Lemmas 11.1, 11.2 and Remark 2.1]{derdzinski-kp}, the following holds 
for any $\,(\mf\nh,g,\vt)\,$ as above with compact $\,\mf\nh$, the objects 
(\ref{tmm}), and $\,v=\navp$.
\begin{enumerate}
  \def\theenumi{{\rm\roman{enumi}}}
\item[{\rm(i)}] $Q=g(\hn v,v)\,$ is a $\,C^\infty$ function of $\,\vt$, 
leading to data $\,\tm,\tp,a,Q\,$ with (\ref{pbd}).
\item[{\rm(ii)}] The flow of the Kil\-ling vector field $\,u=J\nh v\,$ 
is periodic.
\item[{\rm(iii)}] $\sd^\pm$ are (connected) totally geodesic compact complex 
sub\-man\-i\-folds of $\,\mf\nh$.
\item[{\rm(iv)}] $\sd^+\nnh\cup\sd^-$ is the zero set of $\,v$, that is, the 
set of critical points of $\,\vt$.
\end{enumerate}
(Conclusion (iv) is a special case of a result due to Wang 
\cite[Lemma 3]{wang}.) Furthermore, restricting $\,\vt\mapsto Q\,$ in (i) to 
the open interval $\,(\tm,\tp)\,$ we have
\begin{equation}\label{dqt}
dQ/d\vt\,=\,2\psi\hh,\hskip12pt\mathrm{and\ so}\hskip7pt\psi\to\mp a\hskip7pt
\mathrm{as}\hskip7pt\vt\to\tpm\hh,
\end{equation}
$\psi\,$ being the function with (\ref{nvv}) on the open set 
$\,\mf'$ on which $\,v\ne0\,$ (so that $\,\psi\,$ is also a $\,C^\infty$ 
function of $\,\vt$). This is clear as (\ref{tnd}) and (\ref{nvv}) give 
$\,dQ=2\psi\,d\vt\,$ on $\,\mf\nh'\nh$. Finally, by \cite[Lemma 1]{wang} (see 
also \cite[Example 8.1 and Lemma 8.4(iv)]{derdzinski-kp}),
\begin{equation}\label{vtg}
v\,\mathrm{\ is\ tangent\ to\ every\ geodesic\ normal\ to\ }\,\sd^\pm\nh.
\end{equation}
\end{remark}
\begin{remark}\label{prcrm}Under the assumptions of Remark~\ref{ascdt}, for 
$\,a\,$ as in Remark~\ref{ascdt}(i), $\,\mp\hh a\,$ is the unique nonzero 
eigen\-val\-ue of the Hess\-i\-an of $\,\vt\,$ (that is, of 
$\,\dv=\nabla\nh v$) at any critical point $\,y\in\sd^\pm\nh$. The 
$\,\mp\hh a$-eigen\-space of $\,\dv\nnh_y\w$ is the normal space 
$\,\nr\hskip-2.4pt_y\w\sd^\pm\nh$, and 
$\,\mathrm{Ker}\hskip2.7pt\dv\nnh_y\w=\tyb^\pm$ (which thus constitutes the 
$\,0$-eigen\-space of $\,\dv\nnh_y\w$ unless $\,\sd^\pm\nh=\{y\}$).

In fact, as $\,\vt\,$ is a Morse-Bott function 
\cite[Example 8.1]{derdzinski-kp}, applying 
\cite[Lemma 8.4(i)]{derdzinski-kp} we see that $\,\nr\hskip-2.4pt_y\w\sd^\pm$ 
is the eigen\-space of $\,\dv\nnh_y\w$ for its unique nonzero eigen\-val\-ue, 
and so $\,\mathrm{Ker}\hskip2.7pt\dv\nnh_y\w=\tyb^\pm$ in view of 
self-ad\-joint\-ness of $\,\dv\nnh_y\w$. That the nonzero eigen\-val\-ue 
equals $\,\mp\hh a\,$ is obvious from (\ref{dqt}), since (\ref{nvv}) amounts 
to $\,\dv\nh v=\psi\hskip.4ptv$. Cf.\ \cite[Theorem 1.3]{miyaoka}.
\end{remark}
Still assuming compactness of a ge\-o\-des\-\hbox{ic\hs-}\hskip0ptgra\-di\-ent K\"ah\-ler 
triple $\,(\mf\nh,g,\vt)$, let $\,\nr^\delta\nnh\sd^\pm$ be the bundle of radius 
$\,\delta\,$ normal open disks around the zero section in the normal bundle 
$\,\nr\hskip-2.3pt\sd^\pm\nh$, with $\,\delta\,$ characterized by (\ref{int}). 
According to \cite[Lemma 10.3]{derdzinski-kp}, $\,\delta\,$ is then the 
distance between $\,\sd^+$ and $\,\sd^-\nh$, while, with 
$\,\mathrm{Exp}^\perp$ as in Remark~\ref{nexpm},
\begin{equation}\label{nex}
\begin{array}{l}
\mathrm{the\ restriction\ to\ }\nr^\delta\nnh\sd^\pm\mathrm{\,of\ the\ 
normal\ exponential\ mapping}\\
\mathrm{Exp}^\perp\nnh:\nr\hskip-2.3pt\sd^\pm\hs\to\,\mf\,\mathrm{\ is\ a\ 
dif\-feo\-mor\-phism\ 
}\,\nr^\delta\nnh\sd^\pm\to\,\mf\nnh\smallsetminus\nh\sd^\mp.
\end{array}
\end{equation}
Cf.\ \cite{bolton}, \cite[Lemma 2]{wang}, \cite[Theorem 1.1]{miyaoka}. Its 
inverse $\,\mf\nnh\smallsetminus\nh\sd^\mp\nh\to\nr^\delta\nnh\sd^\pm\nh$, 
composed with the projection $\,\nr^\delta\nnh\sd^\pm\nh\to\sd^\pm\nh$, yields 
a new disk-bundle projection
\begin{equation}\label{dbp}
\pi^\pm\nnh:\mf\nnh\smallsetminus\nh\sd^\mp\nh\to\sd^\pm.
\end{equation}
\begin{remark}\label{ptnrs}Clearly, $\,\pi^\pm\nnh\circ\mathrm{Exp}^\perp$ 
is the nor\-mal-bun\-dle projection 
$\,\nr\hskip-2.3pt\sd^\pm\nnh\to\sd^\pm\nnh$. Also, according to the lines 
preceding (\ref{dbp}),
\begin{equation}\label{img}
\begin{array}{l}
\mathrm{the\ image\ }\,\,\pi^\pm(x)\,\,\mathrm{\ of\ any\ 
}\,\,x\in\mf'\,\mathrm{\ is\ the\ unique\ }\,\,y\in\sd^\pm\,\mathrm{\ that\ 
can}\\
\mathrm{be\ joined\ to\ }\,\,x\,\,\mathrm{\ by\ a\ (necessarily\ unique)\ 
geodesic\ segment\ }\,\,\ic\hskip-3.3pt_x\w\,\mathrm{\ of}\\
\mathrm{length\ less\ than\ }\hs\delta\hs\mathrm{\ emanating\ from\ 
}\,y\,\mathrm{\ in\ a\ direction\ normal\ to\ }\hs\sd^\pm\nnh,
\end{array}
\end{equation}
which implies \cite[Remark 4.6, Example 8.1 and 
Theorem 10.2(iii)\hs--\hs(vi)]{derdzinski-kp} that $\,\pi^\pm$ sends every 
$\,x\in\mf\nnh\smallsetminus\nh\sd^\mp$ to the unique point nearest $\,x\,$ in 
$\,\sd^\pm\nh$.
\end{remark}
In the next lemma, by a {\medit leaf\hs} we mean -- as usual -- a maximal 
integral manifold.
\begin{lemma}\label{vsbkr}{\medit 
Under the above hypotheses, 
$\,\mathcal{V}\hs\subseteq\,\mathrm{Ker}\hskip2.3ptd\pi^\pm$ for the 
in\-te\-gra\-ble distribution\/ $\,\mathcal{V}=\mathrm{Span}\hs(\hn v,u)\,$ 
on\/ $\,\mf'\nh=\mf\smallsetminus(\sd^+\nnh\cup\sd^-)$, cf.\ 
Lemma\/~{\rm\ref{dvgww}(a)} and Remark\/~{\rm\ref{ascdt}(iv)}. If\/ $\,\xi$ 
is a unit vector normal to\/ $\,\sd^\pm$ at a point\/ $\,y$, then, with\/ 
$\,\delta\,$ as in\/ {\rm(\ref{nex})},
\begin{enumerate}
  \def\theenumi{{\rm\alph{enumi}}}
\item[{\rm(a)}] the punctured radius\/ $\,\delta\,$ disk\/ 
$\,\{z\xi:z\in\bbC\hskip6pt\mathrm{and}\hskip6pt0<|z|<\delta\}\,$ in\/ 
$\,\nr\hskip-2.4pt_y\w\sd^\pm$ is mapped by\/ $\,\exp_y\w\,$ 
dif\-feo\-mor\-phic\-al\-ly onto a leaf\/ $\,\lf\,\,$ of\/ $\,\mathcal{V}\nh$.
\end{enumerate}
Furthermore, every leaf\/ $\,\lf\hs\subseteq\mf'$ of\/ $\,\mathcal{V}\,$ has 
the following properties.
\begin{enumerate}
  \def\theenumi{{\rm\alph{enumi}}}
\item[{\rm(b)}] The closure of\/ $\,\lf\,\,$ in\/ $\,\mf\,$ is a totally 
geodesic complex sub\-man\-i\-fold, bi\-hol\-o\-mor\-phic to\/ $\,\bbCP^1$ and 
equal to\/ $\,\lf\hs\cup\{y_+\w,y_-\w\}$, where\/ $\,y_\pm\w\nh\in\sd^\pm$ 
are such that\/ $\,\{y_\pm\w\}=\pi^\pm\nh(\lf)$.
\item[{\rm(c)}] The leaf\/ $\,\lf\,\,$ arises from\/ {\rm(a)} \,for\/ some\/ 
unit normal vector\/ $\,\xi\,$ at the point\/ $\,y=y_\pm\w$ corresponding to\/ 
$\,\lf\hs\,$ as in\/ {\rm(b)}, and then
\begin{equation}\label{pex}
\pi^\mp\nh(\mathrm{Exp}^\perp\nh(y_\pm\w,z\xi))\,
=\,y_\mp\w\hskip12pt\mathrm{whenever}\hskip6ptz\in\bbC\hskip6pt\mathrm{and}
\hskip6pt0<|z|\le\delta\hh,
\end{equation}
$\mathrm{Exp}^\perp$ being the normal exponential mapping\/ 
$\,\nr\hskip-2.3pt\sd\to\mf\nh$.
\end{enumerate}
}
\end{lemma}
\begin{proof}Let us fix $\,x,y\,$ and $\,\ic\hskip-3.3pt_x\w$ as in 
(\ref{img}). Due to Remark~\ref{ascdt}(iv), the Kil\-ling field 
$\,u=J\nh v\,$ vanishes along $\,\sd^\pm\nh$, so that its 
infinitesimal flow at $\,y\,$ preserves both $\,\tyb^\pm$ and 
$\,\nr\hskip-2.4pt_y\w\sd^\pm\nh$. The images of $\,\ic\hskip-3.3pt_x\w$ under 
the flow transformations of $\,u\,$ thus are geodesic segments normal to 
$\,\sd^\pm$ emanating from $\,y\,$ and, as a consequence of (\ref{img}), 
$\,\pi^\pm$ maps them all onto $\,\{y\}$. In other words, the union of such 
segments, with the point $\,y$ removed, is simultaneously a subset of the 
$\,\pi^\pm\nh$-pre\-im\-age of $\,y\,$ as well as -- according to (\ref{vtg}), 
(\ref{gvv}) and parts (ii), (iv) of Remark~\ref{ascdt} -- a surface embedded 
in $\,\mf'\nnh$. This surface is, due to its very definition and (\ref{vtg}), 
tangent to both $\,u\,$ and $\,v\,$ which, in view of (\ref{nex}), yields (a); 
note that, by (\ref{loc}.a) and Remark~\ref{prcrm}, the orbit of $\,\xi\,$ 
under the flow of $\,\du=\nabla\nh u\,$ at $\,y\,$ consists of all unit 
complex multiples of $\,\xi$.

What we just observed about the orbit of $\,\xi\,$ clearly ensures smoothness 
of the closure of the leaf at $\,y$. By (\ref{nex}) and (\ref{vtg}), the 
union of $\,\ic\hskip-3.3pt_x\w$ and its analog for the same point $\,x\,$ 
and the {\medit other\/} projection $\,\pi^\mp$ is a length $\,\delta\,$ 
geodesic segment joining $\,y\in\sd^\pm$ to its other endpoint 
$\,y_\mp\w\in\sd^\mp\nh$. The above discussion of the images of such a 
segment under the flow of $\,u\,$ applies equally well to $\,y_\mp\w$, so that 
(b) -- (c) follow from Lemma~\ref{dvgww}(a) and the fact 
that $\,x\in\mf'$ was arbitrary.
\end{proof}
\begin{remark}\label{ppmgc}Let $\,(\mf\nh,g,\vt)\,$ be a Grass\-mann\-i\-an or CP 
triple, constructed as in Section~\ref{eg} from some data (\ref{dta}.i) or 
(\ref{dta}.ii). We use the notation of (\ref{sym}) and (\ref{dbp}).
\begin{enumerate}
  \def\theenumi{{\rm\alph{enumi}}}
\item[{\rm(a)}] We already know that the critical manifolds $\,\sd^\pm$ of 
$\,\vt\,$ are given by (\ref{spm}).
\item[{\rm(b)}] In the case of (\ref{spm}.c) (or, (\ref{spm}.b) and 
(\ref{spm}.d)), $\,\pi^\pm$ acts on $\,\ws\,$ as the orthogonal projection 
into $\,\ls\,$ (or, respectively, into $\,\ls\hskip-2.5pt^\perp$)
\item[{\rm(c)}] When $\,\sd^\pm$ has the form (\ref{spm}.a), $\,\pi^\pm$ 
sends $\,\ws\,$ to $\,\ls\oplus(\ws\nh\cap\ls\hskip-2.5pt^\perp)$.
\item[{\rm(d)}] The leaf of $\,\mathcal{V}\,$ through any $\,\ws\in\mf'$ 
consists 
\begin{enumerate}
  \def\theenumi{{\rm\alph{enumi}}}
\item[{\rm(d1)}] for (\ref{dta}.i) -- of all 
$\,\ls\hskip-2.5pt'\oplus\ws'\nh$, where 
$\,\ws'\nh=\ws\nh\cap\ls\hskip-2.5pt^\perp$ and $\,\ls\hskip-2.5pt'$ is any 
line in the plane $\,\ls\oplus(\ws'\nh\cap\ws^\perp)\,$ other than the lines 
$\,\ls\,$ and $\,\ws'\nh\cap\ws^\perp$ themselves,
\item[{\rm(d2)}] for (\ref{dta}.ii) -- of all lines other than $\,\ws'$ and 
$\,\ws''$ in the plane $\,\ws'\nh\oplus\ws''\nnh$, where $\,\ws'$ and 
$\,\ws''$ denote the orthogonal projections of $\,\ws\,$ into $\,\ls\,$ 
and $\,\ls\hskip-2.5pt^\perp\nnh$.
\end{enumerate}
\end{enumerate}
Namely, in both cases, let $\,G'$ be the complex Lie group of all 
com\-plex-lin\-e\-ar au\-to\-mor\-phisms of $\,\vs\,$ preserving both 
$\,\ls\,$ and $\,\ls\hskip-2.5pt^\perp\nnh$. The obvious action of $\,G'$ on 
the Stie\-fel manifold $\,\mathrm{St}\hn_k\w\nnh\tw\,$ (see 
Remark~\ref{grass}) descends to a hol\-o\-mor\-phic action on 
$\,\mathrm{Gr}\nh_k\w\nnh\vs\nh$, which becomes one on 
$\,\mathrm{P}\vs=\mathrm{G}\hn_1\w\nnh\vs\,$ when $\,k=1$. The elements of 
the center of $\,G'\nh$, restricted to both sub\-spaces $\,\ls\,$ and 
$\,\ls\hskip-2.5pt^\perp\nnh$, are complex multiples of $\,\mathrm{Id}$, and 
the action of the center on $\,\mathrm{Gr}\nh_k\w\nnh\vs\,$ includes 
the circle subgroup $\,S^1$ of $\,G\,$ generated by the Kil\-ling field $\,u$, 
mentioned in the lines following (\ref{dta}). Hol\-o\-mor\-phic\-i\-ty of the 
action implies that the flow of the gradient $\,v=\navp$, related to $\,u\,$ 
via $\,u=J\nh v$, also consists of transformations of 
$\,\mathrm{Gr}\nh_k\w\nnh\vs\,$ arising from the action of the center, and -- 
for dimensional reasons -- the orbits of the center coincide with the leaves 
of $\,\mathcal{V}=\mathrm{Span}\hs(\hn v,u)$. This easily gives (d). Now (b) 
-- (c) follow: by Lemma~\ref{vsbkr}(b), the two 
$\,\pi^\mp\hskip-2pt$-im\-a\-ges of any leaf of $\,\mathcal{V}\,$ are the two 
points that, added to the leaf, yield its closure.
\end{remark}
As $\,\mathcal{V}\hs\subseteq\,\mathrm{Ker}\hskip2.3ptd\pi^\pm$ 
(Lemma~\ref{vsbkr}), we may define vector sub\-bun\-dles 
$\,\mathcal{H}^\pm\nnh$ of $\,T\nnh\mf'$ by
\begin{equation}\label{hpm}
\mathcal{H}^\pm\hs
=\,\,\mathcal{V}^\perp\nh\cap\,\mathrm{Ker}\hskip2.3ptd\pi^\mp,\hskip22pt
\mathrm{so\ that}\hskip12pt\mathrm{Ker}\hskip2.3ptd\pi^\pm\hs
=\,\hs\mathcal{V}\hs\oplus\mathcal{H}^\mp.
\end{equation}
\begin{theorem}\label{dcomp}{\medit 
Given a ge\-o\-des\-ic-gra\-di\-ent K\"ah\-ler triple\/ $\,(\mf\nh,g,\vt)\,$ 
with compact\/ $\,\mf\nh$, 
$\,v,\mf'\nh,Q,\mathcal{V}\nh,\dv,\sd^\pm\nh,\tpm,\pi^\pm\nh,\mathcal{H}^\pm$ 
be defined as in\/ {\rm(\ref{sym})}, {\rm(\ref{tmm})} and\/ {\rm(\ref{hpm})}. 
Then the bundle en\-do\-mor\-phism\/ $\,2(\vt-\hn\tpm)\dv\hs-\hs Q\,$ of\/ 
$\,T\nnh\mf\nh$, restricted to\/ $\,\mathcal{V}^\perp\nh$, has constant rank 
on\/ $\,\mf'\nnh$, while
\begin{equation}\label{hvk}
\mathcal{H}^\mp\,
=\,\,\mathcal{V}^\perp\nh\cap\hs\mathrm{Ker}\hskip2pt
[2(\vt-\hn\tpm)\dv\hs-\hs Q]\hh,
\end{equation}
and some sub\-bun\-dle\/ $\,\mathcal{H}\,$ of\/ $\,T\nnh\mf'$ yields an\/ 
$\,\dv$-in\-var\-i\-ant complex orthogonal decomposition
\begin{equation}\label{tme}
T\nnh\mf'\hs=\,\,\mathcal{V}\hskip1.2pt\oplus\hs\mathcal{H}^+\nnh
\oplus\hs\mathcal{H}^-\nnh\oplus\hs\mathcal{H}\hh.
\end{equation}
Furthermore, for any\/ $\,\ic\nh\subseteq\mf'$ chosen as at the beginning of 
Section\/~{\rm\ref{hj}},the closure of\/ $\,\ic$ in\/ $\,\mf\,$ admits a 
unit-speed\/ $\,C^\infty\nnh$ parametrization\/ 
$\,[\hs t_-\w,t_+\w]\ni t\mapsto x(t)\,$ which, restricted 
to\/ $\,(t_-\w,t_+\w)$, is a parametrization of\/ $\,\ic\hs$ satisfying\/ 
{\rm(\ref{dxe})} along with the following conditions.
\begin{enumerate}
  \def\theenumi{{\rm\alph{enumi}}}
\item[{\rm(a)}] The endpoint\/ $\,y_\pm\w\nh=x(t_\pm\w)\,$ lies in\/ 
$\,\sd^\pm\nh$, and\/ $\,\dot x(t_\pm\w)\,$ is normal to\/ $\,\sd^\pm$ at\/ 
$\,y_\pm\w$.
\item[{\rm(b)}] Every solution\/ 
$\,(t_-\w,t_+\w)\ni t\mapsto w(t)\in\mathcal{V}^\perp_{\hskip-1ptx(t)}$ of\/ 
{\rm(\ref{nxw})} along\/ $\,\ic\hs$ has a\/ $\,C^\infty\nnh$ extension to\/ 
$\,[\hs t_-\w,t_+\w]\,$ such that\/ $\,d\pi^\pm_{x(t)}[w(t)]=w(t_\pm\w)\,$ 
whenever\/ $\,t\in(t_-\w,t_+\w)$.
\item[{\rm(c)}] The bundle projection\/ 
$\,\pi^\pm\nnh:\mf\nnh\smallsetminus\nh\sd^\mp\nh\to\sd^\pm$ is 
hol\-o\-mor\-phic.
\item[{\rm(d)}] If\/ $\,w\in\mathcal{W}[\tpm]$, cf.\ {\rm(\ref{wze})}, then, 
in\/ {\rm(b)}, 
$\,w(t_\pm\w)=0$, and\/ $\,[\nabla\hskip-3pt_{\dot x}\w w](t_\pm\w)\,$ 
is normal to\/ $\,\sd^\pm$ at\/ $\,y_\pm\w\nh=x(t_\pm\w)\,$ as well as 
orthogonal to\/ $\,\dot x(t_\pm\w)\,$ and\/ $\,J\dot x(t_\pm\w)$.
\item[{\rm(e)}] If\/ $\,w\,$ lies in the direct sum of spaces\/ 
$\,\mathcal{W}[\zx]\ne\{0\}\,$ with\/ $\,\zx\ne\tpm\hh$, for a fixed sign\/ 
$\,\pm\hs$, then\/ $\,w(t_\pm\w)\,$ is tangent to\/ $\,\sd^\pm$ at\/ 
$\,y_\pm\w\nh=x(t_\pm\w)$, and\/ 
$\,[\nabla\hskip-3pt_{\dot x}\w w](t_\pm\w)=0$.
\item[{\rm(f)}] Whenever\/ $\,t\in(t_-\w,t_+\w)\,$ and\/ $\,x=x(t)$, the 
assignment\/ 
$\,w(t)\mapsto(\hn w(t_\pm\w),\,[\nabla\hskip-3pt_{\dot x}\w w](t_\pm\w))$, with\/ 
$\,w\,$ as in\/ {\rm(b)}, is a\/ $\,\bbC$-lin\-e\-ar isomorphism\/ 
$\,\mathcal{V}^\perp_{\hskip-1ptx}\to\tyb^\pm\hskip-2.3pt\times\nh\nr'_y\hs$, 
where\/ $\,y=y_\pm\w$ and\/ $\,\nr'_y\,$ denotes the orthogonal complement 
of\/ $\,\mathrm{Span}\hs(\dot x(t_\pm\w),\hs J\dot x(t_\pm\w))\,$ in\/ 
$\,\nr\hskip-2.4pt_y\w\sd^\pm\nh$. At the same time, $\,w(t)\,$ then equals 
the image, under the differential 
of the normal exponential mapping\/ 
$\,\mathrm{Exp}^\perp\nnh:\nr\hskip-2.3pt\sd^\pm\nnh\to\mf\,$ at\/ 
$\,(y,\xi)\in\nr\hskip-2.3pt\sd^\pm$ given by\/ $\,y=x(t_\pm\w)\,$ and\/ 
$\,\xi=(t-t_\pm\w)\hh\dot x(t_\pm\w)$, of the vector tangent to\/ 
$\,\nr\hskip-2.3pt\sd^\pm$ at\/ $\,(y,\xi)\,$ which equals the sum of the 
vertical vector\/ 
$\,\eta=(t-t_\pm\w)[\nabla\hskip-3pt_{\dot x}\w w](t_\pm\w)\,$ and 
the\/ $\,\mathrm{D}\hn$-\hn hor\-i\-zon\-tal lift of\/ $\,w(t_\pm\w)\,$ to\/ 
$\,(y,\xi)$, for the normal connection\/ $\,\hs\mathrm{D}\,$ in\/ 
$\,\nr\hskip-2.3pt\sd^\pm\nnh$. Similarly, $\,u_{x(t)}\w$, for 
$\,u=J\nh v$, is the image, under the differential of\/ 
$\,\mathrm{Exp}^\perp$ at\/ $\,(y,\xi)$, of the vertical vector\/ 
$\,\eta=\mp\hh ai\hh\xi$.
\item[{\rm(g)}] For any\/ $\,w,w\hh'\nh\in\mathcal{W}\nh$, the function\/ 
$\,Q^{-\nnh1}g(\dv\nh w,w\hh'\hh)\,$ is constant on\/ $\,\ic\hs$ and the 
restriction of\/ $\,g(\hn w,w\hh'\hh)\,$ to\/ $\,\ic\hs$ is an af\-fine 
function of\/ $\,\vt:\ic\nh\to\bbR\,$ with the derivative\/ 
$\,d\hs[\hs g(\hn w,w\hh'\hh)]/\nh d\vt=2\hh Q^{-\nnh1}g(\dv\nh w,w\hh'\hh)$.
\item[{\rm(h)}] Explicitly, in\/ {\rm(g)}, with\/ $\,a\,$ as in 
Remark\/~{\rm\ref{ascdt}(i)}, either sign\/ $\,\pm\hs$, and\/ 
$\,y=y_\pm\w\nh=x(t_\pm\w)$,
\begin{enumerate}
  \def\theenumi{{\rm\alph{enumi}}}
\item[{\rm(h1)}] $g(\hn w,w\hh'\hh)
=(\tp\hskip-2.2pt-\nh\tm)^{-\nnh1}|\hh\vt-\hn\tmp|\hs g_y\w(\hn w_\pm\w,w_\pm'\hh)\,$ if\/ 
$\,w\in\mathcal{W}[\tmp]\,$ and\/ $\,w\hh'\nh\in\mathcal{W}$,
\item[{\rm(h2)}] $g(\hn w,w\hh'\hh)=g_y\w(\hn w_\pm\w,w_\pm'\hh)
-a^{-\nnh1}|\hh\vt-\hn\tpm|\,g_y\w(R_y\w(\hn w_\pm\w,J\hskip-2pt_y\w w_\pm'\hh)\hh
\dot x\nh_\pm\w,J\hskip-2pt_y\w\dot x\nh_\pm\w)\,$ if\/ $\,w,w\hh'$ both satisfy the 
assumption of\/ {\rm(e)},
\item[{\rm(h3)}] $g(\hn w,w\hh'\hh)
=2\hh a^{-\nnh1}|\hh\vt-\hn\tpm|\,g_y\w([\nabla\hskip-3pt_{\dot x}\w w]_\pm\w,
[\nabla\hskip-3pt_{\dot x}\w w\hh'\hh]_\pm\w)\,$ if\/ 
$\,w,w\hh'\nh\in\mathcal{W}[\tpm]$,
\end{enumerate}
\end{enumerate}
where the subscript\/ $\,\pm\,$ next to\/ 
$\,w,w\hh'\nh,\nabla\hskip-3pt_{\dot x}\w w,\nabla\hskip-3pt_{\dot x}\w w\hh'$ 
or\/ $\,\dot x\,$ represents their evaluation at\/ $\,t_\pm\w$.
}
\end{theorem}
\begin{remark}\label{bthex}Since $\,|\hh\vt-\hn\tpm|=\mp\hh(\vt-\hn\tpm)\,$ 
and $\,\pm\hh(\tp\hskip-2.2pt-\nh\tm)=\tpm\nh-\tmp$, applying $\,d/\nh d\vt\,$ to the 
right-hand side in (h1), or (h2), or (h3), we get the three values
\[
(\tpm\nh-\tmp)^{-\nnh1}\nh g_y\w(\hn w_\pm\w,\nh w_\pm'\hh),\hskip3.6pt
\pm\hh a^{-\nnh1}\nh g_y\w(R_y\w(\hn w_\pm\w,\nh J\hskip-2pt_y\w w_\pm'\hh)\hh
\dot x\nh_\pm\w,\nh J\hskip-2pt_y\w\dot x\nh_\pm\w)\hh,\hskip3.5pt
\mp2\hh a^{-\nnh1}\nh g_y\w([\nabla\hskip-3pt_{\dot x}\w w]_\pm\w,
\nh[\nabla\hskip-3pt_{\dot x}\w w\hh'\hh]_\pm\w).
\]
As a consequence of parts (g) -- (h) of Theorem~\ref{dcomp}, this triple 
provides the three expressions for 
$\,2\hh Q^{-\nnh1}g(\dv\nh w,w\hh'\hh)\,$ in the cases (h1), (h2) and (h3), 
respectively.

Note that the three different formulae for $\,g(\hn w,w\hh'\hh)\,$ in (h1), 
(h2) and -- with the reversed sign -- in (h3), are all simultaneously valid 
when $\,w,w\hh'\nh\in\mathcal{W}[\tmp]$.
\end{remark}
\begin{remark}\label{inequ}Under the assumptions of Theorem~\ref{dcomp},
\begin{enumerate}
  \def\theenumi{{\rm\roman{enumi}}}
\item[{\rm(i)}] the relation $\,\xi=(t-t_\pm\w)\hh\dot x(t_\pm\w)\,$ in (f) 
clearly gives $\,\dot x\nh_\pm\w=\mp\hs\xi/\hn|\hh\xi|\,$ in (h2),
\item[{\rm(ii)}] by (d) -- (f), the images under the differential of 
$\,\mathrm{Exp}^\perp$ of vertical (or, horizontal) vectors 
tangent to $\,\nr\hskip-2.3pt\sd^\pm$ at the point $\,(y,\xi)\,$ appearing in 
(f) have the form
\begin{enumerate}
  \def\theenumi{{\rm\roman{enumi}}}
\item[{\rm($*$)}] $w(t)\,$ for $\,w\,$ satisfying the hypothesis of (d) (or, 
respectively, of (e)),
\end{enumerate}
\item[{\rm(iii)}] the differential of $\,\pi^{\hs\pm}$ at any $\,x\in\mf'$ 
maps the summands $\,\mathcal{H}_x^\pm$ and $\,\mathcal{H}_x\w$ in (\ref{tme}) 
iso\-mor\-phic\-al\-ly onto the images 
$\,d\pi\nh_x^{\hs\pm}\nh(\mathcal{H}_x^\pm\nh)$ and 
$\,d\pi\nh_x^{\hs\pm}\nh(\mathcal{H}_x\w)$, orthogonal to each other in 
$\,\tyb^\pm$ for $\,y=\pi^\pm\nh(x)$,
\item[{\rm(iv)}] one has $\,(\tp\hskip-2.2pt-\nh\tm)\hs g_x\w(\hn w,w\hh'\hh)
=|\hh\vt(x)-\hn\tmp|\hs 
g_y\w(d\pi\nh_x^{\hs\pm}w,d\pi\nh_x^{\hs\pm}w\hh'\hh)\,$ whenever 
$\,w\in\mathcal{H}_x^\pm$ and $\,w\hh'\nh\in\mathcal{V}^\perp_{\hskip-1ptx}$ 
at any $\,x\in\mf'\nh$, while $\,y=\pi^\pm\nh(x)$,
\item[{\rm(v)}] (\ref{tmm}) and (a) imply the inequality of 
Theorem~\ref{jacob}(viii) everywhere in $\,\mf'\nh$.
\end{enumerate}
Only (iii) and (iv) require further explanations. For (iii), 
$\,d\pi\nh_x^{\hs\pm}$ is injective on the space 
$\,\mathcal{H}_x^\pm\nh\oplus\mathcal{H}_x$, orthogonal, by (\ref{hpm}) and 
(\ref{tme}), to its kernel 
$\,\mathcal{V}_{\nh x}\w\oplus\mathcal{H}_x^\mp\nh$. Or\-thog\-o\-nal\-i\-ty 
in (\ref{hpm}) also shows, via (\ref{hvk}), (\ref{wze}.b) and (\ref{loc}.d), 
that vectors in $\,\mathcal{H}_x^\pm$ (or, in $\,\mathcal{H}_x$) have the form 
($*$) with $\,x=x(t)$, cf.\ (f), and the former remain orhogonal to the latter 
as $\,t\,$ varies, leading to (iii) as a consequence of the final clause of 
(b). Assertion (iv) is nothing else than (h1) for $\,w=w(t)\,$ at $\,x=x(t)$, 
cf.\ (f), where $\,y=\pi^\pm\nh(x)\,$ and $\,d\pi\nh_x^{\hs\pm}w=w_\pm\w$ by 
(\ref{img}) and (b).
\end{remark}
\begin{remark}\label{holom}As another immediate consequence of 
Theorem~\ref{dcomp}, the assignment 
$\,x\mapsto d\pi\nh_x^{\hs\pm}\nh(\mathcal{H}_x^\pm\nh)
=d\pi\nh_x^{\hs\pm}\nh(\mathcal{V}_{\nh x}\w\oplus\mathcal{H}_x^\pm\nh)\,$ 
defines a hol\-o\-mor\-phic section of the bundle over $\,\mf'$ arising via 
the pull\-back under $\,\pi^\pm$ from 
$\,\mathrm{Gr}\nh_k\w\hn(T\nnh\sd^\pm\nh)$, for a suitable integer 
$\,k=k_\pm\w$. Here 
$\,\mathrm{Gr}\nh_k\w\hn(T\nnh\sd^\pm\nh)\,$ is the Grass\-mann\-i\-an bundle 
with the fibres $\,\mathrm{Gr}\nh_k\w\nh(\tyb^\pm\nh)$, $\,y\in\sd^\pm$ (cf.\ 
Section~\ref{eg}), hol\-o\-mor\-phic\-i\-ty and the equality 
$\,d\pi\nh_x^{\hs\pm}\nh(\mathcal{H}_x^\pm\nh)
=d\pi\nh_x^{\hs\pm}\nh(\mathcal{V}_{\nh x}\w\oplus\mathcal{H}_x^\pm\nh)\,$ are 
clear from Theorem~\ref{dcomp}(c) (which also implies, due to (\ref{hpm}), 
that $\,\mathcal{V}\hs\oplus\mathcal{H}^\pm$ is a hol\-o\-mor\-phic 
sub\-bun\-dle of $\,T\nnh\mf'$) and (\ref{tme}) (which, combined with 
(\ref{hpm}), ensures constancy of the dimension $\,k=k_\pm\w$ of the spaces 
$\,d\pi\nh_x^{\hs\pm}\nh(\mathcal{H}_x^\pm\nh)$).
\end{remark}

\section{Proof of Theorem~\ref{dcomp}\done}\label{pt}
\setcounter{equation}{0}
We begin by establishing (a) - (f) under the stated assumptions about 
$\,\ic\nh$.

Let $\,(t_-\w,t_+\w)\mapsto x(t)\,$ be a parametrization of $\,\ic\hs$ with 
(\ref{dxe}). As $\,\vt\,$ then is clearly an increasing function of $\,t$, it 
has some limits $\,\htpm$ as $\,t\to t_\pm\w$, finite due to boundedness of 
$\,\vt$. The length of $\,\ic\hs$ obviously equals the integral of 
$\,Q^{-\nnh1\nh/2}$ over $\,(\htm,\htp)\subseteq(\tm,\tp)$, and so it is finite in 
view of (\ref{int}). This implies the existence of limits $\,x(t_\pm\w)\,$ of 
$\,x(t)\,$ as $\,t\to t_\pm\w$. Furthermore, each $\,x(t_\pm\w)\,$ lies in 
$\,\sd^\pm$ since, if one $\,x(t_\pm\w)\,$ did not, Remark~\ref{ascdt}(iv) 
would yield $\,v\ne0\,$ at $\,x(t_\pm\w)$, contradicting maximality of 
$\,\ic\nh$. Thus, $\,[\hs t_-\w,t_+\w]\mapsto x(t)\,$ parametrizes the closure 
of $\,\ic\hs$. Next, $\,\mf\smallsetminus(\sd^+\nnh\cup\sd^-)\,$ is, by 
(\ref{nex}) and (\ref{vtg}), a disjoint union of maximal integral curves of 
$\,v$, each of which has two limit points, one in $\,\sd^-$ and one in 
$\,\sd^+\nh$, and the corresponding limit directions of the curve are normal 
to $\,\sd^-$ and $\,\sd^+\nh$. Since $\,\ic\hs$ is one of these curves, (a) 
follows.

In (b), a $\,C^\infty$ extension to $\,[\hs t_-\w,t_+\w]\,$ must exist as 
$\,w\,$ is a Ja\-co\-bi field; see Theorem~\ref{jacob}(ii). To obtain (d) -- 
(e), we fix $\,w\in\mathcal{W}[\zx]$, so that, from (\ref{lgm}) -- (\ref{wze}),
\begin{equation}\label{swe}
\mathrm{i)}\hskip6pt\dv\nh w\,=\,Qw/[2(\vt-\zx)]\hh,\hskip22pt
\mathrm{ii)}\hskip6pt\nabla\hskip-3pt_{\dot x}\w w\,=\,Q^{\hs1\nh/2}\hs w/[2(\vt-\zx)]
\hh.
\end{equation}
Let $\,y=x(t_\pm\w)\,$ and $\,w_y\w=w(t_\pm\w)$. By 
Remark~\ref{ascdt}(iv), $\,Q=\vt-\hn\tpm=0\,$ on $\,\sd^\pm$ while, in view of 
(a) and (\ref{dqt}), $\,Q/[2(\vt-\hn\tpm)]\,$ evaluated at $\,x(t)\,$ tends to 
$\,\mp a\ne0\,$ as $\,t\to t_\pm\w$. If $\,\zx=\tpm$, (\ref{swe}.ii) 
multiplied by $\,Q^{\hs1\nh/2}$ thus yields $\,w_y\w=0$, and the relation 
$\,\dv\nh w\hh'\nh=Qw\hh'\nh/[2(\vt-\zx)]\,$ for 
$\,w\hh'\nh=\nabla\hskip-3pt_{\dot x}\w w$, obvious from (\ref{swe}), implies 
that $\,[\nabla\hskip-3pt_{\dot x}\w w](t_\pm\w)$ lies in the 
$\,\mp\hh a$-eigen\-space of $\,\dv\nnh_y\w$. When $\,\zx\ne\tpm$, (\ref{swe}.i) 
and (\ref{swe}.ii) give, respectively, $\,\dv\nnh_y\w w_y\w=0\,$ and 
$\,[\nabla\hskip-3pt_{\dot x}\w w](t_\pm\w)=0$. Due to Remark~\ref{prcrm}, 
this  proves (d) and (e): or\-thog\-o\-nal\-i\-ty in (d) follows since $\,w\,$ 
and $\,\nabla\hskip-3pt_{\dot x}\w w\,$ take values in 
$\,\mathcal{V}^\perp\nh$, for $\,\mathcal{V}=\mathrm{Span}\hs(\hn v,u)$ (so 
that $\,g(\hn w,v)=g(\hn w,u)=0$), while $\,\dot x=v/|v|\,$ by (\ref{dxe}), 
and $\,u=J\nh v$.

Furthermore, the assignment in (f) is well-defined, injective, 
com\-plex-lin\-e\-ar and $\,(\tyb^\pm\hskip-2.3pt\times\nh\nr'_y)$-val\-ued 
due to parts (iii), (ii), (iv) of Theorem~\ref{jacob} and, respectively, (d) 
-- (e). The first claim of (f) thus follows since both spaces have the same 
dimension. The second (or, third) one is in turn immediate from (\ref{dxp}) 
applied, at $\,r=1$, to any $\,w\in\mathcal{W}\nh$, cf.\ 
Theorem~\ref{jacob}(ii) (or, to $\,w=u$), with $\,y,\xi,\eta\,$ as in (f), and 
$\,\hat w\,$ defined by $\,\hat w(\hn r\nh)=w(rt+(1-r)\hh t_\pm\w)$. (That 
$\,r\mapsto\hat w(\hn r\nh)\,$ then is a Ja\-cobi field along the geodesic 
$\,r\mapsto x(rt+(1-r)\hh t_\pm\w)\,$ follows from Theorem~\ref{jacob}(ii) or, 
respectively, (\ref{loc}.b) and Remark~\ref{kiljc}, while, in the latter case, 
due to (\ref{loc}.a) along with Remarks~\ref{ascdt}(iv) and~\ref{prcrm}, 
$\,w=u\,$ realizes the initial conditions 
$\,(\hn u,\hs\nabla\hskip-3pt_{dx\hn/\hn dr}\w u)=(0,\mp\hh ai\hh\xi)\,$ at 
$\,r=0$.)

The remaining equality $\,d\pi^\pm_{x(t)}[w(t)]=w(t_\pm\w)\,$ in (b) now 
becomes an obvious consequence of the second part of (f) combined with the 
first line of Remark~\ref{ptnrs}. This proves (b) and, combined with 
Theorem~\ref{jacob}(iv), implies (c).

Next, for $\,t\mapsto x(t)\,$ as in (a) -- (f), any $\,t\in(t_-\w,t_+\w)$, a 
fixed sign $\,\pm\hh$, and $\,x=x(t)$, Theorem~\ref{jacob}(iii), (\ref{hpm}) 
and (b) give 
$\,\mathcal{H}^\mp_x
=\{w(t):w\in\mathcal{W}\hs\,\,\mathrm{and}\,\,w(t_\pm\w)=0\}$. Writing any 
$\,w\in\mathcal{W}\,$ as $\,w=w\hh'\nh+w\hh''\nh$, where 
$\,w\hh'\nh\in\mathcal{W}[\tpm]\,$ and $\,w\hh''$ lies in the direct sum of 
the spaces $\,\mathcal{W}[\zx]\ne\{0\}\,$ with\/ $\,\zx\ne\tpm\hh$, 
cf.\ Theorem~\ref{jacob}(iv), we see that, by (d) -- (e), the isomorphism in 
(f) sends $\,w\hh'\nh(t)\,$ and $\,w\hh''\nh(t)$, respectively, to pairs of 
the form $\,(0,\,\cdot\,)\,$ and $\,(\,\cdot\,,0)$. Thus, 
$\,w(t)\in\mathcal{H}^\mp_x$ if and only if $\,w\hh''\nh=0$, that is, 
$\,w\in\mathcal{W}[\tpm]$. Combining Theorem~\ref{jacob}(vi) with (\ref{lgm}) 
and (\ref{wze}.b), one now obtains (\ref{hvk}), so that (\ref{hpm}) implies 
the con\-stant-rank assertion preceding (\ref{hvk}). On the other hand, 
$\,\mathcal{H}^+_x$ and $\,\mathcal{H}^-_x$ are mutually orthogonal at every 
$\,x\in\mf'\nh$, being, by (\ref{hvk}), contained in eigen\-spaces 
corresponding to different eigen\-values of the self-ad\-joint operator 
$\,\dv\nnh_x\w$, cf.\ (\ref{loc}.d), so that (\ref{tme}) follows.

Let $\,w,w\hh'\nh\in\mathcal{W}\nh$. Constancy of 
$\,Q^{-\nnh1}g(\dv\nh w,w\hh'\hh)\,$ along $\,\ic\hs$ trivially follows from 
(\ref{dvg}.iii) and (\ref{vtg}), cf.\ Lemma~\ref{dvgww}(b) and  parts (i) -- 
(ii) of Theorem~\ref{jacob}. The operators $\,d/\nh d\vt\,$ and $\,d_v\w$ 
acting on functions $\,\ic\nh\to\bbR\,$ are in turn related by 
$\,d_v\w\nh=Q\,d/\nh d\vt$, since (\ref{dxe}) gives 
$\,d_v\w\nh=Q^{\hs1\nh/2}d_{\dot x}\w\nh=Q^{\hs1\nh/2}d/\nh dt$, while 
$\,d/\nh dt=Q^{\hs1\nh/2}d/\nh d\vt\,$ due to (\ref{dtq}). Now (g) is 
immediate from (\ref{dvg}.ii).

In (h), all three right-hand sides are af\-fine functions of $\,\vt\,$ with 
the correct values at $\,t=t_\pm\w$ (that is, limits at the endpoint 
$\,y_\pm\w\nh=x(t_\pm\w)$). Proving (h) is thus reduced by (g) to showing 
that, in each case, $\,\chi=2\hh Q^{-\nnh1}g(\dv\nh w,w\hh'\hh)\,$ coincides 
with the derivative of the right-hand side provided by Remark~\ref{bthex}, 
which -- even though $\,\chi\,$ is constant on $\,\ic\nh$, cf.\ (g) -- will be 
achieved via evaluating the limit of $\,\chi\,$ at $\,y_\pm\w\nh\in\ic\hs$ or, 
equivalently, at $\,t_\pm\w\nh\in[\hs t_-\w,t_+\w]$. When 
$\,w\in\mathcal{W}[\tmp]$, (\ref{wze}.b) and (\ref{lgm}) imply that 
$\,2\hh Q^{-\nnh1}\dv\nh w=(\vt-\hn\tmp)^{-\nnh1}w\,$ and, consequently, 
$\,\chi=(\vt-\hn\tmp)^{-\nnh1}g(\hn w,w\hh'\hh)\,$ has the value (and limit) 
$\,\pm\hh(\tp\hskip-2.2pt-\nh\tm)^{-\nnh1}g_y\w(\hn w_\pm\w,w_\pm'\hh)\,$ at 
$\,y=y_\pm\w$, as required in (h1).

Let $\,w,w\hh'$ now satisfy the hypotheses of (e). Consequently, along 
$\,\ic\nh\smallsetminus\{y_\pm\w\}$,
\begin{equation}\label{ttz}
Q\hh,\hskip9pt\dv\nh w\hh',\hskip9ptQ^{-\nnh1}g(\dv\nh w,\dv\nh w\hh'\hh)\hskip13pt
\mathrm{all\ tend\ to}\hskip6pt0\hskip6pt\mathrm{at}\hskip6pty\hh,\hskip6pt
\mathrm{where}\hskip6pty=y_\pm\w\hh.
\end{equation}
In fact, $\,Q(y)=0\,$ by (a). Next, $\,Q^{-\nnh1}\dv\nh w\,$ is bounded near the 
endpoint $\,y\,$ of $\,\ic\nh\smallsetminus\{y\}$ (and similarly for 
$\,w\hh'\hh$); to see this, we may assume that 
$\,w\in\mathcal{W}[\zx]\,$ with $\,\zx\ne\tpm$, cf.\ (e), and then 
(\ref{wze}.b) and (\ref{lgm}) give 
$\,2\hh Q^{-\nnh1}\dv\nh w=(\vt-\zx)^{-\nnh1}w$, which is bounded as 
$\,\vt\to\tpm$ since, due to (b), $\,w\,$ has a limit at $\,t=t_\pm\w$. Now 
(\ref{ttz}) follows.

In view of (\ref{ttz}) and (b), we may now evaluate the limit of 
$\,\chi=2\hh Q^{-\nnh1}g(\dv\nh w,w\hh'\hh)$ as $\,t\to t_\pm\w$ using 
l'H\^opital's rule: it coincides with the limit of 
$\,2\hs d_{\dot x}\w[\hs g(\dv\nh w,w\hh'\hh)]/\dot Q$. By (\ref{loc}.c) for 
$\,\dot x\,$ rather than of $\,w$, (\ref{nxw}), (\ref{loc}.d) and (\ref{dtq}), 
this last expression is the sum of two terms, 
$\,\psi^{-\nnh1}Q^{-\nnh1\nh/2}g(R(\hn u,\dot x)\hh w,J\nh w\hh'\hh)\,$ and 
$\,2\psi^{-\nnh1}Q^{-\nnh1}g(\dv\nh w,\dv\nh w\hh'\hh)$. According to 
(\ref{ttz}) and (\ref{dqt}), only the first term contributes to the limit and, 
as it equals $\,\psi^{-\nnh1}g(R(J\dot x,\dot x)\hh w,J\nh w\hh'\hh)$, cf.\ 
(\ref{gvv}) and (\ref{vtg}), relation (\ref{dqt}) yields (h2).

Finally, suppose that $\,w,w\hh'\nh\in\mathcal{W}[\tpm]$. It follows that
\begin{equation}\label{qoh}
Q^{-\nnh1\nh/2}w\,
\to\,\mp\hh a^{-\nnh1}[\nabla\hskip-3pt_{\dot x}\w w]_\pm\w\hskip13pt
\mathrm{at}\hskip6pty\hh,\hskip6pt\mathrm{where}\hskip6pty=y_\pm\w\hh,
\end{equation}
and analogously for $\,w\hh'\nh$. Namely, $\,Q\,$ and $\,w\,$ vanish at 
$\,y\,$ (see (a), (d)), while 
$\,(Q^{\hs1\nh/2})\nnh\dot{\phantom o}\nnh=\psi\,$ by (\ref{dtq}), and so 
$\,[\nabla\hskip-3pt_{\dot x}\w w]/(Q^{\hs1\nh/2})\nnh\dot{\phantom o}\nnh
=\psi^{-\nnh1}\nabla\hskip-3pt_{\dot x}\w w$. L'H\^opital's rule and (\ref{dqt}) now 
imply (\ref{qoh}). Since 
$\,\dv\nnh_y\w[\nabla\hskip-3pt_{\dot x}\w w]_\pm\w\nh
=\mp\hh a\hh[\nabla\hskip-3pt_{\dot x}\w w]_\pm\w$ by (d) and 
Remark~\ref{prcrm}, assertion (h3) is obvious from (\ref{qoh}), completing the 
proof of Theorem~\ref{dcomp}.
\begin{remark}\label{dppdm}With the same notations and assumptions as in 
Theorem~\ref{dcomp}, denoting by $\,k_\pm\w$ and $\,q\,$ the complex fibre 
dimensions of the sub\-bundles $\,\mathcal{H}^\pm$ and $\,\mathcal{H}\,$ of 
$\,T\nnh\mf'\nh$, we have, for $\,m=\dimc\nh\mf\,$ and 
$\,d_\pm\w\nh=\dimc\nh\sd^\pm\nh$, 
\begin{equation}\label{dmq}
d_+\w+\hs\,d_-\w\hs=\,\hs m\,-\,1\,+\,q\hh,
\end{equation}
as one sees adding the equalities $\,d_\pm\w\nh=m-1-k_\pm\w$ and 
$\,m=1+k_+\w\nnh+k_-\w\nnh+q$ (the former due to (\ref{dbp}) and (\ref{hpm}), 
the latter to (\ref{tme})). Consequently,
\begin{equation}\label{dpp}
d_+\w+\hs\,d_-\w\hs\,\ge\,\,m\,-\,1\hh,
\end{equation}
with equality if and only if the distribution $\,\mathcal{H}\,$ in (\ref{tme}) 
is \hbox{$\,0$-}\hskip0ptdi\-men\-sion\-al, that is, if
\begin{equation}\label{tmp}
T\nnh\mf'\,=\,\,\mathcal{V}\hskip1.2pt\oplus\hs\mathcal{H}^+\nnh
\oplus\hs\mathcal{H}^-\nnh.
\end{equation}
The explicit descriptions of $\,\sd^\pm$ in (\ref{spm}.c) -- (\ref{spm}.d) 
clearly show that
\begin{equation}\label{dem}
d_+\w+\hs\,d_-\w\hs=\,\hs m\,-\,1\,\,\mathrm{\ \ for\ every\ CP\ triple\ 
}\,(\mf\nh,g,\vt)\hh.
\end{equation}
\end{remark}

\section{Examples: Nontrivial modifications\done}\label{nm}
\setcounter{equation}{0}
\begin{remark}\label{realz}For any two functions $\,\vt\mapsto Q\,$ and 
$\,\hat\vt\mapsto\hat Q\,$ having the properties listed in (\ref{pbd}), with 
the same $\,\tpm$ and $\,a$, there must exist an increasing $\,C^\infty\nnh$ 
dif\-feo\-mor\-phism $\,[\hh\tm,\tp]\ni\vt\mapsto\hat\vt\in[\hh\tm,\tp]\,$ 
which realizes
\begin{equation}\label{qdt}
\begin{array}{l}
\mathrm{the\ equality\ }\hs\hat Q\,d\nh/\nnh d\hat\vt\hs
=\hh Q\,d\nh/\nnh d\vt\,\mathrm{\ of\ vector\ fields\ on\ 
}\,[\hh\tm,\tp]\,\mathrm{\ expressed}\\
\mathrm{in\ terms\ of\ the\ two\ 
dif\-feo\-mor\-phic\-al\-ly}\hyp\mathrm{re\-lat\-ed\ coordinates\ }
\,\hat\vt\,\mathrm{\ and\ }\,\vt.
\end{array}
\end{equation}
Such a dif\-feo\-mor\-phism is unique up to compositions from the left (or, 
right) with transformations forming the flow of the first (or, second) vector 
field in (\ref{qdt}).

To see this, apply Remark~\ref{adtet} to $\,t=\vt-\hn\tpm$ and 
$\,\gamma=\mp\hh Q/(2a)$ (or, $\,t=\hat\vt-\hn\tpm$ and 
$\,\gamma=\mp\hh\hat Q/(2a)$), obtaining a function $\,\theta\,$ (or, 
$\,\hat\theta$), unique up to a positive constant factor and vanishing at 
$\,\vt=\hn\tpm$ (or, $\,\hat\vt=\hn\tpm$), with 
$\,d\hh\theta\hn/\nh d\vt=\theta/Q\,$ (or, 
$\,d\hh\hat\theta\hn/\nh d\hat\vt=\hat\theta/\hat Q$), the derivative being 
positive everywhere in $\,[\hh\tm,\tp]$. Adjusting the constant factor, we may 
require $\,\vt\mapsto\theta\,$ and $\,\hat\vt\mapsto\hat\theta\,$ to be 
increasing dif\-feo\-mor\-phisms of $\,[\hh\tm,\tp]$ onto {\medit the 
same\/} interval having the endpoint $\,0$, and then define 
$\,\vt\mapsto\hat\vt\,$ by declaring $\,\theta$ ``equal to'' 
$\,\hat\vt\mapsto\hat\theta$, that is, letting $\,\vt\mapsto\hat\vt\,$ be 
$\,\vt\mapsto\theta\,$ followed by the inverse of 
$\,\hat\vt\mapsto\hat\theta$. Consequently, 
$\,d\hat\vt\hn/\nh d\vt=\hat Q/Q\,$ on $\,(\tm,\tp)$, which amounts to 
(\ref{qdt}).

The uniqueness clause is obvious: the only self-dif\-feo\-mor\-phisms 
$\,\zeta\,$ of $\,(\tm,\tp)$ preserving a given vector field without zeros 
are its flow transformations, since $\,\zeta$ acts on an integral curve as a 
shift of the parameter.
\end{remark}
\begin{theorem}\label{ntrmd}{\medit 
For the data\/ $\,\tpm$ and\/ $\,\vt\mapsto Q\,$ related via 
Remark\/~{\rm\ref{ascdt}(i)} to a given compact ge\-o\-des\-ic-gra\-di\-ent 
K\"ah\-ler triple\/ $\,(\mf\nh,g,\vt)$, and any increasing\/ $\,C^\infty\nnh$ 
dif\-feo\-mor\-phism\/ $\,[\hh\tm,\tp]\ni\vt\mapsto\hat\vt\in[\hh\tm,\tp]$, 
there exists a\/ $\,C^\infty\nnh$ function\/ 
$\,[\hh\tm,\tp]\ni\vt\mapsto\phi\in\bbR$, unique up to additive constants, 
such that\/ $\,\,\hat\vt\,=\,\vt\hs+\hh Q\,d\phi\hn/\nh d\vt$.

With\/ $\,\hat\vt,\phi\,$ treated, due to their dependence on\/ $\,\vt$, as 
functions on the complex manifold\/ $\,\mf\nh$, the formula\/ $\,\hg=g
-2(i\hs\partial\hskip1.7pt\cro\hskip0pt\phi)(J\,\cdot\,,\,\cdot\,)\,$ 
then defines another K\"ah\-ler metric on\/ $\,\mf\nh$, and\/ 
\begin{enumerate}
  \def\theenumi{{\rm\alph{enumi}}}
\item[{\rm(a)}] $(\mf,\hg,\hat\vt)\,$ is a new ge\-o\-des\-ic-gra\-di\-ent 
K\"ah\-ler triple.
\end{enumerate}
In addition, denoting by\/ $\,\hat\vt\mapsto\hat Q\,$ the analog of\/ 
$\,\vt\mapsto Q$ arising when Remark\/~{\rm\ref{ascdt}(i)} is applied to\/ 
$\,(\mf,\hg,\hat\vt)$, and by\/ $\,\hat\nabla\hskip-.5pt\hat\vt\,$ the\/ 
$\,\hg$-gra\-di\-ent of\/ $\,\hat\vt$, \,one has\/ {\rm(\ref{qdt})} and\/ 
$\,\hat\nabla\hskip-.5pt\hat\vt=\nh\navp$.
}
\end{theorem}
\begin{proof}As $\,\hat\vt=\vt\hs+\hh Q\hh\phi'\nh$, where 
$\,(\hskip2.3pt)'\nh=d\nh/\nnh d\vt$, our assumption about 
$\,\vt\mapsto\hat\vt\,$ gives $\,\hat\vt'\nh>0$ and 
$\,\tm\nh\le\hs\hat\vt\le\tp$, leading to the inequalities
\begin{equation}\label{neq}
\tm\le\,\vt\hs+\hs Q\hh\phi'\le\hs\tp\hskip7pt\mathrm{(strict\ except\ 
at}\hskip5pt\vt=\tpm\mathrm{)\ and}\hskip7pt1+Q'\phi'\nh+Q\hh\phi''\nh>0\hh.
\end{equation}
Note that $\,\phi\,$ exists since, by Remark~\ref{smdiv} and 
(\ref{pbd}), $\,\hat\vt-\nh\vt\,$ and $\,Q\,$ are smoothly divisible by 
$\,\vt-\hn\tpm$, their quotients being equal at $\,\tpm$ to the value of 
$\,\hat\vt'\nh-1\,$ and $\,\mp2a$, respectively, and so, as $\,\hat\vt'\nh>0$, 
\begin{equation}\label{fgt}
\mp2a\hh\phi'\hs>\,-\nnh1\hskip8pt\mathrm{at}\hskip7pt\vt=\tpm\hh.
\end{equation}
For the self-ad\-joint bundle en\-do\-mor\-phism $\,K\,$ of $\,T\nnh\mf\,$ 
with $\,\hg=g(K\,\cdot\,,\,\cdot\,)\,$ one has
\begin{equation}\label{cei}
K\,\hs=\,\,\mathrm{Id}\,\hs+\,2\phi'\nh\dv\,
+\,\phi''[\hh g(\hn v,\,\cdot\,)\hh v+g(\hn u,\,\cdot\,)\hh u]\hh,
\end{equation}
where $\,v,u,\dv\,$ are, as usual, given by $\,v=\navp$, $\,u=J\nh v$, and 
$\,\dv=\nabla\nh v$. 

We proceed to prove positivity of $\,K\,$ at all points of $\,\mf\nh$, 
considering two separate cases: $\,y\in\sd^\pm$ and 
$\,x\in\mf'\nh=\mf\smallsetminus(\sd^+\nnh\cup\sd^-)$, cf.\ 
Remark~\ref{ascdt}(iv).

If $\,y\,$ lies in either critical manifold $\,\sd^\pm\nh$, the relations 
$\,v_y\w=u_y\w=0\,$ and $\,\vt(y)=\tpm$ imply positivity of $\,K\nnh_y\w$ as a 
consequence of (\ref{fgt}) since, by Remark~\ref{prcrm}, any eigen\-val\-ue of 
$\,\dv\nnh_y\w$ must be equal to $\,0\,$ or $\,\mp\hh a$.

On $\,\mf'\nh$, we use the $\,\dv$-in\-var\-i\-ant decomposition 
$\,T\nnh\mf'\nh=\mathcal{V}\oplus\mathcal{V}^\perp\nnh$, cf.\ (\ref{tme}). 
In view of (\ref{loc}.f) -- (\ref{loc}.g), the restriction of 
$\,2\dv=2\nabla\nh v\,$ to $\,\mathcal{V}=\mathrm{Span}\hs(\hn v,u)\,$ equals 
$\,Q'$ times $\,\mathrm{Id}$. Using (\ref{cei}) and (\ref{gvv}) we now see 
that $\,K\,$ acts in $\,\mathcal{V}\,$ via multiplication by the function 
$\,1+Q'\phi'\nh+Q\hh\phi''\nh$, which is positive according to (\ref{neq}). 
Theorem~\ref{jacob}(vi) states in turn that the eigen\-values of 
$\,\dv\nnh_x\w:\mathcal{V}^\perp_{\hskip-1ptx}\nh
\to\mathcal{V}^\perp_{\hskip-1ptx}$, for $\,x\in\mf'\nh$, have the form 
$\,\lambda_\zx\w(x)\,$ with (\ref{lgm}) and (\ref{wze}.a). Writing 
$\,K,\dv,\vt,Q,\phi'$ instead of their values at $\,x$, we conclude from 
(\ref{cei}) that the corresponding eigen\-values of $\,(\vt-\zx)\hh K\,$ are 
$\,\vt\hs+\hs Q\hh\phi'\nh-\zx\,$ and so, due to the (strict) first inequality 
of (\ref{neq}), they all lie in the interval $\,(\tm\nh-\zx,\tp\nh-\zx)$. 
Positivity of $\,K\,$ on $\,\mathcal{V}\,$ thus easily follows both when 
$\,\zx<\hh\tm\nnh<\vt\,$ and when $\,\vt<\hh\tp\nnh<\zx$.

Consequently, $\,\hg\,$ is a K\"ah\-ler metric on $\,\mf\nh$, with 
the K\"ah\-ler form $\,\hat\omega=\hg(\hat J\,\cdot\,,\,\cdot\,)$ related to 
$\,\omega=g(\hat J\,\cdot\,,\,\cdot\,)\,$ by 
$\,\hat\omega=\omega+2i\hs\partial\hskip1.7pt\cro\hskip0pt\phi$. Applying 
(\ref{ojv}) to $\,v=\navp\,$ and $\,\phi\,$ rather than $\,f\nnh$, we obtain 
$\,\hg(\hn v,\,\cdot\,)=g(\hn v,\,\cdot\,)-2\hh\omega\hh(J\nh v,\,\cdot\,)
=d\vt+d(d_v\w\phi)$. (Note that $\,J\nh v=u$ and (\ref{gvv}) gives 
$\,d_u\w\phi=d_u\w\vt=0$, since $\,\phi\,$ is a function of $\,\vt$.) As 
$\,d_v\w\vt=Q$, cf.\ (\ref{dvt}), $\,v\,$ is thus the $\,\hg$-gra\-di\-ent 
of $\,\vt+d_v\w\phi=\vt\hs+\hh Q\hh\phi'\nh=\hs\hat\vt$. On the other hand, 
again from (\ref{dvt}), $\,\hat Q=\hg(v,v)\,$ equals 
$\,d_v\w\hat\vt=\hat\vt'd_v\w\vt=\hat\vt'Q$, which is a function of $\,\vt$, 
and of $\,\hat\vt$, proving both (a) (see Lemma~\ref{ggqft}) and (b).
\end{proof}
\begin{remark}\label{alltq}Let $\,G\,$ be the group of all auto\-mor\-phisms 
(Definition~\ref{ggktr}) of a given compact ge\-o\-des\-ic-gra\-di\-ent 
K\"ah\-ler triple $\,(\mf\nh,g,\vt)$. Then every quadruple 
$\,\tm,\tp,a$, $\hat\vt\mapsto\hat Q\,$ satisfying the analog of (\ref{pbd}) 
arises when Remark~\ref{ascdt}(i) is applied to a suitably chosen 
$\,G$-in\-var\-i\-ant ge\-o\-des\-ic-gra\-di\-ent K\"ah\-ler triple 
$\,(\mf,\hg,\hat\vt)\,$ with the same underlying complex manifold $\,\mf\nh$.

In fact, a trivial modification (see Remark~\ref{trivl}) followed by rescaling 
of the metric allows us to assume that $\,\tpm$ and $\,a\,$ are the same as 
those for $\,(\mf\nh,g,\vt)$. Our claim is now obvious from Remark~\ref{realz} 
and Theorem~\ref{ntrmd}.
\end{remark}
\begin{remark}\label{spcas}As a special case of Remark~\ref{alltq}, for the 
first triple using the Fu\-bi\-ni-Stu\-dy metric $\,g\,$ and $\,G\,$ as in the 
lines preceding (\ref{dta}), all quadruples $\,\tm,\tp,a,\hs\vt\mapsto Q$ 
with (\ref{pbd}) are realized, via Remark~\ref{ascdt}(i), by \,CP triples 
$\,(\mf\nh,g,\vt)\,$ having arbitrarily fixed values of $\,m=\dimc\nh\mf\,$ 
and $\,d_\pm\w\nh=\dimc\nh\sd^\pm$ that satisfy (\ref{dem}). 
\end{remark}
\begin{remark}\label{cnvrs}Conversely, we can apply Remark~\ref{realz} and 
Theorem~\ref{ntrmd} to canonically modify any given \,CP triple, obtaining one 
with the Fu\-bi\-ni-Stu\-dy metric and the same group $\,G$.

We will not use the eas\-i\-ly-ver\-i\-fied fact that, for such a 
Fu\-bi\-ni-Stu\-dy \,CP triple, 
$\,(\tp\hskip-2.2pt-\nh\tm)\hh Q=2a\hh(\vt-\hn\tpm)(\tp\hskip-2.3pt-\vt)\,$ 
and, in (\ref{dta}.ii), the value of $\,\vt\,$ at $\,\bbC(\xi+\eta)$, where 
$\,\xi\in\ls\,$ and $\,\eta\in\ls\hskip-2.5pt^\perp\nh$, equals 
$\,(\tpm|\hh\xi|^2\nh+\tmp|\eta|^2)/(|\hh\xi|^2\nh+|\eta|^2)\,$ for some sign 
$\,\pm\hs$.
\end{remark}

\section{The nor\-mal-ge\-o\-des\-ic bi\-hol\-o\-mor\-phisms\done}\label{ng}
\setcounter{equation}{0}
In this section $\,(\mf\nh,g,\vt)\,$ is a fixed compact 
ge\-o\-des\-\hbox{ic\hs-}\hskip0ptgra\-di\-ent K\"ah\-ler triple 
(Definition~\ref{ggktr}). We use the notation of (\ref{sym}), denote by 
$\,\tm,\tp,a,Q\,$ the data (\ref{pbd}) associated with $\,(\mf\nh,g,\vt)\,$ 
(see Remark~\ref{ascdt}(i)), and choose for them the further data (\ref{sgn}) 
-- (\ref{int}), so that a sign $\,\pm\,$ is fixed as well. We also let 
$\,\sd,\nr\nh,h,\lr\,$ and $\,\mathrm{D}\,$ stand for $\,\sd^\pm\nnh$, the 
normal bundle $\,\nr\hskip-2.3pt\sd^\pm\nnh$, the sub\-man\-i\-fold metric of 
$\,\sd$, the Riemannian fibre metric in $\,\nr\,$ induced by $\,g$, and 
the Chern connection of $\,\lr\,$ in $\,\nr\nnh$, cf.\ (d) in 
Section~\ref{ck}. We write $\,(y,\xi)\in\nr\,$ when $\,y\in\sd\,$ and 
$\,\xi\in\nr\hskip-2.4pt_y\w$, as in Remark~\ref{tlspc}.

Using the normal exponential dif\-feo\-mor\-phism 
$\,\mathrm{Exp}^\perp\nnh:\nr^\delta\nnh\sd^\pm\nh
\to\mf\nnh\smallsetminus\nh\sd^\mp$ in (\ref{nex}), we define 
$\,\varPhi=\hs\varPhi^\pm\nnh:\nr\to\mf\nnh\smallsetminus\nh\sd^\mp\nh$, 
depending on the sign $\,\pm\hs$, to be the composite
\begin{equation}\label{phe}
\varPhi\,\,=\,\,\mathrm{Exp}^\perp\nnh\circ\da\hh,
\end{equation}
where $\,\da:\nr\to\nr^\delta\nnh\sd^\pm$ is given by $\,\da(y,\xi)=y\,$ if 
$\,\xi=0\,$ and, otherwise,
\begin{equation}\label{dyx}
\begin{array}{l}
\da(y,\xi)\hs=\hs(y,t\hh\xi)\mathrm{,\ where\ 
}\,t\hs=\hs\sigma\nnh/\nnh\hr\,\mathrm{\ for\ }\,\hr\hs=\hh|\hs\xi|\mathrm{,\ 
the\ function\ }\,\hs\sigma\\\mathrm{of\ the\ variable\ 
}\,\hr\in[\hs0,\infty)\,\mathrm{\ being\ chosen\ as\ above,\ with\ 
(\ref{dsr}).}
\end{array}
\end{equation}
Note that $\,\da\,$ is a homeo\-mor\-phism and, restricted to the complement 
$\,\nr'\nh=\nr\nh\smallsetminus\sd$ of the zero section, it becomes a 
dif\-feo\-mor\-phism 
$\,\nr'\nh\to\nr^\delta\nnh\sd^\pm\nnh\smallsetminus\sd^\pm\nh$. In fact, 
$\hs t\hh\xi\hs$ with $\,t=\hs\sigma\nnh/\nnh\hr\,$ determines $\,\xi\,$ 
(smoothly if $\,\xi\ne0$), since $\,|\hh t\hh\xi|=\sigma\,$ and $\,\sigma\,$ 
determines $\,\hr$ according to the line preceding (\ref{dsr}). 
Consequently, $\,\varPhi:\nr\to\mf\nnh\smallsetminus\nh\sd^\mp$ is a 
homeo\-mor\-phism, and the restriction $\,\varPhi:\nr'\nh\to\mf'$ a 
dif\-feo\-mor\-phism. In addition,
\begin{equation}\label{pcf}
\begin{array}{l}
\pi^\pm\nnh\circ\varPhi^\pm\mathrm{\ equals\ the\ 
nor\-mal}\hyp\mathrm{bun\-dle\ projection\ 
}\,\nr\hskip-2.3pt\sd^\pm\nnh\to\sd^\pm
\end{array}
\end{equation}
due to (\ref{phe}), the fibre\hh-\hn preserving property of $\,\da$, and 
the first line of Remark~\ref{ptnrs}.
\begin{remark}\label{sends}Suppose that a vector field $\,w\,$ on $\,\nr'$ is
\begin{enumerate}
  \def\theenumi{{\rm\alph{enumi}}}
\item[{\rm(a)}] the $\,\mathrm{D}\hh$-hor\-i\-zon\-tal lift of a vector 
field on $\,\sd$, or
\item[{\rm(b)}] a vertical vector field of the form 
$\,(y,\xi)\mapsto\varTheta\xi\,$ for some com\-plex-lin\-e\-ar 
vec\-tor-bun\-dle morphism $\,\varTheta:\nr\to\nr\nnh$, skew-ad\-joint 
relative to $\,\lr\,$ at every point.
\end{enumerate}
Then $\,\da$, restricted to $\,\nr'\nh$, sends $\,w\,$ onto its restriction to 
$\,\nr'\cap\nr^\delta\nnh\sd^\pm\nh.$

In fact, let $\,r\mapsto(y(\hn r\nh),\xi(\hn r\nh))\,$ be an integral curve of 
$\,w$. Then the function $\,r\mapsto|\hh\xi(\hn r\nh)\hn|$ is constant, and 
so, by (\ref{dyx}), 
$\,\da(y(\hn r\nh),\xi(\hn r\nh))=(y(\hn r\nh),c\hs\xi(\hn r\nh))\,$ with some 
real constant $\,c$. This proves our claim since, in case (b), $\,w\,$ 
restricted to every fibre $\,\nr\hskip-2.4pt_y\w$, being a linear vector field 
on $\,\nr\hskip-2.4pt_y\w$, is invariant under multiplications by scalars.
\end{remark}
\begin{theorem}\label{first}{\medit 
For either critical manifold\/ $\,\sd^\mp$ of\/ $\,\vt\,$ in any compact 
ge\-o\-des\-ic-gra\-di\-ent K\"ah\-ler triple\/ $\,(\mf\nh,g,\vt)$, the 
triple\/ $\,(\mf\nnh\smallsetminus\nh\sd^\mp\nnh,g,\vt)\,$ is isomorphic to 
one constructed in Section\/~{\rm\ref{ev}} \hskip.8ptfrom some data\/ 
{\rm(\ref{pbd})} -- {\rm(\ref{sgn})} and\/ $\,\sd,h,\nr\nh,\lr$.

The data consist of\/ {\rm(\ref{pbd})} associated with\/ $\,(\mf\nh,g,\vt)\,$ 
as in Remark\/~{\rm\ref{ascdt}(i)}, any choice of\/ $\,\vt\mapsto\hr\,$ 
with\/ {\rm(\ref{sgn})} for\/ {\rm(\ref{pbd})} and our fixed sign\/ 
$\,\pm\hs$, the sub\-man\-i\-fold metric\/ $\,h\,$ and normal bundle\/ 
$\,\nr=\nr\hskip-2.3pt\sd^\pm\nh$ of\/ $\,\sd=\sd^\pm\nnh$, and the 
fibre metric\/ $\,\lr\,$ in\/ $\,\nr\hs$ induced by\/ $\,g$. Furthermore,
\begin{enumerate}
  \def\theenumi{{\rm\roman{enumi}}}
\item[{\rm(i)}] the required isomorphism\/ 
$\,\nr\nh\to\mf\nnh\smallsetminus\nh\sd^\mp$ is provided by the mapping\/ 
$\,\varPhi=\hs\varPhi^\pm$ with\/ {\rm(\ref{phe})}, which, in particular, must 
be bi\-hol\-o\-mor\-phic,
\item[{\rm(ii)}] $\varPhi\,$ sends the horizontal distribution of the Chern 
connection\/ $\,\mathrm{D}\,$ of\/ $\,\lr\,$ in\/ $\,\nr\nnh$, cf.\ {\rm(d)} 
of Section\/~{\rm\ref{ck}}, onto the summand\/ 
$\,\mathcal{V}\hs\oplus\mathcal{H}^\pm$ in\/ {\rm(\ref{tme})},
\item[{\rm(iii)}] the leaves of\/ $\,\hs\mathcal{V}\,$ are precisely the same 
as the\/ $\,\varPhi\hh$-im\-a\-ges of all punctured complex lines through\/ 
$\,0\,$ in the normal spaces of\/ $\,\sd$.
\end{enumerate}
In the special case where\/ 
$\,T\nnh\mf'\,=\,\,\mathcal{V}\hskip1.2pt\oplus\hs\mathcal{H}^+\nnh
\oplus\hs\mathcal{H}^-\nnh$, that is, the summand distribution\/ 
$\,\mathcal{H}$ \hbox{in\/ {\rm(\ref{tme})}} is\/ 
\hbox{$\,0$-}\hskip0ptdi\-men\-sion\-al, formula\/ {\rm(\ref{hgv}.c)} used in 
the construction of \,Section\/~{\rm\ref{ev}} may also be replaced by the 
following equality, using the simplified notation of\/ 
{\rm(\ref{hgv}.c):} 
\begin{equation}\label{alt}
\hg(\hn w,w\hh'\hh)\,=\,\frac{|\hh\vt-\hn\tmp|}{\tp\hskip-2.2pt-\nh\tm}\,h(\hn w,w\hh'\hh)\hh.
\end{equation}
}
\end{theorem}
\begin{proof}It suffices to prove that the restriction of $\,\varPhi\,$ to 
$\,\nr'\nh=\nr\nh\smallsetminus\sd\,$ is an isomorphism between the 
ge\-o\-des\-ic-gra\-di\-ent K\"ah\-ler triples $\,(\nr'\nnh,\hg,\hat\vt)\,$ 
and $\,(\mf'\nnh,g,\vt)$, since the analogous conclusion about $\,\varPhi\,$ 
itself then follows from \cite[Lemma 16.1]{derdzinski-maschler-06}.

We start by  establishing the equality
\begin{equation}\label{vcf}
\vt\circ\varPhi\,=\,\hat\vt\hh.
\end{equation}
Namely, $\,|\hh\hr\hh\xi|=\hr\,$ for any $\,\hr\in(0,\infty)\,$ and any 
$\,(y,\xi)\in\nr$ with $\,|\hs\xi|=1$, so that 
$\,\varPhi(y,\hr\hh\xi)=x\nh_\sigma\w$, where 
$\,x\nh_\sigma\w=\exp_y\w\sigma\xi\,$ and $\,\sigma\hs$ depends on $\,\hr\,$ 
as in (\ref{dsr}). Since $\,\sigma\mapsto x\nh_\sigma\w$ is a unit-speed 
geodesic, (\ref{vtg}) and (\ref{dtq}) give 
$\,d\hs[\vt(x\nh_\sigma\w)]/d\sigma=\mp\hs Q^{\hs1\nh/2}\nh$, the sign factor 
being due to the relation $\,d\hh(x\nh_\sigma\w)/d\sigma=\mp\hh v/|v|\,$ 
(immediate from (\ref{tmm}) with $\,v=\navp$). Here $\,Q=g(\hn v,v)\,$ depends 
on $\,\vt(x\nh_\sigma\w)\,$ as in Remark~\ref{ascdt}(i). However, according to 
Remark~\ref{compo} and the text preceding (\ref{hgv}.a) -- (\ref{hgv}.b), the 
same autonomous equation 
$\,d\hs[\hat\vt(y,\hr\hh\xi)]/d\sigma=\mp\hh Q^{\hs1\nh/2}$ holds when 
$\,\vt(x\nh_\sigma\w)\,$ is replaced by $\,\hat\vt(y,\hr\hh\xi)$, with {\medit 
the same} dependence of $\,Q\,$ on the unknown function. The uniqueness clause 
of Remark~\ref{compo} thus gives 
$\,\vt(\varPhi(y,\hr\hh\xi))=\vt(x\nh_\sigma\w)=\hat\vt(y,\hr\hh\xi)$, as 
required.

One has two complex di\-rect-sum decompositions, 
$\,T\nnh\mf'=\,\mathcal{V}\hskip1.2pt\oplus\hs\mathcal{H}^\mp\nnh
\oplus\hs\mathcal{H}^\bullet$ and 
$\,T\nnh\nr'=\,\hat{\mathcal{V}}\hskip1.2pt\oplus\hs\hat{\mathcal{H}}^\mp\nnh
\oplus\hs\hat{\mathcal{H}}^\bullet\nh$, orthogonal relative to $\,g\,$ and, 
respectively, $\,\hg$. The former arises from (\ref{tme}) if one sets 
$\,\mathcal{H}^\bullet\nh=\,\mathcal{H}^\pm\nnh\oplus\hs\mathcal{H}$. 
In the latter $\,\hat{\mathcal{V}}\nh,\hat{\mathcal{H}}^\mp$ and 
$\,\hat{\mathcal{H}}^\bullet$ are the distributions introduced in the lines 
following (\ref{tnp}). First, for $\,\hat u\,$ as in (\ref{lvf}) and our 
$\,u=J\nh v$, where $\,v=\navp$, we show that
\begin{equation}\label{prs}
\begin{array}{rl}
\mathrm{i)}&\da\,\mathrm{\ preserves\ 
}\,\hat{\mathcal{V}}\nh,\,\hat{\mathcal{H}}^\mp\nnh,
\hat{\mathcal{H}}^\bullet\mathrm{\ and\ }\,\hat u\hh,\\
\mathrm{ii)}&\mathrm{Exp}^\perp\mathrm{\ sends\ 
}\,\hat{\mathcal{V}}\nh,\,\hat{\mathcal{H}}^\mp\nnh,
\hat{\mathcal{H}}^\bullet\nh,\hs\hat u\,\mathrm{\ to\ 
}\,\mathcal{V}\nh,\,\mathcal{H}^\mp\nnh,\mathcal{H}^\bullet\nh,\hs u\hh,\\
\mathrm{iii)}&
\mathrm{both\ }\,\da\,\mathrm{\ and\ }\,\mathrm{Exp}^\perp\mathrm{\ act\ 
com\-plex}\hyp\mathrm{lin\-e\-ar\-ly\ on\ }\,\hat{\mathcal{H}}^\mp\mathrm{\ 
and\ }\,\hat{\mathcal{H}}^\bullet\nnh.
\end{array}
\end{equation}
More precisely, $\,\da\,$ (or, $\,\mathrm{Exp}^\perp$) appearing in 
(\ref{dyx}) (or, (\ref{nex})), restricted to $\,\nr'$ (or, 
$\,\nr'\cap\nr^\delta\nnh\sd^\pm$), sends 
$\,\hat{\mathcal{V}}\nh,\,\hat{\mathcal{H}}^\mp\nnh,
\hat{\mathcal{H}}^\bullet\nh,\hs\hat u\,$ onto their restrictions to 
$\,\nr'\cap\nr^\delta\nnh\sd^\pm$ (or, respectively, onto 
$\,\mathcal{V}\nh,\,\mathcal{H}^\mp\nnh,\mathcal{H}^\bullet\nh,\hs u$). The 
claims about $\,\hat{\mathcal{V}}\,$ in (\ref{prs}.i) -- (\ref{prs}.ii) follow 
as $\,\da\,$ clearly preserves each leaf of $\,\hat{\mathcal{V}}\nh$, that is, 
each punctured complex line through $\,0\,$ in the normal space 
$\,\nr\hskip-2.4pt_y\w\sd\,$ at any point $\,y\in\sd$, while, by 
Lemma~\ref{vsbkr}(a), $\,\mathrm{Exp}^\perp$ maps the leaves of 
$\,\hat{\mathcal{V}}\,$ intersected with $\,\nr'\nnh\cap\nr^\delta\nnh\sd^\pm$ 
onto leaves of $\,\mathcal{V}\nh$. This also proves (ii). Next, the 
class of vertical vector fields of Remark~\ref{sends}(b) obviously includes 
$\,\hat u\,$ and, locally, some of them span $\,\hat{\mathcal{H}}^\mp\nh$. 
Remark~\ref{sends} thus yields the remainder of (\ref{prs}.i), while 
(\ref{prs}.iii) for $\,\da\,$ follows from com\-plex-lin\-e\-ar\-i\-ty of the 
$\,\mathrm{D}\hh$-hor\-i\-zon\-tal lift operation (due to 
Lemma~\ref{ddrsq}(i)), and the fact that $\,\da\,$ acts on the vertical vector 
fields in Remark~\ref{sends}(b) as the identity operator. On the other hand, 
(\ref{prs}.ii) in the case of $\,\hat{\mathcal{H}}^\mp$ and 
$\,\hat{\mathcal{H}}^\bullet$ (or, of $\,\hat u$) is an immediate consequence 
of the second (or, third) claim in Theorem~\ref{dcomp}(f). (To be specific, 
for $\,\hat{\mathcal{H}}^\mp$ and $\,\hat{\mathcal{H}}^\bullet$ this is clear 
from Remark~\ref{inequ}(ii) combined with (\ref{hpm}) -- (\ref{tme}).) 
Finally, the com\-plex-lin\-e\-ar\-i\-ty assertion of Theorem~\ref{dcomp}(f) 
implies (\ref{prs}.iii).

By (\ref{prs}), the dif\-feo\-mor\-phism 
$\,\varPhi=\mathrm{Exp}^\perp\nnh\circ\da:\nr'\nh\to\mf'$ maps 
$\,\hat{\mathcal{V}}\nh,\hat{\mathcal{H}}^\mp$ and 
$\,\hat{\mathcal{H}}^\bullet$ onto $\,\mathcal{V}\nh,\mathcal{H}^\mp$ and 
$\,\mathcal{H}^\bullet\nnh$. Proving the theorem is thus reduced to showing 
that
\begin{equation}\label{mps}
\begin{array}{l}
\hat J\,\mathrm{\ and\ }\,\hg\mathrm{,\ on\ each\ of\ the\ three\ summands\ 
}\,\,\hat{\mathcal{V}}\nh,\,\hat{\mathcal{H}}^\mp\mathrm{\ and\ 
}\,\hat{\mathcal{H}}^\bullet\nnh\mathrm{,\ correspond}\\
\mathrm{under\ the\ differential\ }\,d\varPhi\,\mathrm{\ to\ }\,J\,\mathrm{\ 
and\ }\,g\,\mathrm{\ on\ 
}\,\mathcal{V}\nh,\,\mathcal{H}^\mp\mathrm{\ and\ 
}\,\mathcal{H}^\bullet\nnh\mathrm{,\ respectively\nh.}
\end{array}
\end{equation}
To begin with, for $\,\hat Q\,$ as in Section~\ref{ev}, $\,\hat v\,$ given by 
(\ref{lvf}), and our $\,v=\navp$,
\begin{equation}\label{qav}
\varPhi\,\,\mathrm{\ pushes\ }\,\hat Q,\hat u\,\mathrm{\ and\ 
}\,\hat v\,\mathrm{\ forward\ onto\ }\,Q,u\,\mathrm{\ and\ }v\hh.
\end{equation}
In the case of $\,\hat Q\,$ this amounts to $\,Q\circ\varPhi=\hat Q$, which is 
a trivial consequence of (\ref{vcf}) and the fact that $\,\hat Q\,$ was 
defined in Section~\ref{ev} to be the same function of $\,\hat\vt\,$ as 
$\,Q\,$ is of $\,\vt$. For $\,\hat u$, (\ref{qav}) follows from (\ref{prs}) 
and (\ref{phe}). Next, any integral curve of $\,\hat v\,$ in 
$\,\nr\hskip-2.4pt_y\w\smallsetminus\{0\}\,$ has, up to a shift of the 
parameter, the form $\,r\mapsto(y,e^{\mp\hh ar}\xi)$ with a unit vector 
$\,\xi\in\nr\hskip-2.4pt_y\w$, so that 
$\,\da(y,e^{\mp\hh ar}\xi)=(y,\sigma\hs\xi)$, where in addition to the curve 
parameter $\,r\nh$, two more real variables are used: 
$\,\hr=e^{\mp\hs ar}\nnh$, 
and $\,\sigma\,$ related to $\,\hr\,$ via (\ref{dsr}). The chain rule thus 
yields 
$\,d\sigma/dr=\mp\hh a\hr\,d\sigma\nh/\hn d\nh\hr=\mp\hh Q^{\hs1\nh/2}\nh$, 
while $\,\varPhi(y,e^{\mp\hh ar}\xi)=x(\nh\sigma\nh)\,$ for 
$\,x(\nh\sigma\nh)=\exp_y\w\sigma\xi$. Since 
$\,\sigma\mapsto x(\nh\sigma\nh)\,$ is a unit-speed geodesic, (\ref{vtg}) and 
(\ref{dxe}) give $\,d\hs[x(\nh\sigma\nh)]/d\sigma=\mp\hs Q^{\hs1\nh/2}\nh$, 
with $\,Q\,$ evaluated at $\,x(\nh\sigma\nh)$, and the sign factor arising 
from (\ref{tmm}), as $\,v=\navp$. Applying the chain rule again, we obtain 
$\,d\hs[x(\nh\sigma\nh)]/dr=v\nh_{x(\nh\sigma\nh)}\w$ and, consequently, 
(\ref{qav}).

The claim made in (\ref{mps}) about 
$\,\hat{\mathcal{V}}=\mathrm{Span}\hs(\hat v,\hat u)\,$ and 
$\,\mathcal{V}=\mathrm{Span}\hs(\hn v,u)\,$ is now obvious from (\ref{qav}) 
and (\ref{gvv}) along with (\ref{ana}). 

For the remaining two pairs of summands, (\ref{mps}) in the case of 
$\,\hat J,J\,$ (or, $\,\hg,g$)  is a direct consequence of (\ref{prs}) and 
(\ref{phe}) (or, respectively, of (i) -- (ii) in Remark~\ref{inequ} along with 
parts (h2) -- (h3) of Theorem~\ref{dcomp}, (\ref{vcf}) and (\ref{hgv})). Note 
that, by (\ref{dyx}), $\,\da\,$ leaves $\,\xi/\hn|\hh\xi|\,$ unchanged, while 
$\,\hr=|\hh\xi|\,$ in (h2).

Finally, if $\,T\nnh\mf'=\,\mathcal{V}\hskip1.2pt\oplus\hs\mathcal{H}^+\nnh
\oplus\hs\mathcal{H}^-$ in (\ref{tme}), Remark~\ref{inequ}(ii) allows us to 
use (h1) in Theorem~\ref{dcomp}, instead of (h2), obtaining (\ref{alt}).
\end{proof}
\begin{corollary}\label{totgd}{\medit 
Suppose that\/ $\,(\mf\nh,g,\vt)\,$ is a compact ge\-o\-des\-ic-gra\-di\-ent 
K\"ah\-ler triple. Then, for $\,\mathcal{V}\,$ and\/ $\,\mathcal{H}^\pm$ 
appearing in\/ {\rm(\ref{tme})}, with either sign\/ $\,\pm\hs$, the 
distribution\/ $\,\mathcal{V}\oplus\mathcal{H}^\pm$ is in\-te\-gra\-ble and 
its leaves are totally geodesic in\/ $\,(\mf'\nh,g)$.
}
\end{corollary}
\begin{proof}Use Theorem~\ref{first} and Theorem~\ref{ehggk}(b) (or -- for 
in\-te\-gra\-bil\-i\-ty -- (\ref{hpm})).
\end{proof}

\section{Immersions of complex projective spaces\done}\label{ic}
\setcounter{equation}{0}
In the next result the inclusions 
$\,\nr\hskip-2.4pt_y\w
\subseteq\mathrm{P}(\bbC\times\nnh\nr\hskip-2.4pt_y\w)\,$ 
and $\,\mathrm{P}\nnh\nr\hskip-2.4pt_y\w
\subseteq\mathrm{P}(\bbC\times\nnh\nr\hskip-2.4pt_y\w)\,$ 
come from the standard identification (\ref{inc}) for 
$\,\vs=\nr\hskip-2.4pt_y\w$, where $\,y\in\sd^\pm\nh$. Let us also note 
that, by (\ref{hpm}) and Corollary~\ref{totgd}, the restriction to the normal 
space 
$\,\nr\hskip-2.4pt_y\w=\nr\hskip-2.4pt_y\w\sd^\pm\nh
\subseteq\nr\hskip-2.3pt\sd^\pm$ 
of the bi\-hol\-o\-mor\-phism 
$\,\varPhi:\nr\hskip-2.3pt\sd^\pm\nnh\to\mf\nnh\smallsetminus\nh\sd^\mp$ (see 
Theorem~\ref{first}) constitutes
\begin{equation}\label{tgh}
\mathrm{a\ totally\ geodesic\ hol\-o\-mor\-phic\ embedding\ 
}\,\,\,\varPhi:\nr\hskip-2.4pt_y\w\to\mf\nnh\smallsetminus\nh\sd^\mp\nh.
\end{equation}
\begin{theorem}\label{tgimm}{\medit 
Given a compact ge\-o\-des\-ic-gra\-di\-ent K\"ah\-ler triple\/ 
$\,(\mf\nh,g,\vt)\,$ and a fixed sign\/ $\,\pm\hs$, let\/ $\,y\,$ be a point 
of the critical manifold $\,\sd^\pm\nh$. Then the following conclusions hold.
\begin{enumerate}
  \def\theenumi{{\rm\alph{enumi}}}
\item[{\rm(a)}] The embedding\/  
$\,\varPhi:\nr\hskip-2.4pt_y\w\to\mf\nnh\smallsetminus\nh\sd^\mp\,$ with\/ 
{\rm(\ref{tgh})} has an extension to a totally geodesic hol\-o\-mor\-phic 
immersion\/ $\,\vps:\mathrm{P}(\bbC\times\nnh\nr\hskip-2.4pt_y\w)\to\mf\nh$.
\item[{\rm(b)}] The mapping\/ $\,\vps\,$ in\/ {\rm(a)} restricted to the 
projective hyper\-plane\/ $\,\mathrm{P}\nnh\nr\hskip-2.4pt_y\w
\subseteq\mathrm{P}(\bbC\times\nnh\nr\hskip-2.4pt_y\w)$ at infinity is a totally 
geodesic hol\-o\-mor\-phic immersion\/ 
$\,F:\mathrm{P}\nnh\nr\hskip-2.4pt_y\w\to\sd^\mp\nh$, and the metric that it 
induces on\/ $\,\mathrm{P}\nnh\nr\hskip-2.4pt_y\w$ equals\/ 
$\,2(\tp\hskip-2.2pt-\nh\tm)/\hn a\,$ times the Fu\-bi\-ni-Stu\-dy metric, 
cf.\ Remark\/~{\rm\ref{fbstm}}, arising from the inner product\/ $\,g_y\w$ 
in\/ $\,\nr\hskip-2.4pt_y\w\hh$, \,for\/ $\,a,\tpm$ as in 
Remark\/~{\rm\ref{ascdt}(i)},
\item[{\rm(c)}] the images of the immersion\/ 
$\,F:\mathrm{P}\nnh\nr\hskip-2.4pt_y\w\to\sd^\mp$ in\/ {\rm(b)} and of its 
differential at any point\/ $\,\bbC\xi$, where\/ 
$\,(y,\xi)\in\nr\hskip-2.3pt\sd^\pm$ and\/ $\,\xi\ne0$, coincide with the\/ 
$\,\pi^\mp\hskip-2pt$-im\-a\-ge of the leaf of\/ 
$\,\mathrm{Ker}\hskip2.3ptd\pi^{\hs\pm}$ in\/ $\,\mf'\nh$  passing through\/ 
$\hs\,x=\varPhi(y,\xi)\,$ and, respectively, with the subspace\/ 
$\,d\pi\nh_x^\mp\nh(\mathcal{H}_x^\mp)
=d\pi\nh_x^\mp\nh(\mathcal{V}_{\nh x}\w\oplus\mathcal{H}_x^\mp)\,$ of\/ 
$\,\tyb^\mp\nh$. 
\end{enumerate}
}
\end{theorem}
\begin{proof}As a consequence of Theorems~\ref{first} and~\ref{dcomp}(c), the 
composite $\,\pi^\mp\nnh\circ\varPhi\,$ maps
$\,\nr\hskip-2.3pt\sd^\pm\nnh\smallsetminus\nh\sd^\pm$ (the complement of the
zero section in $\,\nr\hskip-2.3pt\sd^\pm$) hol\-o\-mor\-phic\-al\-ly into
$\,\sd^\mp\nh$. The restriction of $\,\pi^\mp\nnh\circ\varPhi\,$ to
$\,\nr\hskip-2.4pt_y\w\smallsetminus\{0\}
\subseteq\nr\hskip-2.3pt\sd^\pm\nnh\smallsetminus\nh\sd^\pm\nh$, being, by (\ref{pex}) 
and (\ref{phe}), constant on each punctured complex line through $\,0$, thus 
descends to
\begin{equation}\label{dsc}
\mathrm{a\ hol\-o\-mor\-phic\ immersion\ 
}\,\,\,F:\mathrm{P}\nnh\nr\hskip-2.4pt_y\w\to\sd^\mp,
\end{equation}
where the immersion property of $\,F\,$ is an immediate consequence of the 
fact, established below, that both 
$\,\pi^\mp\nnh:\varPhi\hs(\nh\nr\hskip-2.4pt_y\w\smallsetminus\{0\}\nnh)
\to\sd^\mp$ and 
$\,\pi^\mp\nnh\circ\varPhi:\nr\hskip-2.4pt_y\w\smallsetminus\{0\}\to\sd^\mp$ 
have constant (complex) rank, equal to 
$\,\dimc\nh\nr\hskip-2.4pt_y\w\hs-\hs1$. As $\,\varPhi\,$ is a 
bi\-hol\-o\-mor\-phism, it suffices to verify this last claim for the former 
mapping; we do it noting that 
$\,\varPi=\varPhi\hs(\nh\nr\hskip-2.4pt_y\w\smallsetminus\{0\}\nnh)$ 
coincides with the $\,\pi^\pm\nh$-pre\-im\-age of $\,y\,$ (due to (\ref{pcf}) 
and Remark~\ref{ptnrs}), and hence forms a leaf of 
$\,\mathrm{Ker}\hskip2.3ptd\pi^\pm\nh
=\hs\mathcal{V}\hs\oplus\mathcal{H}^\mp$ restricted to $\,\mf'\nh$, cf.\ 
(\ref{hpm}). That $\,\pi^\mp:\varPi\to\sd^\mp$ satisfies the required 
rank condition is now clear: the kernel of its differential at any point 
$\,x\,$ coincides, by (\ref{hpm}) and (\ref{tme}), with 
$\,\mathcal{V}_{\nh x}\w$, while $\,\mathcal{V}=\mathrm{Span}\hs(\hn v,u)$.

The mapping $\,\vps:\mathrm{P}(\bbC\times\nnh\nr\hskip-2.4pt_y\w)\to\mf\nh$, 
equal to $\,\varPhi\,$ on $\,\nr\hskip-2.4pt_y\w$ and to $\,F\,$ on 
$\,\mathrm{P}\nnh\nr\hskip-2.4pt_y\w$, is continuous. Namely, if it were not, 
we could pick a sequence 
$\,\xi\nh_j\w\in\nr\hskip-2.4pt_y\w$, $\,j=1,2,\dots\hs$, such that 
$\,|\hs\xi\nh_j\w|\to\infty\,$ and $\,\xi\nh_j\w/|\hs\xi\nh_j\w|\to\xi\,$ as 
$\,j\to\infty\,$ for some unit vector $\,\xi\in\nr\hskip-2.4pt_y\w$, while no 
sub\-se\-quence of the image sequence $\,\vps(\xi\nh_j\w)\,$ tends to 
$\,F(\bbC\xi)$. The resulting limit relation 
$\,\sigma\hskip-3pt_j\w\to\delta$, where $\,\sigma\hskip-3pt_j\w$ corresponds 
to $\,\hr\nnh_j\w=|\hs\xi\nh_j\w|\,$ as in the line preceding (\ref{dsr}), 
combined with (\ref{phe}), now gives 
$\,\vps(\xi\nh_j\w)=\varPhi(\xi\nh_j\w)
=\mathrm{Exp}^\perp\nh(y,
\sigma\hskip-3pt_j\w\hs\xi\nh_j\w\nh/\hskip-2.3pt\hr\nnh_j\w\nnh)$ 
which -- due to continuity of $\,\mathrm{Exp}^\perp$ and (\ref{pex}) -- 
converges to $\,\mathrm{Exp}^\perp\nh(y,\delta\xi)=y_\mp\w$, for a specific 
point $\,y_\mp\w$. However, (a) -- (b) in Lemma~\ref{vsbkr} and the definition 
of $\,F\,$ also give $\,y_\mp\w=F(\bbC\xi)$, which contradicts our choice of 
$\,\xi\nh_j\w$, proving continuity of $\,\vps$.

Hol\-o\-mor\-phic\-i\-ty of $\,\vps\,$ is now obvious from Remark~\ref{hlext} 
applied to $\,\varPi=\mathrm{P}(\bbC\times\nnh\nr\hskip-2.4pt_y\w)$ and its 
co\-di\-men\-sion-one complex sub\-man\-i\-fold
$\,\varLambda=\mathrm{P}\nnh\nr\hskip-2.4pt_y\w$. Furthermore,
\begin{equation}\label{imm}
\vps\,\,\,\mathrm{is\ an\ immersion}.
\end{equation}
To see this, first note that $\,\vps\,$ has two restrictions, $\,F\,$ to 
$\,\mathrm{P}\nnh\nr\hskip-2.4pt_y\w$ and $\,\varPhi\,$ to the dense open 
sub\-man\-i\-fold $\,\nr\hskip-2.4pt_y\w$, already known to be immersions, 
the former into $\,\sd^\mp\nnh$, cf.\ (\ref{tgh})  -- (\ref{dsc}). Next, for 
any unit vector $\,\xi\in\nr\hskip-2.4pt_y\w$, if $\,\varLambda\hn'$ denotes 
the projective line in 
$\,\mathrm{P}(\bbC\times\nnh\nr\hskip-2.4pt_y\w)\,$ joining $\,\bbC(1,0)\,$ to the 
point $\,\bbC\xi\in\mathrm{P}\nnh\nr\hskip-2.4pt_y\w$ (identified via 
(\ref{inc}) with $\,\bbC(0,\xi)$), then the restriction of $\,\vps\,$ to 
$\,\varLambda\hn'$ is an embedding with the image 
$\,\varLambda=\vps(\varLambda\hn'\hh)\,$ forming a complex sub\-man\-i\-fold 
of $\,\mf\nh$, bi\-hol\-o\-mor\-phic to $\,\bbCP^1\nh$, and intersecting each 
of $\,\sd^+$ and $\,\sd^-$ orthogonally at a single point. In fact, 
Lemma~\ref{vsbkr} yields all the claims just made except the `embedding' 
property; we obtain the latter from Remark~\ref{known}(b), which we use to 
conclude that the resulting hol\-o\-mor\-phic mapping 
$\,\vps\nnh:\varLambda\hn'\hn\to\varLambda$, being injective (since so is 
$\,\varPhi$), must be a bi\-hol\-o\-mor\-phism. Now (\ref{imm}) follows.

For obvious reasons of continuity, (\ref{tgh}) implies that the 
hol\-o\-mor\-phic immersion 
$\,\vps:\mathrm{P}(\bbC\times\nnh\nr\hskip-2.4pt_y\w)\to\mf\,$ is totally 
geodesic, which establishes (a). Finally, Remarks~\ref{fbstu}, 
\ref{ascdt}(iii), \ref{ttgim} and Theorem~\ref{first} give rise to (b), 
completing the proof.
\end{proof}
\begin{remark}\label{codim}For $\,m,d_\pm\w,k_\pm\w,q\,$ as in 
Remark~\ref{dppdm}, the co\-di\-men\-sion 
$\,\dimc\nh\sd^\mp\nnh-\hh\dimc\nh\nr\hskip-2.4pt_y\w$ of the immersion 
$\,F\,$ in Theorem~\ref{tgimm}(b) equals $\,q$. In fact, 
$\,\dimc\nh\nr\hskip-2.4pt_y\w=m-d_\pm\w\nh-1$, and so, by (\ref{dmq}), 
$\,\dimc\nh\sd^\mp\nh-\dimc\nh\nr\hskip-2.4pt_y\w=(m-d_\pm\w\nh-1)-d_\mp\w=q$.
\end{remark}
\begin{remark}\label{spmcp}Suppose that the distribution $\,\mathcal{H}\,$ in 
(\ref{tme}) is \hbox{$\,0$-}\hskip0ptdi\-men\-sion\-al or, in other words, 
$\,T\nnh\mf'\,=\,\,\mathcal{V}\hskip1.2pt\oplus\hs\mathcal{H}^+\nnh
\oplus\hs\mathcal{H}^-\nnh$. Then, for either sign $\,\pm\hs$, the critical 
manifold $\,\sd^\pm\nh$, with its sub\-man\-i\-fold metric, must be 
bi\-hol\-o\-mor\-phic\-al\-ly isometric to a complex projective space carrying 
the Fu\-bi\-ni-Stu\-dy metric multiplied by 
$\,2(\tp\hskip-2.2pt-\nh\tm)/\hn a\hh$.

In fact, the isometric immersion $\,F\,$ of Theorem~\ref{tgimm}(b), having 
co\-di\-men\-sion zero (cf.\ Remark~\ref{codim}), is necessarily a 
bi\-hol\-o\-mor\-phism (Remark~\ref{ftcvr}).
\end{remark}

\section{Consequences of condition {\rm(\ref{spn})}\done}\label{ch}
\setcounter{equation}{0}
The results stated and proved below use Definition~\ref{ggktr}, the notations 
of (\ref{sym}), (\ref{dbp}), (\ref{hpm}), and the notion of projectability 
introduced in Section~\ref{pd}.
\begin{lemma}\label{intpr}{\medit 
For a compact ge\-o\-des\-ic-gra\-di\-ent K\"ah\-ler triple\/ 
$\,(\mf\nh,g,\vt)$, the following three conditions are mutually equivalent.
\begin{enumerate}
  \def\theenumi{{\rm\alph{enumi}}}
\item[{\rm(i)}] The distribution\/ 
$\,\hs\zy=\mathcal{V}\hs\oplus\mathcal{H}^+\hskip-3pt\oplus\mathcal{H}^-$ on\/ 
$\,\mf'\hh$ is in\-te\-gra\-ble.
\item[{\rm(ii)}] $\mathrm{Ker}\hskip2.3ptd\pi^-\hs
=\,\hs\mathcal{V}\hs\oplus\mathcal{H}^+\,$ is\/ 
$\,\pi^+\hskip-2pt$-pro\-ject\-able.
\item[{\rm(iii)}] $\mathrm{Ker}\hskip2.3ptd\pi^+\hs
=\,\hs\mathcal{V}\hs\oplus\mathcal{H}^-\,$ is\/ 
$\,\pi^-\hskip-2pt$-pro\-ject\-able.
\end{enumerate}
In\/ {\rm(ii)} -- {\rm(iii)} one may also replace\/ 
$\,\mathcal{V}\hs\oplus\mathcal{H}^\pm\,$ by\/ $\,\mathcal{H}^\pm$ or\/ 
$\,\zy$. If\/ {\rm(i)} -- {\rm(iii)} hold, then\/{\rm:}
\begin{enumerate}
  \def\theenumi{{\rm\alph{enumi}}}
\item[{\rm(iv)}] The immersions of Theorem\/~{\rm\ref{tgimm}(c)} are all 
embeddings.
\item[{\rm(v)}] The\/ $\,\pi^\pm\hskip-1.2pt$-im\-a\-ges\/ $\,\zy^\pm\nnh$ of 
the in\-te\-gra\-ble distribution\/ $\,\hs\zy\,$ on\/ $\,\mf'$ are 
in\-te\-gra\-ble hol\-o\-mor\-phic distributions on\/ $\,\sd^\pm$ and have 
totally geodesic leaves bi\-hol\-o\-mor\-phic\-al\-ly isometric to complex 
projective spaces carrying\/ $\,2(\tp\hskip-2.2pt-\nh\tm)/\hn a\,$ times the 
Fu\-bi\-ni-Stu\-dy metric, cf.\ Theorem\/~{\rm\ref{tgimm}(b)}. These leaves 
coincide with the images of the embeddings in\/ {\rm(iv)}, and form the fibres 
of hol\-o\-mor\-phic bundle projections\/ 
$\,\mathrm{pr}^\pm\nnh:\sd^\pm\nnh\to\sb^\pm\,$ for some compact complex base 
manifolds\/ $\,\sb^\pm\nnh$.
\item[{\rm(vi)}] The summand\/ $\,\mathcal{H}\,$ in\/ {\rm(\ref{tme})} is\/ 
$\,\pi^\pm\hskip-1.2pt$-pro\-ject\-able and its\/ 
$\,\pi^\pm\hskip-1.2pt$-im\-a\-ge coincides with the orthogonal complement 
of\/ $\,\zy^\pm$ in\/ $\,T\nnh\sd^\pm\nh$.
\item[{\rm(vii)}] The leaf space\/ $\,\sb=\mf'\hskip-3pt/\hskip-1.8pt\zy\,$ 
admits a unique structure of a compact complex manifold such that the 
quo\-tient projection\/ $\,\mf'\nh\to\mf'\hskip-3pt/\hskip-1.8pt\zy\,$ 
constitutes a hol\-o\-mor\-phic fibration while, for either sign\/ $\,\pm\,$ 
and\/ $\,\mathrm{pr}^\pm\nnh:\sd^\pm\nnh\to\sb^\pm\,$ as in\/ {\rm(iv)}, the 
mapping\/ $\,\sb\to\sb^\pm\nnh$, sending each leaf 
of\/ $\,\zy\,$ to its image under\/ $\,\mathrm{pr}^\pm\nnh\circ\pi^\pm\nnh$, 
is a bi\-hol\-o\-mor\-phism.
\item[{\rm(viii)}] There exists a unique hol\-o\-mor\-phic bundle projection\/ 
$\,\pi:\mf\to\sb\,$ with\/ $\,\mathrm{Ker}\hskip2.3ptd\pi=\zy$ on\/ 
$\,\mf'$ such that, for both signs\/ $\,\pm\hs$, the restriction of\/ 
$\,\pi\,$ to\/ $\,\mf'$ equals\/ 
$\,\beta^\pm\nnh\circ\mathrm{pr}^\pm\nnh\circ\pi^\pm\nh$, where\/ 
$\,\beta^\pm$ is the inverse of the bi\-hol\-o\-mor\-phism\/ 
$\,\sb\to\sb^\pm$ in\/ {\rm(vii)}.
\item[{\rm(ix)}] $R^{\mathrm{D}}\nh(\hn w,w\hh'\hh)\,
=\,-ia\hh(\tp\hskip-2.2pt-\nh\tm)^{-\nnh1}h(J\nh w,w\hh'\hh):\nr\nh\to\nr\nnh$, 
with the notation of\/ {\rm(\ref{nwt})}, for the sub\-man\-i\-fold metric\/ 
$\,h\,$ of\/ $\,\sd^\pm\nh$, the normal connection\/ $\,\mathrm{D}\,$ in its 
normal bundle\/ $\,\nr=\nr\hskip-2.3pt\sd^\pm\nh$, any vector field\/ 
$\,w\hh'$ on\/ $\,\sd^\pm\nh$, and any section\/ $\,w\,$ of\/ $\,\zy^\pm\nh$, 
cf.\ {\rm(v)}.
\end{enumerate}
}
\end{lemma}
\begin{proof}Since $\,\mathcal{V}\hs\oplus\mathcal{H}^\pm$ are both 
in\-te\-gra\-ble by (\ref{hpm}), the mutual equivalence of (i), (ii), (iii) 
and the in\-te\-gra\-bil\-i\-ty claim in (v) are all immediate from 
Lemma~\ref{intgr} applied to 
$\,\ez^\pm\nnh=\hs\mathcal{V}\hs\oplus\mathcal{H}^\pm\nnh$, along with 
(\ref{hpm}) and (\ref{tme}). The immersions mentioned in 
Theorem~\ref{tgimm}(c) thus have nonsingular images, namely, the leaves 
$\,\varPi\,$ of the distribution $\,\zy^\pm$ in (v), so that (iv) 
follows from Remark~\ref{ftcvr} applied to 
$\,\mathrm{P}\nnh\nr\hskip-2.4pt_y\w$ standing for $\,\bbCP\hh^l\nh$, with 
$\,l=k_\mp\w$ defined in Remark~\ref{dppdm}, and such a leaf $\,\varPi\nh$. 
The remaining part of (v) is a direct consequence of Theorem~\ref{tgimm}(b) 
and Remark~\ref{fibra}.

At any $\,y\in\sd^\pm\nh$, the image 
$\,d\pi\nh_x^{\hs\pm}\nh(\mathcal{H}_x^\pm\nh)$ is now independent of the choice 
of $\,x\in\mf'$ with $\,\pi^\pm\nh(x)=y$, and hence so is its orthogonal 
complement $\,d\pi\nh_x^{\hs\pm}\nh(\mathcal{H}_x\w)$ in $\,\tyb^\pm$ (see 
Remark~\ref{inequ}(iii)), proving assertion (vi).

The mappings $\,\sb\to\sb^\pm$ in (vii) are obviously bijective, and lead to 
an identification $\,\sb^+\nnh=\hs\sb^-$ which is a bi\-hol\-o\-mor\-phism, as 
one sees restricting $\,\pi^\pm$ to ``local'' complex sub\-man\-i\-folds of 
$\,\mf'$ which the composite bundle projections 
$\,\mf'\nnh\to\sd^\pm\nnh\to\sb^\pm$ (with fibres provided by the leaves of 
$\,\zy$) send bi\-hol\-o\-mor\-phic\-al\-ly onto open sub\-man\-i\-folds of 
$\,\sb^\pm\nnh$. This yields (vii). For (viii), it suffices to note that the 
two composite bundle projections 
$\,\mathrm{pr}^\pm\nh\circ\pi^\pm\nnh:\mf\nnh\smallsetminus\nh\sd^\mp\nnh
\to\sb\,$ agree, by (vii), on the intersection $\,\mf'$ of their domains, cf.\ 
Remark~\ref{ascdt}(iv), while the union of their domains is $\,\mf\nh$.

For (ix), Theorem~\ref{first} allows us to identify 
$\,\mf\nnh\smallsetminus\nh\sd^\mp$ with $\,\nr\,$ so that (\ref{hgv}.c) and 
(\ref{gsw}) hold under the assumptions following (\ref{hgv}). Since $\,w\,$ 
lies in the $\,\pi^\pm\hskip-1.2pt$-im\-a\-ge $\,\zy^\pm\nnh$ of 
$\,\mathcal{H}^\pm\nnh$, cf.\ (ii), (iii), (v), formula (\ref{hvk}) gives 
$\,2\dv\nh w=Qw/(\vt-\hn\tmp)\,$ for its $\,\mathrm{D}\hh$-hor\-i\-zon\-tal 
lift, also denoted by $\,w$. Replacing 
$\,2\dv\nh w\,$ in (\ref{gsw}) with $\,Qw/(\vt-\hn\tmp)\,$ and multiplying the 
result by $\,(\vt-\hn\tmp)\hh Q^{-\nnh1}\nh$, we get an expression for 
$\,g(\hn w,w\hh'\hh)\,$ which, equated to (\ref{hgv}.c), yields 
$\,\la\nh R^{\mathrm{D}}\nh(\hn w,J\nh w\hh'\hh)\hh\xi,i\hh\xi\ra=
-\hh a\hh(\tp\hskip-2.2pt-\nh\tm)^{-\nnh1}\la\xi,\xi\ra\hs h(\hn w,w\hh'\hh)$, 
since $\,\hr^2\nh=\la\xi,\xi\ra$ while, obviously, 
$\,|\hh\vt-\hn\tpm|=\mp\hh(\vt-\hn\tpm)$. Applying the last equality to 
$\,J\nh w\,$ instead of $\,w$, and using (b) in Section~\ref{ck} along with 
Her\-mit\-i\-an symmetry of 
$\,\la\nh R^{\mathrm{D}}\nh(\hn w,w\hh'\hh)\hh\xi,i\hh\eta\ra
=-\la i\hh R^{\mathrm{D}}\nh(\hn w,w\hh'\hh)\hh\xi,\eta\ra$ in $\,\xi,\eta$, 
we obtain the required relation in (ix).
\end{proof}

Note that the above proof of (ix) in Lemma~\ref{intpr} actually uses the 
assumptions (i) -- (iii): without them, the formula 
$\,\la\nh R^{\mathrm{D}}\nh(\hn w,J\nh w\hh'\hh)\hh\xi,i\hh\xi\ra=
-\hh a\hh(\tp\hskip-2.2pt-\nh\tm)^{-\nnh1}\la\xi,\xi\ra\hs h(\hn w,w\hh'\hh)$, 
rather than being valid for any given $\,w\in\zy_y^\pm\nh$, 
$\,y\in\sd^\pm\nh$, and {\medit all\/} vectors $\,\xi\,$ normal to $\,\sd^\pm$ 
at $\,y$, would hold only when $\,w\,$ lies in some subspace of $\,\tyb^\pm$ 
depending on $\,\xi$.

Let us now fix a K\"ah\-ler manifold $\,(\hat\sd,\hat h)$, and consider 
pairs $\,\nr\nh,\lr\,$ formed by a hol\-o\-mor\-phic complex vector bundle 
$\,\nr\,$ over $\,\hat\sd\,$ and the real part $\,\lr\,$ of a Her\-mit\-i\-an 
fibre metric in $\,\nr\nnh$, the Chern connection of which -- see 
Section~\ref{ck} -- satisfies the curvature condition 
$\,R^{\mathrm{D}}\nh(\hn w,w\hh'\hh)
=2i\hh\hat h(J\nh w,w\hh'\hh):\nr\nh\to\nr$ 
for any vector fields $\,w,w\hh'$ tangent to $\,\hat\sd$, where the notation 
of (\ref{nwt}) is used.
\begin{lemma}\label{rdehj}{\medit 
Whenever\/ $\,\hat\sd\,$ is simply connected and such\/ 
$\,\nr\nh,\lr\,$ exist, they are essentially unique, in the sense that, given 
another pair\/ $\,\nr'\nh,\lr\hn'$ with the same property, some 
hol\-o\-mor\-phic vec\-tor-bun\-dle isomorphism\/ $\,\nr\nh\to\nr'$ takes\/ 
$\,\lr\,$ to\/ $\,\lr\hn'\nh$.
}
\end{lemma}
\begin{proof}Remark~\ref{crvcm} implies that the Chern connections 
$\,\mathrm{D}\,$ and $\,\mathrm{D}'$ induce a flat metric connection in the 
bundle $\,\mathrm{Hom}\hs(\nr\nh,\nr')$. The required isomorphism is now 
provided by a global parallel section of $\,\hs\mathrm{Hom}\hs(\nr\nh,\nr')\,$ 
chosen so as to transform $\,\lr$ into $\,\lr\hn'$ at one point, and its 
hol\-o\-mor\-phic\-i\-ty follows from (e) in Section~\ref{ck}.
\end{proof}
\begin{theorem}\label{cptrp}{\medit 
For a compact ge\-o\-des\-ic-gra\-di\-ent K\"ah\-ler triple\/ $\,(\mf\nh,g,\vt)$, 
the following two conditions are equivalent. 
\begin{enumerate}
  \def\theenumi{{\rm\roman{enumi}}}
\item[{\rm(i)}] $(\mf\nh,g,\vt)\,$ is iso\-mor\-phic to a \,CP\ triple, 
defined as in Section\/~{\rm\ref{eg}}.
\item[{\rm(ii)}] $d_+\w+\hs\,d_-\w\hs=\,m\,-\,1$, where\/ $\,m=\dimc\nh\mf\,$ 
and\/ $\,d_\pm\w\nh=\dimc\nh\sd^\pm\nh$. In other words, cf.\ 
Remark\/~{\rm\ref{dppdm}}, 
$\,T\nnh\mf'\hs=\,\,\mathcal{V}\hskip1.2pt\oplus\hs\mathcal{H}^+\nnh
\oplus\hs\mathcal{H}^-\nnh$, that is, $\,\mathcal{H}\,$ in\/ {\rm(\ref{tme})} 
is\/ \hbox{$\,0$-}\hskip0ptdi\-men\-sion\-al.
\end{enumerate}
In this case, the assertion of \,Theorem\/~{\rm\ref{first}}, including\/ 
{\rm(\ref{alt})}, is satisfied by\/ 
$\,(\mf\nnh\smallsetminus\nh\sd^\mp\nnh,g,\vt)$, with either fixed sign\/ 
$\,\pm\,$ and\/ $\,(\sd,h)\,$ bi\-hol\-o\-mor\-phic\-al\-ly isometric 
to a complex projective space carrying\/ 
$\,2(\tp\hskip-2.2pt-\nh\tm)/\hn a\,$ times the Fu\-bi\-ni-Stu\-dy metric, 
$\,\nr\,$ and\/ $\,\lr\,$ being, up to a hol\-o\-mor\-phic vec\-tor-bun\-dle 
isomorphism, the normal bundle of the latter treated as a linear variety in\/ 
$\,\bbCP^m$ and its Her\-mit\-i\-an fibre metric induced by the 
Fu\-bi\-ni-Stu\-dy metric of\/ $\,\bbCP^m\nh$.

Furthermore, the isomorphism types of \,CP\ triples\/ $\,(\mf\nh,g,\vt)\,$ having 
any given values of\/ $\,d_\pm\w$ and\/ $\,m\,$ in\/ {\rm(ii)} are in a 
natural bijective correspondence, obtained by applying 
Remark\/~{\rm\ref{ascdt}(i)}, with quadruples\/ 
$\,\tm,\tp,a,\hs\vt\mapsto Q\,$ that satisfy\/ {\rm(\ref{pbd})}.
}
\end{theorem}
\begin{proof}First, (i) implies (ii) according to (\ref{dem}).

Assuming now (ii), let us use Remark~\ref{spcas} to select a CP triple 
$\,(\bbCP^m\nh,g'\nh,\vt')$ realizing the same data $\,d_\pm\w,\tpm,a\,$ and 
$\,\vt\mapsto Q$, in (ii) above and Remark~\ref{ascdt}(i), as our 
$\,(\mf\nh,g,\vt)\,$ (which also establishes the surjectivity part of the final 
clause). With either fixed sign $\,\pm\hs$, denoting 
$\,\sd^\mp\nh,\sd^\pm$ by $\,\sd,\varPi$, and their analogs for 
$\hs(\bbCP^m\nh,g'\nh,\vt')\hs$ by $\,\sd'\nh,\varPi'\nh$, we choose the 
isomorphisms $\,\nr\nh\to\mf\nnh\smallsetminus\nh\varPi$ and 
$\,\nr'\nh\to\bbCP^m\nh\nnh\smallsetminus\nh\varPi'$ by applying 
Theorem~\ref{first}(i) to both triples. As (i) has already been shown to yield 
(ii), we may now also apply Remark~\ref{spmcp} to both of them, identifying 
the critical manifolds $\,\sd,\sd'$ (and their sub\-man\-i\-fold metrics) with 
a complex projective space  $\,\hat\sd\,$ (and, respectively, with the 
Fu\-bi\-ni-Stu\-dy metric $\,\hat h\,$ multiplied by 
$\,2(\tp\hskip-2.2pt-\nh\tm)/\hn a$). Next, (ix) in Lemma~\ref{intpr} holds 
for both triples, so that the pairs $\,\nr\nh,\lr\,$ and $\,\nr'\nh,\lr\hn'$ 
associated with them via Theorem~\ref{first} satisfy, along with 
$\,\hat\sd=\sd=\sd'$ and $\,\hat h$, the assumptions -- as well as the 
conclusion -- of Lemma~\ref{rdehj}. Thus, some hol\-o\-mor\-phic 
vec\-tor-bun\-dle isomorphism $\,\nr\nh\to\nr'$ takes $\,\lr\,$ to 
$\,\lr\hn'$ and, since the metrics $\,\hg,\hg'$ on $\,\nr\,$ and $\,\nr'$ 
constructed in Section~\ref{ev} depend only on $\,\lr,\lr\hn'$ (aside from the 
data fixed above and shared by both triples), this isomorphism is a 
hol\-o\-mor\-phic isometry of $\,(\nr\nh,\hg)$ onto $\,(\nr'\nh,\hg')$, 
sending $\,\vt\,$ to its analog on $\,\nr'\nh$. In view of 
\cite[Lemma 16.1]{derdzinski-maschler-06}, it can be extended to an 
isomorphism between the triples $\,(\mf\nh,g,\vt)\,$ and 
$\,(\bbCP^m\nh,g'\nh,\vt')$. We thus obtain injectivity in the final clause 
and the fact that (ii) yields (i).
\end{proof}

\section{Horizontal extensions of CP triples}\label{he}
\setcounter{equation}{0}
Once again, we use the notation of (\ref{sym}), (\ref{tmm}) and (\ref{nex}), 
assuming $\,(\mf\nh,g,\vt)\,$ to be a compact ge\-o\-des\-ic-gra\-di\-ent 
K\"ah\-ler triple (Definition~\ref{ggktr}).
\begin{lemma}\label{ttgds}{\medit 
Suppose that\/ conditions\/ {\rm(i)} -- {\rm(iii)} along with the other 
assumptions of Lemma\/~{\rm\ref{intpr}} hold for a triple\/ 
$\,(\mf\nh,g,\vt)$, and\/ $\,\pi,\sb\,$ are as in 
Lemma\/~{\rm\ref{intpr}(viii)}.
\begin{enumerate}
  \def\theenumi{{\rm\alph{enumi}}}
\item[{\rm(a)}] Given a\/ $\,\pi$-pro\-ject\-a\-ble nonzero local section\/ 
$\,w\,$ of the distribution\/ $\,\mathcal{H}\,$ in\/ {\rm(\ref{tme})},
\begin{enumerate}
  \def\theenumi{{\rm\alph{enumi}}}
\item[{\rm(a1)}] $w\,$ commutes with the vector fields\/ $\,v=\navp\,$ and\/ 
$\,u=J\nh v$,
\item[{\rm(a2)}] $w\,$ is $\,\pi^\pm\hskip-1.2pt$-pro\-ject\-able for both 
signs\/ $\,\pm\hs$,
\item[{\rm(a3)}] the local flow of\/ $\,w\,$ in\/ $\,\mf'$ preserves the 
distributions\/ $\,\mathcal{V}\nh,\mathcal{H}^+$ and\/ $\,\mathcal{H}^-\nnh$.
\end{enumerate}
\item[{\rm(b)}] The leaves of the in\-te\-gra\-ble distribution\/ 
$\,\hs\zy=\mathcal{V}\hs\oplus\mathcal{H}^+\hskip-3pt\oplus\mathcal{H}^-$ on\/ 
$\,\mf'\hh$ are totally geodesic complex sub\-man\-i\-folds of\/ 
$\,\mf'\hh$ and all the local flows mentioned in\/ {\rm(a3)} act between them 
via local isometries.
\end{enumerate}
}
\end{lemma}
\begin{proof}Any $\,w\,$ in (a) is normal to the totally geodesic leaves of 
the in\-te\-gra\-ble distributions $\,\mathcal{V}\hs\oplus\mathcal{H}^\pm$ 
(see Corollary~\ref{totgd}), while $\,v,u\,$ are both tangent to them, as 
$\,\mathcal{V}=\mathrm{Span}\hs(\hn v,u)$. Therefore, 
$\,\nabla\hskip-3pt_v\w w\,$ and $\,\nabla\hskip-3pt_u\w w$, being, as a 
result, also normal to those leaves for both signs\/ $\,\pm\hs$, are -- by 
(\ref{tme}) -- sections of $\,\mathcal{H}$. The same is true of 
$\,\nabla\hskip-3pt_w\w v,\nabla\hskip-3pt_w\w u$ (and hence of 
$\,[v,w],[u,w]$) due to $\,\dv$-in\-var\-i\-ance in (\ref{tme}), with 
$\,\dv=\nabla\nh v$ and $\,\nabla\nh u=\du=JS=SJ$, cf.\ (\ref{loc}.a). At the 
same time, $\,\pi$-pro\-ject\-a\-bil\-i\-ty of $\,w$ implies, via 
Remark~\ref{liebr} and Lemma~\ref{intpr}(viii), that $\,[v,w]\,$ and 
$\,[u,w]\,$ are sections of 
$\,\zy=\mathrm{Ker}\hskip2.3ptd\pi=\mathcal{H}^\perp$. We thus obtain (a1) 
along with (a3) for $\,\mathcal{V}\nh$. Next, (a2) follows: due to 
$\,\pi$-pro\-ject\-a\-bil\-i\-ty of $\,w$, with $\,y\in\sd^\pm$ fixed, 
$\,d\pi\nnh_x\w w_x\w$ is independent of the choice of $\,x\in\mf'$ such that 
$\,\pi^\pm\nh(x)=y$, and hence so must be $\,d\pi\nh_x^{\hs\pm}w_x\w$, as the 
differential at $\,y\,$ of the bundle projection 
$\,\beta^\pm\nnh\circ\mathrm{pr}^\pm\nnh:\sd^\pm\nnh\to\sb\,$ (see (vii) -- 
(viii) in Lemma~\ref{intpr}) sends $\,d\pi\nh_x^{\hs\pm}w_x\w$ to 
$\,d\pi\nnh_x\w w_x\w$, which determines $\,d\pi\nh_x^{\hs\pm}w_x\w$ uniquely 
due to its being orthogonal, by Lemma~\ref{intpr}(v), to $\,\zy_y^\pm\nh$, 
for the vertical distribution $\,\zy^\pm$ of 
$\,\beta^\pm\nnh\circ\mathrm{pr}^\pm\nnh$.

We obtain the remainder of (a3) by noting that, for either fixed sign 
$\,\pm\hs$,
\begin{equation}\label{lft}
\begin{array}{l}
\mathrm{all\ vectors\ in\ }\,\,\mathcal{H}^\pm\mathrm{\ are\ realized\ by\ 
}\,\pi^\pm\hskip-1.2pt\hyp\mathrm{pro\-ject\-able\ local}\\
\mathrm{sections\ }\hs w^\pm\nnh\mathrm{\ of\ }\hs\mathcal{H}^\pm\nnh\mathrm{\ 
commuting\ with\ }\hs w\hh\mathrm{,\ the\ 
}\hs\pi^\pm\hskip-1.2pt\hyp\mathrm{im\-a\-ges}\\
\hat w^\pm\nnh\mathrm{\ of\ which\ also\ commute\ with\ the\ 
}\pi^\pm\hskip-1.2pt\hyp\mathrm{im\-a\-ge\ \nh}\,\hat w\,\mathrm{\nh\ of\ 
\nh}\,w.
\end{array}
\end{equation}
Namely, since the $\,\pi^\pm\hskip-1.2pt$-im\-a\-ge $\,\hat w\,$ of $\,w\,$ is 
obviously $\,\zy^\pm\hskip-1.2pt$-pro\-ject\-able, we may prescribe the 
$\,\pi^\pm\hskip-1.2pt$-im\-a\-ge $\,\hat w^\pm$ of $\,w^\pm$ to be a local 
section of $\,\zy^\pm$ commuting with $\,\hat w\,$ (see (v) -- (vi) in 
Lemma~\ref{intpr} and  Remark~\ref{cmmut}), and then lift $\,\hat w^\pm$ to 
$\,\mathcal{H}^\pm\nh$, using Remark~\ref{inequ}(iii). For the resulting lift 
$\,w^\pm\nh$, (\ref{cmt}) and parts (ix), (vi) of Lemma~\ref{intpr} give 
$\,[w,w^\pm]=0$.

We now derive (b) from Remark~\ref{tglvs}. According to Remark~\ref{frthr}, it 
suffices to establish (i) in Remark~\ref{tglvs} for local sections of 
$\,\zy\,$ having the form $\,w\hh'\nh=w^0\nh+\hh w^+\nnh+\hh w^-$ with 
$\,w^\pm$ satisfying (\ref{lft}) and $\,w^0$ equal to a 
con\-stant-co\-ef\-fi\-cient combination of $\,v$ and $\,u$. Orthogonality in 
(\ref{tme}) combined with (\ref{gvv}) shows that 
$\,g(\hn w\hh'\nh,w\hh'\hh)\,$ equals a constant multiple of $\,Q\,$ plus the 
sum of the terms $\,g(\hn w^\pm\nh,w^\pm\nh)$. Verifying part (i) of 
Remark~\ref{tglvs} thus amounts to showing that 
$\,d_w\w Q=d_w\w[\hs g(\hn w^\pm\nh,w^\pm\nh)]=0$. In terms of the 
$\,\pi^\pm\hskip-1.2pt$-im\-a\-ge $\,\hat w^\pm$ of $\,w^\pm\nh$, 
Remark~\ref{inequ}(iv) gives 
$\,(\tp\hskip-2.2pt-\tm)\hs g(\hn w^\pm\nh,w^\pm\nh)
=(\vt-\hn\tmp)\hs h(\hat w^\pm\nnh,\hat w^\pm\nh)$, where $\,h\,$ is the 
sub\-man\-i\-fold metric of $\,\sd^\pm\nh$. Since $\,\tpm$ are constants, our 
claim is thus reduced to two separate parts, $\,d_w\w Q=d_w\w\vt=0\,$ and 
$\,d_w\w[h(\hat w^\pm\nnh,\hat w^\pm\nh)]=0$. The former part is immediate: 
$\,Q\,$ is a function of $\,\vt$, cf.\ Remark~\ref{ascdt}(i), while $\,w\,$ 
and $\,v=\navp\,$ are sections of the mutually orthogonal summands 
$\,\mathcal{H}\,$ and $\,\mathcal{V}=\mathrm{Span}\hs(\hn v,u)$ in 
(\ref{tme}). For the latter part, (a2) allows us to replace $\,w\,$ by its 
$\,\pi^\pm\hskip-1.2pt$-im\-a\-ge $\,\hat w$, noting that 
$\,h(\hat w^\pm\nnh,\hat w^\pm\nh)$ is the 
$\,\pi^\pm\hskip-1.2pt$-pull\-back of a function defined (locally) in 
$\,\sd^\pm\nh$. Now $\,d_w\w[h(\hat w^\pm\nnh,\hat w^\pm\nh)]=0$ due to the 
fact that (ii) implies (i) in Remark~\ref{tglvs}, and $\,\hat w$, or 
$\,\hat w^\pm\nh$, is normal or, respectively, tangent to the totally geodesic 
leaves of the in\-te\-gra\-ble distribution $\,\zy^\pm\nh$, while $\,\hat w$, 
besides being -- as noted above -- projectable along $\,\zy^\pm\nh$, also 
commutes with $\,\hat w^\pm$ (see (v) -- (vi) in Lemma~\ref{intpr} and 
(\ref{lft})).
\end{proof}
We say that a (lo\-cal\-ly-triv\-i\-al) hol\-o\-mor\-phic fibre bundle 
carries a specific lo\-cal-type {\medit fibre geometry\/} if such a geometric 
structure is selected in each of its fibres and suitable local 
$\,C^\infty$ trivializations make the structures appear the same in all 
nearby fibres. For instance, hol\-o\-mor\-phic complex vectors bundle endowed 
with Her\-mit\-i\-an fibre metrics may be referred to as 
\begin{enumerate}
  \def\theenumi{{\rm\roman{enumi}}}
\item[{\rm(i)}] hol\-o\-mor\-phic bundles of Her\-mit\-i\-an vector spaces.
\end{enumerate}
The fact that (i) leads to the presence of the distinguished Chern connection 
(Section~\ref{ck}) has obvious generalizations to two situations (ii) -- (iii) 
discussed below. 

By a {\medit horizontal distribution\/} for a hol\-o\-mor\-phic bundle 
projection $\,\pi:\mf\to\sb$ between complex manifolds, also called a 
{\medit connection in the hol\-o\-mor\-phic bundle\/} $\hs\mf$ \hbox{over 
$\,\sb$,} we mean any $\,C^\infty$ real vector sub\-bun\-dle 
$\,\mathcal{H}\,$ of $\,T\nnh\mf\nh$, complementary to the vertical 
distribution $\,\mathrm{Ker}\hskip2.3ptd\pi$, so that 
$\,T\nnh\mf\,$ is the direct sum of $\,\mathrm{Ker}\hskip2.3ptd\pi\,$ and 
$\,\mathcal{H}$. {\medit Horizontal lifts\/} of vectors tangent to $\,\sb$, 
and of piecewise $\,C^1$ curves in $\,\sb$, as well as {\medit parallel 
transports\/} along such curves, are then defined in the usual fashion, 
although the maximal domain of a lift of a curve (or, of a parallel transport) 
may in general be a proper sub\-in\-ter\-val of the original domain interval. 
This last possibility does not, however, occur in bundles with compact fibres, 
or in vector bundles with linear connections, where horizontal lifts of curves 
and parallel transports are all {\medit global}.

We proceed to describe the {\medit Chern connection\/} $\,\mathcal{H}\,$ in 
the cases of
\begin{enumerate}
  \def\theenumi{{\rm\alph{enumi}}}
\item[{\rm(ii)}] hol\-o\-mor\-phic bundles of Fu\-bi\-ni-Stu\-dy complex 
projective spaces, and
\item[{\rm(iii)}] hol\-o\-mor\-phic bundles of \,CP triples, over any complex 
manifold $\,\sb$.
\end{enumerate}
Their fibre geometries consist of 
Fu\-bi\-ni-Stu\-dy metrics (Remark~\ref{fbstm}) and, respectively, the 
structures of a \,CP triple (Section~\ref{eg}).

For (ii), $\,\mathcal{H}\,$ arises since local $\,C^\infty$ trivializations 
mentioned earlier may be chosen so as to share their domains with local 
hol\-o\-mor\-phic trivializations; the former make the fibre geometry appear 
constant, and the latter turn the bundle, locally, into the projectivization 
(\ref{prz}) of a hol\-o\-mor\-phic vector bundle $\,\ee\,$ endowed with a 
Her\-mit\-i\-an fibre metric $\,(\hskip2.3pt,\hskip1.3pt)\,$ that induces the 
Fu\-bi\-ni-Stu\-dy metrics of the original fibres. Since 
$\,(\hskip2.3pt,\hskip1.3pt)$ is unique up to multiplications by positive 
functions (Remark~\ref{fbstm}), Lemma~\ref{ddrsq}(iv) easily implies that its 
choice does not affect the resulting parallel transports between the 
projectivized fibres, thus giving rise to $\,\mathcal{H}$.

The Chern connection $\,\mathcal{H}\,$ now also arises in case (iii) since, 
according to Remark~\ref{cnvrs}, (iii) is a sub\-case of (ii). The situation 
is, however, more special: the critical manifolds -- analogs of (\ref{tmm}) -- 
in the fibres now constitute two hol\-o\-mor\-phic bundles $\,\sd^\pm$ of 
Fu\-bi\-ni-Stu\-dy complex projective spaces over $\,\sb\,$ (with fibre 
dimensions that need not be both positive; see Remark~\ref{spmcp}), contained 
as sub\-bundles in the original bundle, and invariant under all 
$\,\mathcal{H}$-par\-al\-lel transports. Also, the fibre-ge\-om\-e\-try 
gradients and their $\,J$-im\-a\-ges (analogous to what we normally denote by 
$\,v=\navp\,$ and $\,u$) together form two hol\-o\-mor\-phic vertical vector 
fields $\,v\,$ and $\,u=J\nh v$ on the total space. This is immediate from 
the preceding paragraph, with the two sub\-bundles $\,\sd^\pm$ corresponding 
to a $\,(\hskip2.3pt,\hskip1.3pt)$-or\-thog\-o\-nal hol\-o\-mor\-phic 
decomposition $\,\ee\nh=\ee^+\nnh\oplus\ee^-$ of the lo\-cal\-ly-de\-fined 
vector bundle $\,\ee\nh$, cf.\ (\ref{dta}.ii) and (\ref{spm}.c) -- 
(\ref{spm}.d), while the flow of $\,u$, described in the lines following 
(\ref{dta}), acts in both $\,\ee^\pm$ via multiplications by two (unrelated) 
constant unit complex scalars. In case (ii), or (iii),
\begin{equation}\label{prt}
\begin{array}{l}
\mathrm{the\ }\hs\mathcal{H}\hyp\mathrm{par\-al\-lel\ transports\ are\ 
hol\-o\-mor\-phic\ isometries\ or,}\\
\mathrm{respectively,\ CP}\nh\hyp\mathrm{triple\ isomorphisms\ between\ the\ 
fibres,}\\
\end{array}
\end{equation}
which holds for (ii) since it does for (i), cf.\ Section~\ref{ck} and, 
consequently, also extends to the case of (iii) via the canonical 
modifications in Remarks~\ref{alltq} and~\ref{cnvrs}.

The following assumptions and notations will now be used to construct compact 
ge\-o\-des\-ic-gra\-di\-ent K\"ah\-ler triples, each of which we call a 
{\medit horizontal extension\/} of the \,CP triple provided by any fibre 
$\,(\pi^{-\nnh1}(z),g^z\nh,\vt^z)$.
\begin{enumerate}
  \def\theenumi{{\rm\alph{enumi}}}
\item[{\rm(a)}] $\pi\nh:\nnh\mf\nh\to\nh\sb\,$ and $\,\mathcal{H}\,$ are the 
bundle projection and the Chern connection of a hol\-o\-mor\-phic bundle of 
\,CP triples with a compact base $\,\sb\,$ and the 
\hbox{\,CP\nh-}\hskip0pttriple fibres $\,(\pi^{-\nnh1}(z),g^z\nh,\vt^z)$, 
$\,z\in\sb$, while $\,\sd^\pm$ stand for the above sub\-bundles of 
Fu\-bi\-ni-Stu\-dy complex projective spaces, invariant under 
$\,\mathcal{H}$-par\-al\-lel transports.
\item[{\rm(b)}] We let $\,\tpm,a\,$ be the data associated with 
some\hs$/$\hn any fibre $\,(\pi^{-\nnh1}(z),g^z\nh,\vt^z)$ as in 
Remark~\ref{ascdt}(i), and $\,\vt:\mf\to\bbR\,$ (or, 
$\,\pi^\pm\nnh:\mf\nnh\smallsetminus\nh\sd^\mp\nh\to\sd^\pm$) be the 
$\,C^\infty$ function (or, hol\-o\-mor\-phic bundle projection) which, 
restricted to each $\,\pi^{-\nnh1}(z)$, equals $\,\vt^z$ or, respectively, the 
version of (\ref{dbp}) corresponding to $\,(\pi^{-\nnh1}(z),g^z\nh,\vt^z)$. We 
also set $\,\mf'\nh=\mf\smallsetminus(\sd^+\nnh\cup\sd^-)$.
\item[{\rm(c)}] One is given two K\"ah\-ler metrics $\,\hpm$ on the total 
spaces $\,\sd^\pm$ of our hol\-o\-mor\-phic bundles of Fu\-bi\-ni-Stu\-dy 
complex projective spaces such that either $\,\hpm$ makes the fibres 
$\,\sd^\pm_{\nh z}\nh$, $\,z\in\sb$, or\-thog\-o\-nal to $\,\mathcal{H}\,$ 
along $\,\sd^\pm$ and, restricted to each fibre, $\,\hpm$ equals 
$\,2(\tp\hskip-2.2pt-\nh\tm)/\hn a\,$ times the Fu\-bi\-ni-Stu\-dy metric of 
$\,\sd^\pm_{\nh z}\nh$.
\item[{\rm(d)}] We define a Riemannian metric $\,g\,$ on $\,\mf'$ by 
requiring that $\,\mathcal{H}\,$ be $\,g$-or\-thog\-o\-nal to the vertical 
distribution $\,\mathrm{Ker}\hskip2.3ptd\pi$, that $\,g\,$ agree on the fibres 
$\,\pi^{-\nnh1}(z)\,$ with the metrics $\,g^z\nh$, and that 
$\,(\tp\hskip-2.2pt-\nh\tm)\hs g\hs=(\vt-\hn\tm)\hs\hp\nh
+(\tp\hskip-2.3pt-\vt)\hs\hm$ on $\,\mathcal{H}$, the symbols $\,\hpm$ being 
also used for the $\,\pi^\pm\hskip-1.2pt$-pull\-backs of $\,\hpm\nh$, cf.\ (b) 
-- (c).
\item[{\rm(e)}] Our final assumption is that the Riemannian metric $\,g\,$ on 
the dense open sub\-man\-i\-fold $\,\mf'$ has an extension to a K\"ah\-ler 
metric on $\,\mf$ (still denoted by $\,g$).
\end{enumerate}
\begin{remark}\label{adggk}Under the hypotheses (a) -- (e), the resulting 
horizontal extension $\,(\mf\nh,g,\vt)\,$ is actually a 
ge\-o\-des\-ic-gra\-di\-ent K\"ah\-ler triple. Namely, being a part of the 
geometry of the fibres $\,(\pi^{-\nnh1}(z),g^z\nh,\vt^z)$, the functions 
$\,\vt^z$ are preserved by $\,\mathcal{H}$-par\-al\-lel parallel transports, 
that is, $\,\vt\,$ is constant along $\,\mathcal{H}$, and so its (vertical) 
$\,g$-gra\-di\-ent must, by Remark~\ref{gpgrd}. coincide with the 
hol\-o\-mor\-phic vertical vector field $\,v\,$ described in the lines 
preceding (\ref{prt}). On the other hand, the function $\,Q=g(\hn v,v)$, equal 
- consequently - to its fibre version, is a specific function of $\,\vt$. 
Thus, by Lemma~\ref{ggqft}, $\,\vt\,$ has a hol\-o\-mor\-phic geodesic 
$\,g$-gra\-di\-ent.
\end{remark}
\begin{remark}\label{hhext}Whenever a compact ge\-o\-des\-ic-gra\-di\-ent 
K\"ah\-ler triple $\,(\mf\nh,g,\vt)\,$ is a horizontal extension arising as in 
Remark~\ref{adggk}, the distribution 
$\,\zy=\mathcal{V}\hs\oplus\mathcal{H}^+\hskip-3pt\oplus\mathcal{H}^-$ on 
$\,\mf'$ coming from the decomposition (\ref{tme}) for $\,(\mf\nh,g,\vt)\,$ 
coincides, on $\,\mf'\nh$, with the vertical distribution 
$\,\mathrm{Ker}\hskip2.3ptd\pi\,$ of the bundle projection 
$\,\pi\nh:\nnh\mf\nh\to\nh\sb\,$ (see (a) above) and, consequently, 
$\,\zy\,$ is integrable.

In fact, applying Remark~\ref{tglvs} to $\,\mathcal{H}$-hor\-i\-zon\-tal lifts 
$\,w\,$ of local vector fields on $\,\sb$, we see that, by (\ref{prt}), 
$\,\mathrm{Ker}\hskip2.3ptd\pi\,$ has totally geodesic leaves. Using 
(\ref{img}) for both $\,(\mf\nh,g,\vt)\,$ and the fibres 
$\,(\pi^{-\nnh1}(z),g^z\nh,\vt^z)$, we now conclude that the projections 
$\,\pi^\pm\nnh:\mf\nnh\smallsetminus\nh\sd^\mp\nh\to\sd^\pm$ defined in (b) 
are the same as those in (\ref{dbp}). (Note that, due to the orthogonality 
requirement in (c), the minimizing geodesic segment in $\,\pi^{-\nnh1}(z)$, 
$\,z\in\sb$, joining a point $\,x\in\pi^{-\nnh1}(z)\,$ to $\,\sd^\pm_{\nh z}$ 
a normal to $\,\sd^\pm_{\nh z}\nh$, serves as the segment with the same 
properties for $\,\mf\,$ rather than $\,\pi^{-\nnh1}(z)$.) Now (\ref{hpm}) 
implies that the distribution 
$\,\mathcal{V}\hs\oplus\mathcal{H}^+\hskip-3pt\oplus\mathcal{H}^-$ is 
contained in $\,\mathrm{Ker}\hskip2.3ptd\pi\,$ and, restricted to every fibre, 
equals the analog of 
$\,\mathcal{V}\hs\oplus\mathcal{H}^+\hskip-3pt\oplus\mathcal{H}^-$ for the 
fibre, that is, its tangent bundle (see Theorem~\ref{cptrp}(ii)). Thus, 
$\,\zy\,$ coincides with the full vertical distribution 
$\,\mathrm{Ker}\hskip2.3ptd\pi$.
\end{remark}
\begin{theorem}\label{cpbdl}{\medit 
A ge\-o\-des\-ic-gra\-di\-ent K\"ah\-ler triple\/ $\,(\mf\nh,g,\vt)$, with 
compact\/ $\,\mf\nh$, satisfies one$/\hn$all of the 
mu\-tu\-al\-ly-e\-quiv\-a\-lent conditions\/ {\rm(i)} -- {\rm(iii)} of 
Lemma\/~{\rm\ref{intpr}}, if and only if it is iso\-mor\-phic to a horizontal 
extension of a CP triple, defined as above using\/ {\rm(a)} -- {\rm(e)}.
}
\end{theorem}
\begin{proof}Remark~\ref{hhext} clearly yields the `if' part of our claim.

Conversely, let $\,(\mf\nh,g,\vt)\,$ satisfy (i) -- (iii) in 
Lemma~\ref{intpr}. Lemma~\ref{intpr}(viii) states that 
$\,\zy=\mathcal{V}\hs\oplus\mathcal{H}^+\hskip-3pt\oplus\mathcal{H}^-$ 
coincides, on $\,\mf'\nh$, with the vertical distribution 
$\,\mathrm{Ker}\hskip2.3ptd\pi$ of the hol\-o\-mor\-phic bundle projection 
$\,\pi:\mf\to\sb$. Also, in view of Remark~\ref{trivl}, the leaves of 
$\,\zy\,$ form ge\-o\-des\-ic-gra\-di\-ent K\"ah\-ler triples, due to their 
being complex sub\-man\-i\-folds of $\,\mf\,$ tangent to $\,v=\navp\,$ (since 
$\,\mathcal{V}=\mathrm{Span}\hs(\hn v,u)$) and, as they are also totally 
geodesic (see Lemma~\ref{ttgds}(b)), (\ref{hvk}) and the 
$\,\dv$-in\-var\-i\-ance in (\ref{tme}), with $\,\dv=\nabla\nh v$, imply via 
Theorem~\ref{cptrp} that they are all iso\-mor\-phic to \,CP\ triples. The 
local isometries of Lemma~\ref{ttgds}(b) can obviously be made global due to 
compactness (see the lines preceding (ii) above) which, consequently, turns 
$\,\mf\,$ into a hol\-o\-mor\-phic bundle of \,CP triples over $\,\sb$, in the 
sense of (iii).

On the other hand, the $\,g$-or\-thog\-o\-nal complement of 
$\,\zy=\mathrm{Ker}\hskip2.3ptd\pi\,$ is equal, on $\,\mf'\nh$, to the summand 
$\,\mathcal{H}\,$ in (\ref{tme}). Thus, $\,\mathcal{H}\,$ constitutes a 
connection in the bundle $\,\mf\,$ over $\,\sb$, as defined in the lines 
following (i), and -- being the intersection of the horizontal distribution of 
the Chern connections $\,\mathcal{V}\hs\oplus\mathcal{H}^\pm$ in the normal 
bundles $\,\nr=\nr\hskip-2.3pt\sd^\pm\nh$, cf.\ Theorem~\ref{first}(ii) -- 
$\,\mathcal{H}\,$ itself is, according to (a) in Section~\ref{ck}, the Chern 
connection of the hol\-o\-mor\-phic bundle $\,\mf\,$ of \,CP triples over 
$\,\sb$.

This provides parts (a) -- (b) of the data (a) -- (e) required above, with 
$\,\sd^\pm$ and $\,\tpm,a\,$ given by (\ref{tmm}) and, respectively, 
Remark~\ref{ascdt}(i). The sub\-man\-i\-fold metrics $\,\hpm$ of $\,\sd^\pm$ 
have, by (v) -- (vi) in Lemma~\ref{intpr} and the final clause of 
Theorem~\ref{dcomp}(b), all the properties needed for (c). 

To show that $\,g\,$ satisfies (d), consider two $\,\pi$-pro\-ject\-a\-ble 
nonzero local sections $\,w,w\hh'$ of the distribution 
$\,\mathcal{H}=\zy^\perp\nh$, cf.\ (\ref{tme}). According to 
Lemma~\ref{ttgds}(a) and the last line of Remark~\ref{pralg}, $\,w\,$ and 
$\,w\hh'$ are projectable along $\,\mathcal{V}\,$ and $\,\mathcal{H}^\pm\nh$, 
as well as $\,\pi^\pm\hskip-1.2pt$-pro\-ject\-able, for either sign 
$\,\pm\hs$. Their restrictions to any fixed normal geodesic segment 
$\,\ic\hs$ emanating from $\,\sd^\pm$ thus lie in the space $\,\mathcal{W}\,$ 
(cf.\ (\ref{vtg}) and (i) -- (ii) in Theorem~\ref{jacob}) and, by 
Theorem~\ref{dcomp}(g), $\,g(\hn w,w\hh'\hh)\,$ restricted to $\,\ic\hs$ is 
a (possibly nonhomogeneous) linear function of $\,\vt$. The same linearity 
condition obviously holds for $\,g(\hn w,w\hh'\hh)\,$ when $\,g\,$ is defined 
as in (d), rather than being the metric of our triple $\,(\mf\nh,g,\vt)$. The 
two definitions of $\,g(\hn w,w\hh'\hh)\,$ must now agree, as the two linear 
functions have -- in view of Remark~\ref{inequ}(iii) and the final clause of 
Theorem~\ref{dcomp}(b) -- the same values 
$\,\hpm\nh(\hn w,w\hh'\hh)\,$ at either endpoint $\,\tpm$ of the interval 
$\,[\hh\tm,\tp]$.
\end{proof}
\begin{remark}\label{clone}All compact SKRP triples of Class 1 (cf.\ 
Section~\ref{ev}) must be
\begin{equation}\label{cdo}
\mathrm{iso\-mor\-phic\ to\ horizontal\ extensions\ of\ CP\ triples\ of\ 
complex\ dimension\ }\,1\hh,
\end{equation}
while those of Class 2 are themselves CP triples of a special type. The former 
claim is easily verified using \cite[Theorem 16.3]{derdzinski-maschler-06}; 
for the latter, see Lemma~\ref{cltwo}.

The classification result of \cite[Theorem 6.1]{derdzinski-kp} may be 
rephrased as the conclusion (\ref{cdo}) about all compact 
ge\-o\-des\-ic-gra\-di\-ent K\"ah\-ler triples $\,(\mf\nh,g,\vt)\,$ with 
$\,\dimc\nh\mf=2\,$ other than Class 2 SKRP triples are. Similarly, 
(\ref{cdo}) is the case -- by their very construction -- for the gradient 
K\"ah\-\hbox{ler\hs-}\hskip0ptRic\-ci sol\-i\-tons of Koiso \cite{koiso} and 
Cao \cite{cao}, mentioned in the Introduction.
\end{remark}

\section{Con\-stant-rank multiplications\done}\label{cr}
\setcounter{equation}{0}
In this section all vector spaces are 
fi\-\hbox{nite\hh-}\hskip0ptdi\-men\-sion\-al and complex. Bi\-lin\-e\-ar 
mappings of the type discussed here arise in any compact 
ge\-o\-des\-ic-gra\-di\-ent K\"ah\-ler triple (see Theorem~\ref{ctrkm}), which 
leads to the dichotomy conclusion of Theorem~\ref{dicho}.

A {\medit con\-stant-rank multiplication\/} is any bi\-lin\-e\-ar mapping 
$\,\mu:\on\nnh\times\nh\tw\to\oy\nh$, where $\,\on,\tw\nnh,\oy\,$ are vector 
spaces, such that the function $\,\on\smallsetminus\{0\}\ni\xi
\mapsto\mathrm{rank}\hskip1.7pt\mu(\xi,\,\cdot\,)\,$ is constant or, 
equivalently, $\,\dim\hs\mathrm{Ker}\hskip1.7pt\mu(\xi,\,\cdot\,)\,$ is the 
same for all nonzero $\,\xi\in\on\nh$. When 
$\,\dim\hs\mathrm{Ker}\hskip1.7pt\mu(\xi,\,\cdot\,)=k\,$ for all 
$\,\xi\in\on\smallsetminus\{0\}$, we also say that 
$\,\mu:\on\nnh\times\nh\tw\to\oy\,$ {\medit has the constant rank\/} 
$\,\dim\tw\hskip-2.3pt-k$. With the notations of Section~\ref{eg}, such 
$\,\mu\,$ leads to a mapping
\begin{equation}\label{phi}
\mg:\mathrm{P}\nh\on\to\,\mathrm{Gr}\nh_k\w\nh\tw\hskip10pt
\mathrm{given\ by}\hskip8pt\mg(\hn\bbC\xi)\hs
=\hs\mathrm{Ker}\hskip1.7pt\mu(\xi,\,\cdot\,)\hskip6pt
\mathrm{for}\hskip5pt\xi\in\on\smallsetminus\{0\}\hh.
\end{equation}
\begin{lemma}\label{dfphi}{\medit 
For\/ $\,\mu\,$ and\/ $\,\mg\,$ as above, 
$\,\on\smallsetminus\{0\}\ni\xi
\mapsto\hs\mathrm{Ker}\hskip1.7pt\mu(\xi,\,\cdot\,)
\in\hs\mathrm{Gr}\nh_k\w\nh\tw\,$ and\/ $\,\mg$ are both hol\-o\-mor\-phic. 
In terms of the identification\/ {\rm(\ref{twg})}, the differential of the 
former mapping at any\/ $\,\xi\in\on\smallsetminus\{0\}\,$ sends\/ 
$\,\eta\in\on\,$ to the unique\/ 
$\,H\in\mathrm{Hom}\hs(\ws,\tw\hskip-1.8pt/\hh\ws)\,$ with\/ 
$\,\mu(\eta,w)=\mu(\xi,\nnh-\nnh\tilde H\nh w)\,$ for all\/ 
$\,w\in\ws=\mg(\hn\bbC\xi)$, where\/ $\,\tilde H:\ws\to\tw\,$ is any linear 
lift of\/ $\hs H$.
}
\end{lemma}
\begin{proof}This is obvious if one sets $\,F(\xi)=\mu(\xi,\,\cdot\,)\,$ in 
Remark~\ref{dfrtl}. 
\end{proof}
\begin{example}\label{cstrk}Any given con\-stant-rank multiplication 
$\,\mu:\on\nnh\times\nh\tw\to\oy\,$ leads to further such multiplications, 
$\,\mu':\on\nnh\times\nh\tw\hh'\nh\to\oy'$ and 
$\,\mu^*:\on\nnh\times\nh\oy^*\nh\to\tw^*\nh$, obtained by setting 
$\,\mu'(\xi,\,\cdot\,)=\gamma[\mu(\xi,\alpha\,\cdot\,)]\,$ and  
$\,\mu^*(\xi,\,\cdot\,)=[\mu(\xi,\,\cdot\,)]^*\nh$. Here $\,\tw\hh'\nh,\oy'$ 
are vector spaces, $\,\alpha:\tw\hh'\nh\to\tw\,$ (or, $\,\gamma:\oy\to\oy'$) is 
surjective (or, injective) and linear, while $\,[\hskip3.2pt]^*$ stands for 
the dual of a vector space or a linear operator.
\end{example}
\begin{lemma}\label{holem}{\medit 
If\/ $\,\mu:\on\nnh\times\nh\tw\to\oy\,$ has the constant 
rank\/ $\,\dim\tw\nnh-k\,$ and\/ $\,\mg\,$ with\/ {\rm(\ref{phi})} 
is nonconstant, then\/ $\,\mg\,$ is a hol\-o\-mor\-phic embedding.

Whether\/ $\,\mg\,$ is constant, or not, the same is the case for all 
multiplications\/ $\,\on\nnh\times\nh\tw\to\oy$ of the constant rank\/ 
$\,\dim\tw\nnh-k$, sufficiently close to\/ $\,\mu$.
}
\end{lemma}
\begin{proof}Let $\,\ws\in\mathrm{Gr}\nh_k\w\nh\tw$. The subset of $\,\on\,$ 
consisting of $\,0\,$ and all \hbox{$\xi\in\on\smallsetminus\{0\}\,$} with 
$\,\mg(\hn\bbC\xi)=\ws\,$ is a vector sub\-space. In fact, if 
$\,\xi,\eta\in\on\smallsetminus\{0\}\,$ and 
$\,\ws=\mathrm{Ker}\hskip1.7pt\mu(\xi,\,\cdot\,)
=\mathrm{Ker}\hskip1.7pt\mu(\eta,\,\cdot\,)$, then 
$\,\ws\subseteq\mathrm{Ker}\hskip1.7pt\mu(\zeta,\,\cdot\,)\,$ for any 
$\,\zeta\in\mathrm{Span}\hh(\xi,\eta)\,$ and, unless $\,\zeta=0$, this 
inclusion is actually an equality due to the con\-stant-rank property of 
$\,\mu$.

Therefore, $\,\mg\hs$-pre\-im\-ages of points of 
$\,\mathrm{Gr}\nh_k\w\nh\tw\,$ are linear sub\-va\-ri\-e\-ties in 
$\,\mathrm{P}\nh\on$. If $\,\mg$ is nonconstant, all these 
sub\-va\-ri\-e\-ties are ze\-\hbox{ro\hh-}\hskip0ptdi\-men\-sion\-al, that is, 
$\,\mg\,$ has to be injective. Namely, by Lemma~\ref{nontr}, for the 
K\"ah\-ler form $\,\omega\,$ of any K\"ah\-ler metric on 
$\,\mathrm{Gr}\nh_k\w\nh\tw\nnh$, the integral of $\,\mg\sk\omega\,$ 
over any projective line $\,\pl\,$ in $\,\mathrm{P}\nh\on\,$ is nonzero, and 
so $\,\pl$ cannot lie in the $\,\mg\hs$-pre\-im\-age of a point. Also, 
Lemma~\ref{dfphi} guarantees hol\-o\-mor\-phic\-i\-ty of $\,\mg$.

Let $\,\mg\hs$ now be nonconstant. Then $\,\mg\,$ must be an embedding, that 
is, $\,d\mg_{\bbC\xi}\w$ is injective at any 
$\,\bbC\xi\in\mathrm{P}\nh\on\,$ or, equivalently, the differential of 
$\,\xi\mapsto\hs\mathrm{Ker}\hskip1.7pt\mu(\xi,\,\cdot\,)\,$ 
at any $\,\xi\in\on\smallsetminus\{0\}\,$ has the kernel $\,\bbC\xi$. 
Namely, in Lemma~\ref{dfphi} we may set $\,\tilde H=0$ when $\,H=0$, and so 
$\,\eta\,$ lies in the kernel if and only if the inclusion 
$\,\ws\subseteq\mathrm{Ker}\hskip1.7pt\mu(\eta,\,\cdot\,)$ holds for 
$\,\ws\nh=\mg(\hn\bbC\xi)$. Unless $\,\eta=0$, this inclusion is, as 
before, an equality, and injectivity of $\,\mg\,$ then yields 
$\,\eta\in\bbC\xi$, which completes the proof, the final clause being 
an immediate consequence of that in Lemma~\ref{nontr}.
\end{proof}
Given a compact ge\-o\-des\-ic-gra\-di\-ent K\"ah\-ler triple $\,(\mf\nh,g,\vt)$, 
we use the notation of (\ref{sym}) and (\ref{tmm}) to set, for 
$\,\xi,\eta\in\nr\hskip-2.4pt_y\w\sd^\pm$ and $\,w\in\tyb^\pm\nh$, with either 
fixed sign $\,\pm\hs$,
\begin{equation}\label{zxe}
Z_y^\pm\nh(\xi,\eta)\hh w\,=\,a\hh g_y\w(\xi,\eta)\hh w\,
+\,(\tp\hskip-2.2pt-\nh\tm)\hs R_y\w(\xi,J\hskip-2pt_y\eta)\hh J\hskip-2pt_y\w w\hh.
\end{equation}
Thus, $\,Z_y^\pm\nh(\xi,\eta)\hh w\in\tyb^\pm\nh$, as $\,\xi,\eta\,$ are 
tangent, and $\,w\,$ normal, to the totally geodesic leaf through $\,y\,$ of 
the $\,J$-in\-var\-i\-ant in\-te\-gra\-ble distribution 
$\,\mathrm{Ker}\hskip2.3ptd\pi^\pm\nh
=\hs\mathcal{V}\hs\oplus\mathcal{H}^\mp\nh$, cf.\ (\ref{hpm}), 
Theorem~\ref{dcomp}(c), Corollary~\ref{totgd}, and the first line of 
Remark~\ref{ptnrs}. Also, denoting by $\,Z_y^\pm\nh(\xi,\eta)\,$ the 
en\-do\-mor\-phism $\,w\mapsto Z_y^\pm\nh(\xi,\xi)\hh w\,$ of $\,\tyb^\pm\nh$, 
one has
\begin{equation}\label{zex}
Z_y^\pm\nh(\xi,\eta)=Z_y^\pm\nh(\eta,\xi)
=Z_y^\pm\nh(J\hskip-2pt_y\w\xi,J\hskip-2pt_y\w\eta)\hh w\hh,
\hskip16pt
J\hskip-2pt_y\w[Z_y^\pm\nh(\xi,\eta)]=[Z_y^\pm\nh(\xi,\eta)]J\hskip-2pt_y\w\hh,
\end{equation}
as an obvious consequence of (\ref{rcm}) and (\ref{gjw}). Next, we define a 
com\-plex-bi\-lin\-e\-ar mapping 
$\,\mu_y^\pm\nnh:\nr\hskip-2.4pt_y\w\sd^\pm\nnh\times\tyb^\pm\nh\to
\overline{\mathrm{Hom}}_\bbC\w(\nr\hskip-2.4pt_y\w\sd^\pm\nh,\tyb^\pm\nh)\,$ 
by
\begin{equation}\label{mxw}
\mu_y^\pm(\xi,w)\,=\,Z_y^\pm\nh(J\hskip-2pt_y\w\xi,\,\cdot\,)\hh w\,
+\,Z_y^\pm\nh(\xi,\,\cdot\,)J\hskip-2pt_y\w w\hh.
\end{equation}
By $\,\overline{\mathrm{Hom}}_\bbC\w$ we mean here `the space of 
anti\-lin\-e\-ar operators' and 
$\,\overline{\mathrm{Hom}}_\bbC\w(\nr\hskip-2.4pt_y\w\sd^\pm\nh,\tyb^\pm\nh)$ 
is treated as a complex vector space in which the multiplication by 
$\,i\,$ 
acts via composition with $\,J\hskip-2pt_y\w$ from the left. (The product thus equals 
the given operator $\,\nr\hskip-2.4pt_y\w\sd^\pm\nh\to\tyb^\pm$ {\medit 
followed\/} by $\,J\hskip-2pt_y\w$.) Anti\-lin\-e\-ar\-i\-ty of $\,\mu_y^\pm(\xi,w)\,$ 
and com\-plex-bi\-lin\-e\-ar\-i\-ty of $\,\mu_y^\pm$ are both obvious from 
(\ref{zex}).
\begin{theorem}\label{ctrkm}{\medit 
For a compact ge\-o\-des\-ic-gra\-di\-ent K\"ah\-ler triple\/ 
$\,(\mf\nh,g,\vt)$, a fixed sign\/ $\,\pm\hs$, and any point\/ 
$\,y\in\sd^\pm\nh$, the mapping\/ $\,\mu_y^\pm$ with\/ {\rm(\ref{mxw})} is a 
con\-stant-rank multiplication, cf.\ Section\/~{\rm\ref{cr}}. Furthermore, 
if\/ $\,\mg=\mg_y^\pm$ corresponds to\/ 
$\,\mu=\mu_y^\pm$ as in\/ {\rm(\ref{phi})} and\/ $\,\xi\,$ is any nonzero 
vector normal to\/ $\,\sd^\pm$ at\/ $\,y$, then
\begin{equation}\label{dph}
\begin{array}{rl}
\mathrm{i)}&\mg_y^\pm(\hn\bbC\xi)\,=\,d\pi\nh_x^{\hs\pm}\nh(\mathcal{H}_x^\pm\nh)\,
=\,d\pi\nh_x^{\hs\pm}\nh(\mathcal{V}_{\nh x}\w\oplus\mathcal{H}_x^\pm\nh)\hh,
\hskip12pt\mathrm{where}\hskip7ptx=\varPhi(y,\xi)\hh,\phantom{_{j_j}}\\
\mathrm{ii)}&\mg_y^\pm(\hn\bbC\xi)\,
=\,\mathrm{Ker}\hskip1.7ptZ_y^\pm\nh(\xi,\xi)\hh,\hskip40pt\mathrm{for}
\hskip7ptZ_y^\pm\nh(\xi,\xi)\hskip7pt\mathrm{as\ \ in\ \ (\ref{zex}).}
\phantom{^{1^1}}
\end{array}
\end{equation}
}
\end{theorem}
\begin{proof}Whenever $\,x=\varPhi(y,\xi)\,$ and 
$\,\xi\in\nr\hskip-2.4pt_y\w\sd^\pm\nh\smallsetminus\nh\{0\}$, we have
\begin{equation}\label{krz}
\mathrm{Ker}\hskip1.7ptZ_y^\pm\nh(\xi,\xi)\,=\,
d\pi\nh_x^{\hs\pm}\nh(\mathcal{H}_x^\pm\nh)\,
=\,d\pi\nh_x^{\hs\pm}\nh(\mathcal{V}_{\nh x}\w\oplus\mathcal{H}_x^\pm\nh)\hh.
\end{equation}
In fact, let $\,x=x(t)\in\ic\hs$ as in Theorem~\ref{dcomp}, with some fixed 
$\,t\in(t_-\w,t_+\w)$. According to (\ref{hvk}) and parts (iii), (iv), (vi) of 
Theorem~\ref{jacob}, the vectors forming $\,\mathcal{H}_x^\pm$ are precisely 
the values $\,w(t)\,$ for all $\,w\,$ as in Theorem~\ref{dcomp}(e) which also 
have the property that 
$\,2\hh(\vt-\hn\tmp)\hh Q^{-\nnh1}g(\dv\nh w,w\hh'\hh)=g(\hn w,w\hh'\hh)\,$ 
whenever $\,w\hh'$ satisfies the hypotheses of Theorem~\ref{dcomp}(e). Since 
the values $\,w_\pm'$ in Theorem~\ref{dcomp}(h2) fill $\,\tyb^\pm$ (cf.\ 
assertions (d) -- (f) of Theorem~\ref{dcomp}), replacing $\,g(\hn w,w\hh'\hh)\,$ 
and $\,g(\dv\nh w,w\hh'\hh)\,$ in the last equality with the expressions 
provided by Theorem~\ref{dcomp}(h2) and Remark~\ref{bthex}, we easily verify, 
using (\ref{rcm}) and Remark~\ref{inequ}(i), that $\,w(t)\in\mathcal{H}_x^\pm$ 
if and only if $\,Z_y^\pm\nh(\xi,\eta)\hh w_\pm\w\nh=0$. Now the final clause 
of Theorem~\ref{dcomp}(b) (or, Remark~\ref{holom}) yields the first 
\hbox{(or, second) equality in (\ref{krz}).}

To simplify notations, let us write $\,g,Z,J\,$ rather than 
$\,g_y\w,Z_y^\pm\nh,J\hskip-2pt_y\w$. Since $\,x=\varPhi(y,\xi)\,$ in (\ref{krz}) and 
$\,\varPhi\,$ is hol\-o\-mor\-phic (Theorem~\ref{first}), (\ref{krz}) and 
Remark~\ref{holom} clearly imply that, for a suitable integer $\,k=k_\pm\w$, 
the resulting mapping
\begin{equation}\label{mpg}
\nr\hskip-2.4pt_y\w\sd^\pm\nh\smallsetminus\nh\{0\}\ni\xi\,\,
\mapsto\,\,\mathrm{Ker}\hskip1.7ptZ(\xi,\xi)
\in\hs\mathrm{Gr}\nh_k\w\nh(\tyb^\pm\nh)\hskip12pt\mathrm{is\ 
hol\-o\-mor\-phic.}
\end{equation}
The $\,C^\infty$ version of the assumptions 
listed in Remark~\ref{dfrtl} is thus satisfied if one chooses 
$\,\,U\nh,\tw\nnh,\oy\,$ to be 
$\,\nr\hskip-2.4pt_y\w\sd^\pm\nh\smallsetminus\nh\{0\},\tyb^\pm\nh,\tyb^\pm$ 
and sets $\,F(\xi)=Z(\xi,\xi)$. By (\ref{fxh}), the differential
of (\ref{mpg}) at any nonzero $\,\xi\in\nr\hskip-2.4pt_y\w\sd^\pm$ sends 
any $\,\eta\in\nr\hskip-2.4pt_y\w\sd^\pm$ to the 
unique $\,H:\ws\to\tw\hskip-1.8pt/\hh\ws\nh$, where 
$\,\ws=\hs\mathrm{Ker}\hskip1.7ptZ(\xi,\xi)$, with a linear lift 
$\,\tilde H:\ws\to\tw\nh=\tyb^\pm$ such that
$\,Z(\xi,\xi)\circ\tilde H\,$ equals the restriction of 
$\,-\nnh2Z(\xi,\eta)\,$ to $\,\ws\nh$. (We have 
$\,d\hskip-.8ptF\hskip-3pt_\xi\w=2Z(\xi,\,\cdot\,)$ since $\,Z(\xi,\eta)\,$ is 
real-bi\-lin\-e\-ar and symmetric in $\,\xi,\eta$, cf.\ (\ref{zex}).) 
Consequently,
\begin{equation}\label{cpl}
2Z(\xi,\eta)\hh w\,=\,-Z(\xi,\xi)\hh\tilde H\nh w\hskip12pt\mathrm{for\ all\ 
}\hskip5ptw\in\mathrm{Ker}\hskip1.7ptZ(\xi,\xi)\hh.
\end{equation}
Com\-plex-lin\-e\-ar\-i\-ty of the differential, due to (\ref{mpg}), means 
that (\ref{cpl}) will still hold if we replace $\,\eta\,$ with $\,J\hn\eta\,$ 
and $\,\tilde H\,$ with $\,J\nh\tilde H$. Then, from (\ref{zex}) and 
(\ref{cpl}), $\,2Z(J\xi,\eta)\hh w=-\nnh2Z(\xi,J\hn\eta)\hh w
=Z(\xi,\xi)\hh J\nh\tilde H\nnh w=J[Z(\xi,\xi)\hh\tilde H\nh w]
=-\nnh2J[Z(\xi,\eta)\hh w]=-\nnh2Z(\xi,\eta)\hh J\nh w$. In other words,
$\,Z(J\xi,\eta)\hh w+Z(\xi,\eta)\hh J\nh w=0\,$ whenever 
$\,w\in\mathrm{Ker}\hskip1.7ptZ(\xi,\xi)$ and 
$\,\eta\in\nr\hskip-2.4pt_y\w\sd^\pm\nh$. Thus, by (\ref{mxw}), 
$\,\mathrm{Ker}\hskip1.7ptZ(\xi,\xi)\subseteq\mg_y^\pm(\hn\bbC\xi)
=\mathrm{Ker}\hskip1.7pt\mu_y^\pm(\xi,\,\cdot\,)$, while the opposite 
inclusion is obvious since (\ref{zex}) gives $\,Z(\xi,J\xi)=0$, and so the 
expression $\,Z(J\xi,\eta)\hh w+Z(\xi,\eta)\hh J\nh w=0\,$ for 
$\,\eta=J\xi\,$ equals $\,Z(\xi,\xi)\hh w$.

The equality $\,\mathrm{Ker}\hskip1.7ptZ(\xi,\xi)=\mg_y^\pm(\hn\bbC\xi)\,$ 
and (\ref{krz}) -- (\ref{mpg}) complete the proof.
\end{proof}
The description of $\,\dot x\nh_\pm\w$ in the lines preceding (\ref{mpg}) also 
gives
\begin{equation}\label{nng}
g_y\w(Z_y^\pm\nh(\xi,\xi)w,w)\,\ge\,0\hskip12pt\mathrm{for\ all}\hskip7pt
\xi\in\nr\hskip-2.4pt_y\w\sd^\pm\hskip7pt\mathrm{and}\hskip7pt
w\in\tyb^\pm\nh,
\end{equation}
which one sees taking the limit of the equality in Theorem~\ref{dcomp}(h2) 
with $\,w\hh'\nh=w\,$ as $\,t\in(t_-\w,t_+\w)\,$ approaches the other endpoint 
$\,t_\mp\w$ (and so $\,\vt\to\tmp$).

\section{The dichotomy theorem\done}\label{dt}
\setcounter{equation}{0}
This section uses the notations listed at the beginning of Section~\ref{cc} 
and the symbols $\,k_\pm\w$ of Remark~\ref{dppdm}. Any $\,y\in\sd^\pm$ leads 
to the assignment
\begin{equation}\label{xtd}
\nr\hskip-2.4pt_y\w\sd^\pm\nnh\smallsetminus\nnh\{0\}\ni\xi\hn
\mapsto\nh d\pi\nh_x^{\hs\pm}\nh(\mathcal{H}_x^\pm\nh)
\in\hn\mathrm{Gr}\nh_k\w\nh(\tyb^\pm\nh)\hh,\hskip3.5pt\mathrm{where}
\hskip5ptx\nh=\hn\varPhi(y,\xi)\hskip5pt\mathrm{and}\hskip5ptk\hn=\nh k_\pm\w\hh,
\end{equation}
$\varPhi=\hs\varPhi^\pm$ being defined by (\ref{phe}). (Due to (\ref{hpm}) and 
(\ref{tme}), $\,d\pi\nh_x^{\hs\pm}$ is injective on $\,\mathcal{H}_x^\pm$.)
\begin{theorem}\label{dicho}{\medit 
Given any compact ge\-o\-des\-ic-gra\-di\-ent K\"ah\-ler triple\/ 
$\,(\mf\nh,g,\vt)$, one and only one of the following two cases occurs.
\begin{enumerate}
  \def\theenumi{{\rm\alph{enumi}}}
\item[{\rm(a)}] Either the mappings\/ {\rm(\ref{xtd})} are all constant, for 
both signs\/ $\,\pm\hs$, or
\item[{\rm(b)}] each of\/ {\rm(\ref{xtd})}, for both signs\/ $\,\pm\hs$, 
descends to a nonconstant hol\-o\-mor\-phic embedding\/ 
$\,\mathrm{P}\nnh\nr\hskip-2.4pt_y\w\nh\to\mathrm{Gr}\nh_k\w\nh(\tyb^\pm\nh)$, 
where\/ $\,\mathrm{P}\nnh\nr\hskip-2.4pt_y\w$ is the projective space of\/ 
$\,\nr\hskip-2.4pt_y\w=\nr\hskip-2.4pt_y\w\sd^\pm\nh$.
\end{enumerate}
Condition\/ {\rm(a)} holds if and only if\/ $\,(\mf\nh,g,\vt)\,$ satisfies\/ 
{\rm(i)} -- {\rm(iii)} in Lemma\/~{\rm\ref{intpr}}.
}
\end{theorem}
\begin{proof}In view of Theorem~\ref{ctrkm}, we may use Lemma~\ref{holem} for 
$\,\mg=\mg_y^\pm$ corresponding to $\,\mu=\mu_y^\pm$ as in (\ref{phi}), 
concluding (from an obvious continuity argument) that, with either fixed sign 
$\,\pm\hs$, all the mappings (\ref{xtd}) descend to hol\-o\-mor\-phic 
embeddings of $\,\mathrm{P}\nnh\nr\hskip-2.4pt_y\w$ unless they are 
all constant. Their constancy for one sign implies, however, the same for the 
other, since it amounts to (ii) or (iii) in Lemma~\ref{intpr}, while (ii) and 
(iii) are equivalent. This completes the proof.
\end{proof}
\begin{remark}\label{infty}Case (a) of Theorem~\ref{dicho} is equivalent to 
(\ref{spn}), as one sees combining Lemma~\ref{intpr}(i) with (\ref{hpm}). 
According to (iv) -- (vi) in Lemma~\ref{intpr}), the immersions of 
Theorem~\ref{tgimm}(c) are then embeddings and their images form the leaves of 
foliations on $\,\sd^\mp\nh$, both of which have the same leaf space $\,\sb$.
\end{remark}
\begin{remark}\label{nncst}When (b) holds in Theorem~\ref{dicho}, images of 
the totally geodesic hol\-o\-mor\-phic immersions of Theorem~\ref{tgimm}(c) 
pass through every point $\,y\in\sd^\pm\nh$, realizing an uncountable family 
of tangent spaces: the image of the embedding (\ref{xtd}).
\end{remark}

\section{More on Grass\-mann\-i\-an triples\done}\label{mg}
\setcounter{equation}{0}
We continue using the asumptions and notation of Section~\ref{he}.
\begin{lemma}\label{lfspc}{\medit 
The leaf space\/ $\,\mf'\hskip-2.8pt/\mathcal{V}\,$ of the in\-te\-gra\-ble 
distribution\/ $\,\mathcal{V}=\mathrm{Span}\hs(\hn v,u)\,$ on\/ 
$\,\mf'\nh=\mf\smallsetminus(\sd^+\nnh\cup\sd^-)$, cf.\ 
Lemma\/~{\rm\ref{dvgww}(a)}, carries a natural structure of a compact complex 
manifold of complex dimension\/ $\,m-1$, with\/ $\,m=\dimc\nh\mf\nh$, such 
that the quo\-tient-space projection\/ 
$\,\mf'\nh\to\mf'\hskip-3pt/\mathcal{V}\,$ forms a hol\-o\-mor\-phic fibration 
and, for either sign\/ $\,\pm\hs$, the projectivization\/ 
$\,\mathrm{P}\nnh\nr\hs$ of the normal bundle\/ 
$\,\nr=\nr\hskip-2.3pt\sd^\pm\nh$, defined as in\/ {\rm(\ref{prz})}, is  
bi\-hol\-o\-mor\-phic to\/ $\,\mf'\hskip-3pt/\mathcal{V}\,$ via the 
bi\-hol\-o\-mor\-phisms sending each complex line\/ $\,\pl\,$ through\/ 
$\,0\,$ in the normal space of\/ $\,\sd^\pm$ at any point to the\/ 
$\,\mathrm{Exp}^\perp\nnh$-im\-age of the punctured radius\/ $\,\delta\,$ 
disk in\/ $\,\pl\hh$, the latter image being a leaf of\/ $\,\mathcal{V}\,$ 
according to Lemma\/~{\rm\ref{vsbkr}(a)}.

The mappings\/ {\rm(\ref{dbp})}, restricted to\/ $\,\mf'\nh$, descend to 
hol\-o\-mor\-phic bundle projections
\begin{equation}\label{hbp}
\pi^\pm\nnh:\mf'\hskip-3pt/\mathcal{V}\to\sd^\pm,
\end{equation}
also denoted by\/ $\,\pi^\pm\nh$, which, under the bi\-hol\-o\-mor\-phic 
identifications\/ 
$\,\mf'\hskip-3pt/\mathcal{V}=\mathrm{P}(\nnh\nr\hskip-2.3pt\sd^\pm\nnh)$ of 
the preceding paragraph, coincide with the bundle projections\/ 
$\,\mathrm{P}(\nnh\nr\hskip-2.3pt\sd^\pm\nnh)\to\sd^\pm\nh$.
}
\end{lemma}
\begin{proof}The restrictions 
$\,\varPhi^\pm\nnh=\varPhi:\nr\hskip-2.3pt\sd^\pm\nnh\smallsetminus\nh\sd^\pm
\nh\to\mf'$ given by (\ref{phe}) with the two possible signs $\,\pm\,$ are 
bi\-hol\-o\-mor\-phisms (Theorem~\ref{first}), and hence so is the composite 
of one of them followed by the inverse of the other. At the same time, by 
Theorem~\ref{first}(iii), either of them descends to a bijection 
$\,\mathrm{P}(\nnh\nr\hskip-2.3pt\sd^\pm\nnh)
\to\mf'\hskip-3pt/\mathcal{V}\nh$, and the composite just mentioned yields a 
bi\-hol\-o\-mor\-phism between 
$\,\mathrm{P}(\nnh\nr\hskip-2.3pt\sd^\pm\nnh)\,$ and 
$\,\mathrm{P}(\nnh\nr\hskip-2.3pt\sd^\mp\nnh)$. This turns 
$\,\mf'\hskip-3pt/\mathcal{V}\,$ into a compact complex manifold in a manner 
independent of the bijection used. Our assertion is now immediate from 
(\ref{pcf}).
\end{proof}
\begin{remark}\label{trint}The direct sum of the two vertical distributions 
$\,\mathrm{Ker}\hskip2.3ptd\pi^\pm$ of the projections (\ref{hbp}) is a 
distribution on $\,\mf'\hskip-3pt/\mathcal{V}\nh$, since, at every point 
$\,\lf\in\mf'\hskip-3pt/\mathcal{V}\nh$, they intersect trivially: 
$\,\mathrm{Ker}\hskip2.3ptd\pi_{\nnh\lf}^+\hs
\cap\,\mathrm{Ker}\hskip2.3ptd\pi_{\nnh\lf}^-=\hs\{0\}$. In fact, as a 
consequence of (\ref{tme}), the original vertical distributions on 
$\,\mf'\nh$, given by (\ref{hpm}), intersect along $\,\mathcal{V}\nh$.
\end{remark}
For a Grass\-mann\-i\-an triple $\,(\mf\nh,g,\vt)\,$ obtained as in 
Section~\ref{eg} from some data (\ref{dta}.i), the descriptions of $\,\sd^\pm$ 
provided by (\ref{spm}.a), and
\begin{equation}\label{mpv}
\begin{array}{l}
\mf'\hskip-3pt/\mathcal{V}\,\,\,
=\,\,\hh\{(\ws,\ws')\in\mathrm{Gr}\nh_k\w\nnh\vs
\times\mathrm{Gr}\nh_{k-\nnh1}\w\hskip-2.3pt\vs:\ws'\nh\subseteq\ws\hs\}\hh,
\mathrm{\ under\ which}\\
\pi^\pm\mathrm{\ in\ (\ref{hbp})\ correspond\ to\ 
}\,(\ws,\ws')\mapsto\ws\,\mathrm{\ and\ }\,(\ws,\ws')\mapsto\ws'\nh.
\end{array}
\end{equation}
the equality meaning a natural bi\-hol\-o\-mor\-phic identification. If 
$\,(\mf\nh,g,\vt)\,$ is in turn a CP triple, arising from (\ref{dta}.ii), 
$\,\sd^\pm$ must be as in (\ref{spm}.b), and (\ref{mpv}) is replaced by 
$\,\mf'\hskip-3pt/\mathcal{V}=\hs\sd^+\nnh\times\sd^-\nh$, while 
$\,\pi^\pm$ in (\ref{hbp}) then become the factor projections.

All these claims are immediate consequences of Remark~\ref{ppmgc}(d).
\begin{lemma}\label{ttlds}{\medit 
For a fi\-\hbox{nite\hh-}\hskip0ptdi\-men\-sion\-al complex vector space\/ 
$\,\vs\nh$, any\/ $\,k\in\{1,\dots,\dim\vs\}$, and\/ 
$\,\mf'\hskip-3pt/\mathcal{V}\,$ given by\/ {\rm(\ref{mpv})}, let\/ 
$\,(\ws_0\w,\ws_0'),(\ws,\ws')\in\mf'\hskip-3pt/\mathcal{V}\nh$. Then there 
exist an integer $\,\q\ge1\,$ and\/ 
$\,(\ws_j\w,\ws_j')\in\mf'\hskip-3pt/\mathcal{V}\nh$, $\,j=0,1,\dots,\q$, 
with\/ $\,(\ws_{\!\q}\w,\ws_{\!\q}')=(\ws,\ws')\,$ and\phantom{\hs} 
\hbox{$(\ws_{\!j-\nh1}\w,\ws_{\!j-1}')\sim(\ws_j\w,\ws_j')\,$} whenever\/ 
$\,j=1,\dots,\q$, the notation\/ $\,(\tilde\ws,\tilde\ws')\sim(\ws,\ws')$ 
meaning that\/ $\,\ws=\tilde\ws\,$ or\/ $\,\ws'=\tilde\ws'\nh$.
}
\end{lemma}
\begin{proof}If $\,\ws_0\w=\ws$, our claim is obvious as 
$\,(\ws_0\w,\ws_0')\sim(\ws,\ws')$. Otherwise we may first choose 
$\,\ws_1\w=\ws_0\w$ and $\,\ws_1'$ such that 
$\,\ws_0\w\cap\ws\subseteq\ws_1'\nh\subseteq\ws_0\w$, and then select 
$\,\ws_2'=\ws_1'$ along with $\,\ws_2\w$ spanned by $\,\ws_1'$ and a vector in 
$\,\ws\smallsetminus\ws_0\w$. Now 
$\,(\ws_0\w,\ws_0')\sim(\ws_1\w,\ws_1')\sim(\ws_2\w,\ws_2')\,$ and 
$\,\dim(\ws_2\w\cap\ws)>\dim(\ws_0\w\cap\ws)$. This step may be repeated for 
$\,(\ws_2\w,\ws_2')\,$ instead of $\,(\ws_0\w,\ws_0')$, as long as 
$\,\ws_2\w\ne\ws$.
\end{proof}
\begin{corollary}\label{strbg}{\medit 
Let\/ $\,(\mf\nh,g,\vt)\,$ be any Grass\-mann\-i\-an triple, arising from some 
data\/ {\rm(\ref{dta}.i)} as in Section\/~{\rm\ref{eg}}. Then the direct 
sum\/ $\,\hs\mathcal{V}\hs\oplus\mathcal{H}^+\hskip-3pt\oplus\mathcal{H}^-$ 
appearing in Lemma\/~{\rm\ref{intpr}(i)} is a strongly 
brack\-et-gen\-er\-at\-ing distribution on\/ $\,\mf'\nh$, in the sense that 
any two points of\/ $\,\mf'$ can be joined by a piecewise\/ $\,C^\infty$ curve 
tangent to\/ 
$\,\mathcal{V}\hs\oplus\mathcal{H}^+\hskip-3pt\oplus\mathcal{H}^-\nh$.
}
\end{corollary}
\begin{proof}According to (\ref{mpv}), whenever 
$\,(\tilde\ws,\tilde\ws')\sim(\ws,\ws')\,$ in Lemma~\ref{ttlds}, both 
$\,(\tilde\ws,\tilde\ws')\,$ and $\,(\ws,\ws')\,$ must lie in the same fibre 
of one of the bundle projections (\ref{hbp}). As the fibres of either 
projection (\ref{hbp}), being complex projective spaces (see the last line in 
Lemma~\ref{lfspc}), are connected, the strong brack\-et-gen\-er\-at\-ing 
property thus follows for the di\-rect-sum distribution of Remark~\ref{trint}. 
Our claim is now immediate since 
$\,\hs\mathcal{V}\hs\oplus\mathcal{H}^+\hskip-3pt\oplus\mathcal{H}^-$ projects 
onto that latter distribution under the quo\-tient-space projection 
$\,\mf'\nh\to\mf'\hskip-3pt/\mathcal{V}\nh$, which also has connected fibres 
(bi\-hol\-o\-mor\-phic to twice-punc\-tur\-ed complex projective lines, cf.\ 
Lemma~\ref{vsbkr}(b)).
\end{proof}

\begin{remark}\label{nintg}A compact ge\-o\-des\-ic-gra\-di\-ent K\"ah\-ler 
triple need not, in general, satisfy conditions (i) -- (iii) of 
Lemma~\ref{intpr}, that is, (\ref{spn}). Examples are provided by 
all Grass\-mann\-i\-an triples $\,(\mf\nh,g,\vt)\,$ arising via Lemma~\ref{chone} 
from data (\ref{dta}.i) such that $\,2\le k\le n-2$, where 
$\,n=\dimc\nnh\vs\nh$. 

Namely, in (\ref{dmq}), $\,q=(k-1)(n-1-k)\,$ as $\,m=(n-k)k\,$ (see 
Remark~\ref{grass}) and, similarly, 
$\,\{d_+\w,\hs d_-\w\}=\{(n-k)(k-1),(n-1-k)k\}\,$ from (\ref{spm}.a) -- 
(\ref{spm}.b), where $\,\dimc\nh\ls=1\,$ by (\ref{dta}.i). Thus, $\,q>0\,$ and 
$\,\hs\mathcal{V}\hs\oplus\mathcal{H}^+\hskip-3pt\oplus\mathcal{H}^-$ in 
(\ref{tme}) is a proper sub\-bun\-dle of $\,T\nnh\mf'\nh$. Consequently, due 
to Corollary~\ref{strbg}, it cannot be in\-te\-gra\-ble.
\end{remark}
\begin{remark}\label{twweb}For any compact ge\-o\-des\-ic-gra\-di\-ent 
K\"ah\-ler triple $\,(\mf\nh,g,\vt)$, the leaf space 
$\,\mf'\hskip-3pt/\mathcal{V}\,$ carries what might be called a {\medit 
hol\-o\-mor\-phic $\,2$-web of complex projective spaces}, formed by the two 
hol\-o\-mor\-phic fibrations (\ref{hbp}) with fibres bi\-hol\-o\-mor\-phic to 
(pos\-i\-\hbox{tive\hh-}\hskip0ptdi\-men\-sion\-al) complex projective spaces, 
having the triv\-i\-al-in\-ter\-sec\-tion property of Remark~\ref{trint}. 
There is also a natural hol\-o\-mor\-phic complex line bundle over 
$\,\mf'\hskip-3pt/\mathcal{V}\nh$, the restriction of which to every fibre of 
$\,\pi^+$ (or, $\,\pi^-$), with (\ref{hbp}), is bi\-hol\-o\-mor\-phic\-al\-ly 
isomorphic to the tautological (or, respectively, dual tautological) bundle of 
the fibre. Specifically, the complex line attached to a leaf 
$\,\lf\hs\subseteq\mf'$ of $\,\mathcal{V}\,$ is 
$\,\{0\}\cup\varPhi^{-\nnh1}\nnh(\lf)\subseteq\nr\hskip-2.4pt_y\w\sd^\pm\nh$, 
cf.\ Theorem~\ref{first}(iii); that changing the sign $\,\pm\,$ to $\,\mp\,$ 
leads to its dual complex line follows from 
\cite[Remark 4.1]{derdzinski-maschler-06} and (\ref{hgv}.a) -- (\ref{hgv}.b).
\end{remark}






\end{document}